\documentclass[11pt,a4paper]{article}
\usepackage[utf8]{inputenc}
\usepackage{microtype}
\usepackage[margin=2.5cm]{geometry}
\usepackage{amsmath,amssymb,amsthm,mathtools}
\usepackage{bm}
\usepackage{booktabs,array,multirow,tabularx}
\usepackage{graphicx}
\usepackage{xcolor}
\usepackage{enumitem}
\usepackage{caption,subcaption}
\usepackage{tikz}
\usetikzlibrary{arrows.meta,positioning,shapes.geometric,calc,fit,backgrounds,decorations.pathreplacing}
\usepackage{pgfplots}
\pgfplotsset{compat=1.18}
\usepackage{algorithm}
\usepackage{algpseudocode}
\usepackage[numbers,sort&compress]{natbib}
\usepackage[colorlinks=true,linkcolor=blue!60!black,citecolor=green!45!black,urlcolor=blue!60!black]{hyperref}
\usepackage[capitalise,noabbrev,nameinlink]{cleveref}

\setlist{itemsep=2pt,topsep=3pt}
\allowdisplaybreaks

\theoremstyle{plain}
\newtheorem{theorem}{Theorem}[section]
\newtheorem{lemma}[theorem]{Lemma}
\newtheorem{corollary}[theorem]{Corollary}

\newtheorem{fact}[theorem]{Fact}

\theoremstyle{definition}
\newtheorem{definition}[theorem]{Definition}
\newtheorem{assumption}[theorem]{Assumption}

\newtheorem{problem}[theorem]{Open problem}
\theoremstyle{remark}
\newtheorem{remark}[theorem]{Remark}
\newtheorem{convention}[theorem]{Convention}
\crefname{assumption}{Assumption}{Assumptions}
\crefname{problem}{Open Problem}{Open Problems}
\crefname{convention}{Convention}{Conventions}

\newcommand{\R}{\mathbb{R}}
\newcommand{\E}{\mathbb{E}}
\newcommand{\Prob}{\mathbb{P}}
\IfFileExists{dsfont.sty}{\usepackage{dsfont}\newcommand{\1}{\mathds{1}}}{\newcommand{\1}{\mathbf{1}}} 
\newcommand{\norm}[1]{\left\|#1\right\|}
\newcommand{\inner}[2]{\left\langle #1,#2\right\rangle}

\DeclareMathOperator*{\argmin}{arg\,min}
\DeclareMathOperator*{\argmax}{arg\,max}

\DeclareMathOperator{\dom}{dom}
\DeclareMathOperator{\tr}{tr}
\DeclareMathOperator{\KL}{KL}
\DeclareMathOperator{\TV}{TV}
\DeclareMathOperator{\Ham}{Ham}

\DeclareMathOperator{\Var}{Var}
\newcommand{\cX}{\mathcal{X}}
\newcommand{\cY}{\mathcal{Y}}
\newcommand{\cZ}{\mathcal{Z}}
\newcommand{\cA}{\mathcal{A}}
\newcommand{\cB}{\mathcal{B}}
\newcommand{\cD}{\mathcal{D}}
\newcommand{\cF}{\mathcal{F}}
\newcommand{\cG}{\mathcal{G}}

\newcommand{\Q}{Q}
\newcommand{\Ball}{\mathbb{B}}
\newcommand{\Sphere}{\mathbb{S}}
\newcommand{\Nseq}{N_{\mathrm{seq}}}
\newcommand{\Nf}{N_{f}}
\newcommand{\Ng}{N_{g}}
\newcommand{\Norc}{N_{\mathrm{orc}}}
\newcommand{\Nmv}{N_{\mathrm{inner}}}
\newcommand{\Cert}{\mathsf{C}}
\newcommand{\viol}{\mathsf{v}}
\newcommand{\LagS}{\Phi_\gamma}

\newcommand{\rowinf}[1]{\norm{#1}_{2\to\infty}}
\newcommand{\pos}[1]{[#1]_+}
\newcommand{\Rmult}{R}
\newcommand{\eps}{\varepsilon}
\newcommand{\kap}{\kappa_m}
\newcommand{\Lam}{\Lambda_m}
\newcommand{\nuj}{\nu_m}
\newcommand{\dA}{D_{A}}
\newcommand{\SA}{\Sigma_{A}}
\newcommand{\Sb}{\Sigma_{b}}
\newcommand{\Xr}{\bar x}
\newcommand{\yr}{\bar y}
\newcommand{\rs}{\varrho}
\newcommand{\Jg}{J_\gamma}
\newcommand{\Vy}{V_{\cY}}
\newcommand{\Vx}{V_{\cX}}

\newcommand{\ie}{i.e.}
\newcommand{\eg}{e.g.}
\newcommand{\ZOS}{\textsc{ZO-Sliding}}
\newcommand{\BSMP}{\textsc{B-SMP}}
\newcommand{\RZOS}{\textsc{R-ZO-Sliding}}
\newcommand{\AZOS}{\textsc{ZO-Sliding-A}}

\newcommand{\SmoothStabilizer}{865.1}
\newcommand{\SmoothInner}{678}
\newcommand{\SmoothCalls}{28,064}
\newcommand{\CuspStabilizer}{1896.0}
\newcommand{\CuspInner}{147}
\newcommand{\CuspCalls}{6,824}
\newcommand{\RestartCalls}{81,231}
\newcommand{\RestartJoint}{0.0009}
\newcommand{\RestartCI}{[0.0009, 0.0010]}
\newcommand{\PlainEightCalls}{28,064}
\newcommand{\PlainEightJoint}{0.0016}
\newcommand{\PlainEightCI}{[0.0013, 0.0020]}
\newcommand{\PlainSixteenCalls}{56,128}
\newcommand{\PlainSixteenJoint}{0.0018}
\newcommand{\PlainSixteenCI}{[0.0016, 0.0020]}
\newcommand{\PlainEqCalls}{80,684}
\newcommand{\PlainEqJoint}{0.0019}
\newcommand{\PlainEqCI}{[0.0017, 0.0020]}
\newcommand{\EqBatch}{23}
\newcommand{\NumSeeds}{30}
\newcommand{\RestartRounds}{70}
\newcommand{\SmoothBsmpOneIters}{4,677}
\newcommand{\SmoothBsmpEightIters}{584}
\newcommand{\SmoothNf}{10,848}
\newcommand{\CuspBsmpOneIters}{1,137}
\newcommand{\CuspBsmpEightIters}{142}
\newcommand{\CuspNf}{2,352}
\newcommand{\SmoothBsmpRoundsToZO}{992}
\newcommand{\CuspZORoundsToBsmpEight}{24}
\newcommand{\CuspZORoundsToBsmpOne}{47}
\newcommand{\SmoothBsmpOneRounds}{9,354}
\newcommand{\SmoothBsmpOneJoint}{0.0806}
\newcommand{\SmoothBsmpEightRounds}{1,168}
\newcommand{\SmoothBsmpEightJoint}{0.0014}
\newcommand{\SmoothZOJoint}{0.0016}
\newcommand{\SmoothZOCI}{[0.0013, 0.0020]}
\newcommand{\SmoothWorstJoint}{0.34}
\newcommand{\SmoothBsmpEightCI}{[0.0013, 0.0015]}
\newcommand{\CuspBsmpOneRounds}{2,274}
\newcommand{\CuspBsmpOneJoint}{0.0363}
\newcommand{\CuspBsmpEightRounds}{284}
\newcommand{\CuspBsmpEightJoint}{0.0856}
\newcommand{\CuspZOJoint}{0.0228}

\newcommand{\CuspWorstJoint}{0.74}

\title{\bfseries Gradient-Free Methods for Stochastic Convex Optimization with Stochastic Functional Constraints}
\author{Vadim D. Abronin$^{2}$ \and Alexander V. Gasnikov$^{1,2,3}$ \and Darina M. Dvinskikh$^{4}$}
\date{%
\small $^1$Innopolis University \quad $^2$MIRAI\\[-1pt]
\small $^3$Kharkevich Institute for Information Transmission Problems RAS \quad $^4$HSE University\\[2pt]
\small \texttt{abronin.vd@phystech.edu}, \texttt{gasnikov@yandex.ru}, \texttt{dmdvinskikh@hse.ru}\\[6pt]
\today}

\begin{document}
\maketitle

\begin{abstract}
We develop accelerated gradient-free methods for stochastic convex optimization with constraints defined by expectations. Our batched primal-dual sliding method uses two-point evaluations sharing a random sample and guarantees expected objective error and expected maximum constraint violation at most $\varepsilon$. It achieves $O(\varepsilon^{-1/2})$ sequential oracle rounds for smooth data and $O(d^{1/4}/\varepsilon)$ for nonsmooth data in dimension $d$, recovering the accuracy and dimension dependence of the corresponding unconstrained accelerated methods. The smooth rate is optimal in accuracy. Each round collects all samples needed for its inner primal-dual updates. Total evaluations retain quadratic dependence on inverse accuracy, with only polylogarithmic dependence on the number of constraints under vector feedback. Complementary lower bounds distinguish the dimension cost of gradient estimation from the unavoidable logarithmic cost of estimating noisy constraint levels. For smooth strongly convex objectives, restarts give logarithmic round complexity and objective evaluation cost linear in inverse accuracy, while constraint-level estimation necessarily remains quadratic. Known affine constraints require no constraint-oracle calls.
\end{abstract}



\tableofcontents

\section{Introduction}\label{sec:intro}

\subsection{The problem}
We consider the stochastic convex optimization problem with $m$ stochastic functional constraints
\begin{equation}\label{eq:P}
 \min_{x\in \Q}\ f(x):=\E_\xi F(x,\xi)\qquad\text{s.t.}\qquad g_i(x):=\E_\xi G_i(x,\xi)\le 0,\quad i=1,\dots,m,
\end{equation}
where $\Q\subset\R^d$ is a convex compact set with Euclidean diameter $D$, $\xi$ has an unknown distribution $P$, and $f,g_1,\dots,g_m$ are convex.
Neither the functions nor their (sub)gradients are available. We only have a \emph{stochastic zeroth-order oracle}: given two points $x^+,x^-$, it draws
a fresh sample $\xi\sim P$ and returns
\begin{equation}\label{eq:oracle-intro}
 \bigl(F(x^+,\xi),G_1(x^+,\xi),\dots,G_m(x^+,\xi)\bigr)\quad\text{and}\quad \bigl(F(x^-,\xi),G_1(x^-,\xi),\dots,G_m(x^-,\xi)\bigr).
\end{equation}
Using the same $\xi$ at both points is the standard \emph{two-point feedback} model \citep{agarwal2010optimal,duchi2015optimal,shamir2017optimal}.
The constraints hold only in expectation, so the oracle never tells us whether a point is feasible. We have to estimate feasibility from noisy values,
and this noise does not cancel the way it does in the differences $F(x^+,\xi)-F(x^-,\xi)$.
Such problems appear when the objective and the constraints come from a simulator, a physical experiment or a non-differentiable model, and the
constraints express risk, budget or safety requirements averaged over uncertainty.

For a given accuracy $\eps>0$ we want a point $\hat x\in\Q$ with
\begin{equation}\label{eq:goal}
\E\bigl[f(\hat x)\bigr]-f^\ast\le \eps\qquad\text{and}\qquad \E\Bigl[\max_{1\le i\le m}\pos{g_i(\hat x)}\Bigr]\le\eps ,
\end{equation}
where $f^\ast$ is the optimal value of \eqref{eq:P}, and we want to know exactly what it costs to get one.

\subsection{Three measures of complexity}\label{sec:intro-measures}
Following \citet{gasnikov2022power}, we distinguish oracle calls from \emph{sequential} rounds of calls.
In a \emph{round}, the query points depend only on what was known before the round, so all of them can be sent at once, in parallel or as one
batched request to a simulator. The number of rounds $\Nseq$ determines wall-clock time on a parallel system. The number of calls $\Norc$ (evaluations of the
vector $(F,G_1,\dots,G_m)$) determines the total sampling cost.

With constraints, it also matters what a call is used for. We therefore count separately the calls used for the objective, $\Nf$, and those used for the
constraint vector, $\Ng$. The two kinds of information may have different prices, and we show that their intrinsic complexities differ too. One call can
serve both purposes, so $\Norc\le\Nf+\Ng$.

Finally, some methods do many cheap steps between two rounds: proximal steps in the low-dimensional dual space, or products of a stored constraint
Jacobian with a vector. We denote their number by $\Nmv$; it plays the role of the matrix--vector count in \citet{zhanglan2022}. \Cref{fig:round} shows
one round of our main method.

\begin{figure}[t]
\centering
\resizebox{\textwidth}{!}{%
\begin{tikzpicture}[>=Latex, font=\small,
 box/.style={draw, rounded corners=2pt, align=center, minimum height=9mm, inner sep=4pt},
 orc/.style={box, fill=blue!8},
 cmp/.style={box, fill=orange!12},
 outp/.style={box, fill=green!10}]
 \node[cmp] (center) {center $u_t$\\ (known before the round)};
 \node[orc, right=11mm of center, minimum width=42mm] (queries) {\textbf{oracle round $t$}\\ $(b_p+b_y)S_t$ paired and $b_yS_t$ single queries at $u_t$,\\ $b_p$ paired queries at $u_{t-1}$ (reserve)\\ \emph{all issued in parallel}};
 \node[cmp, right=11mm of queries, minimum width=40mm] (inner) {\textbf{inner loop}: $S_t$ primal--dual\\ proximal steps without new\\ oracle rounds, each consuming\\ one fresh mini-batch of the round};
 \node[outp, right=11mm of inner] (phase) {phase output\\ $x_t,\ \bar y_t$};
 \draw[->] (center) -- (queries);
 \draw[->] (queries) -- (inner);
 \draw[->] (inner) -- (phase);
 \draw[->] (phase.south) -- ++(0,-9mm) -| node[pos=.25,below] {extrapolation $u_{t+1}=\dfrac{\tau_{t+1}u_t+x_t+\theta_{t+1}(x_t-x_{t-1})}{1+\tau_{t+1}}$} (center.south);
 \begin{scope}[on background layer]
  \node[fit=(queries), draw=blue!40, dashed, rounded corners, inner sep=6pt, label={[blue!60!black]above:{counted in $\Nseq$ (one round) and in $\Norc$, $\Nf$, $\Ng$}}] {};
  \node[fit=(inner), draw=orange!60!black, dashed, rounded corners, inner sep=6pt, label={[orange!60!black]above:{counted in $\Nmv$}}] {};
 \end{scope}
\end{tikzpicture}}
\caption{One phase of the zeroth-order primal--dual sliding method (\Cref{alg:sliding}). All oracle queries of phase $t$ are located at the
current center $u_t$ (plus one reserve mini-batch at $u_{t-1}$) and therefore form a single sequential round. The inner loop performs $S_t$
proximal steps without further oracle rounds; each step uses a fresh mini-batch of the round, which keeps the estimation errors conditionally
centered. The number of rounds is $N=O(\sqrt{LD^2/\eps})$ for smooth data, while the total number of calls is
$O\bigl((dM_R^2D^2+\Rmult^2\sigma_g^2\log(m+1))/\eps^2\bigr)$ up to logarithmic factors (gradient information with the factor $d$, level information with
the factor $\log(m+1)$).}
\label{fig:round}
\end{figure}
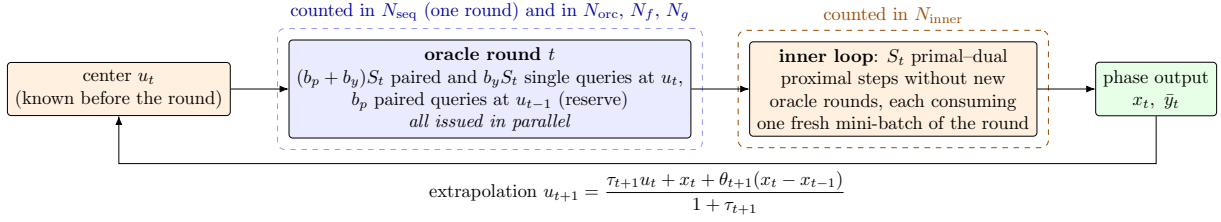

\subsection{Related work}\label{sec:related}

\paragraph{Zeroth-order optimization.} Gradient-free methods go back to \citet{nemirovski1983problem,polyak1990optimal,spall1992multivariate}.
Most modern analyses rely on randomized smoothing. One replaces $f$ by $f_\gamma(x)=\E f(x+\gamma v)$ with $v$ uniform in the unit ball; its gradient
has the single-point representation $\nabla f_\gamma(x)=\frac d\gamma \E[f(x+\gamma u)u]$ with $u$ uniform on the unit sphere
\citep{flaxman2005online,nesterov2017random}. With two evaluations per sample, the estimator
$\frac{d}{2\gamma}[F(x+\gamma u,\xi)-F(x-\gamma u,\xi)]u$ has second moment $O(dM^2)$ for every $\gamma$
\citep{agarwal2010optimal,duchi2015optimal,shamir2017optimal}. The constant $2d$ follows from the Poincar\'e inequality on the sphere
\citep{shamir2017optimal,gasnikov2022power}. This gives $O(dM^2D^2/\eps^2)$ calls for nonsmooth stochastic convex problems, which is optimal
\citep{duchi2015optimal,shamir2017optimal}.

Smoothing also makes the function $\sqrt d M/\gamma$-smooth \citep{gasnikov2022power}, so accelerated methods apply. \citet{gasnikov2022power} used this to
separate the number of sequential iterations from the number of calls. For nonsmooth $M$-Lipschitz objectives on a set of diameter $D$, an accelerated
batched method reaches accuracy $\eps$ in $O(d^{1/4}MD/\eps)$ iterations with $O(dM^2D^2/\eps^2)$ calls in total. For $\mu$-strongly convex objectives
the numbers are $O(d^{1/4}\sqrt{M^2/(\mu\eps)}\log(1/\eps))$ iterations and $O(dM^2/(\mu\eps))$ calls. The same paper also covers bounded adversarial noise and
one-point feedback. For surveys see \citet{gasnikov2022survey,larson2019derivative,conn2009introduction}.
Later work exploits higher-order smoothness \citep{bach2016highly,akhavan2020exploiting,novitskii2021improved,lobanov2024highly}, handles heavy-tailed noise
\citep{kornilov2023accelerated}, combines stochastic and adversarial noise \citep{lobanov2023stochastic}, uses kernel approximations
\citep{lobanov2023kernel}, and studies saddle-point problems \citep{beznosikov2020gradient} and inexact oracles
\citep{bayandina2018gradient,gasnikov2017stochastic}.

\paragraph{Parallel complexity and rounds.} How many rounds an algorithm needs when it may send polynomially many queries per round was studied by
\citet{nemirovski1994parallel}, \citet{balkanski2018parallelization}, \citet{duchi2018minimax}, \citet{bubeck2019complexity} and \citet{diakonikolas2020lower}.
For nonsmooth Lipschitz convex minimization on the Euclidean ball, the smoothing bound $O(d^{1/4}/\eps)$ of \citet{duchi2012randomized} is not known to be tight.
The best lower bounds \citep{bubeck2019complexity,diakonikolas2020lower} are $\tilde\Omega(\min\{\eps^{-2},d^{1/3}\eps^{-2/3}\})$ when many queries per round are
allowed. For $\eps\lesssim d^{-1/4}$, the depth $\tilde O(d^{1/3}\eps^{-2/3})$ is achieved by the ball-acceleration methods of \citet{bubeck2019complexity}, and by
\citet{carmon2023resqueing} with a polynomial number of stochastic gradient queries. Both use (stochastic) gradients. For the two-point value oracle we
consider, $d^{1/4}/\eps$ is still the best known depth.

For smooth convex minimization, accelerated methods with mini-batches achieve the round complexity $\Theta(\sqrt{LD^2/\eps})$
\citep{lan2012optimal,dvurechensky2021accelerated,gorbunov2019optimal}. \citet{zhang2026nearoptimal} recently proved near-optimal lower bounds for randomized
algorithms with an \emph{exact} scalar value oracle. Their model is different from ours: with a common-sample two-point oracle, differences of values are
noise-free but the values themselves are not. \citet{renluo2025parameterfree} obtained parameter-free stochastic zeroth-order methods with near-optimal
oracle complexity for \emph{unconstrained} convex problems; they adapt the step and the smoothing radius, whereas we use fixed schedules.
Our results carry these regimes over to problems with functional constraints.

\paragraph{Functional constraints with first-order information.} Subgradient methods that switch between objective and constraint steps go back to
\citet{polyak1967general}. For nonsmooth convex problems they need $O(M^2D^2/\eps^2)$ subgradient calls \citep{nemirovski1983problem,bayandina2018mirror}.
For stochastic constraints, the cooperative stochastic approximation of \citet{lan2020algorithms} reaches the rate $O(1/\eps^2)$ for both the optimality gap and
the violation, without dual variables. Primal--dual methods work with the Lagrangian saddle-point problem: ConEx \citep{boob2023stochastic} and the methods of
\citet{xu2021first,hamedani2021primal} cover smooth and nonsmooth constraints with stochastic first-order information. Level-set methods
\citep{lin2018level,lin2020data,aravkin2019level} and the level-constrained methods of \citet{boob2025level} avoid dependence on the size of the multipliers.
\citet{deng2026uniformly} give parameter-free first-order methods with optimal oracle complexity for function-constrained problems.

For smooth convex constraints, \citet{zhanglan2022} showed that accelerated constrained gradient descent (ACGD) needs only $O(\sqrt{L/\eps})$ gradient evaluations,
as in the unconstrained case. Here $L=L_f+\norm{y^\ast}L_g$ is the smoothness constant of the Lagrangian; our constant $H=L_0+\Rmult L_g$ comes from this.
In the sliding variant ACGD-S, the $O(1/\eps)$ cost of the bilinear coupling is paid in matrix--vector products with the constraint Jacobian, not in gradient
evaluations. \citet{zhanglan2022} also prove that these complexities are optimal in their first-order model. The lower bound of \citet{ouyang2021lower} for
bilinearly coupled saddle-point problems shows that $O(1/\eps)$ matrix--vector products cannot be improved when the coupling matrix is accessed only through
such products.

\paragraph{Functional constraints with zeroth-order information.} Much less is known here. \citet{nguyen2023stochastic} analyzed a zeroth-order version of
ConEx with Gaussian smoothing and separate gradient estimates for each function. It needs $O((m+1)d/\eps^2)$ calls in the convex case, with the violation
measured in the Euclidean norm of $\pos{g(\hat x)}$. The factor $m+1$ comes from estimating $m+1$ gradients with independent samples, because their oracle
returns one function per call; ours returns all $m+1$ values at once. \citet{usmanova2019safe} studied safe zeroth-order convex learning with unknown
constraints. Other zeroth-order work covers nonconvex constrained problems \citep{balasubramanian2022zeroth,li2022zeroth}, distributed gradient-free mirror
descent with constraints \citep{yu2021distributed}, block methods for Lipschitz expectation-valued problems \citep{shanbhag2021zeroth}, and saddle-point
problems \citep{beznosikov2020gradient,dvinskikh2022noisy}. As far as we know, none of these papers separates rounds from calls, distinguishes objective
from constraint information, or gives accelerated round complexities for constrained zeroth-order problems.

\paragraph{Saddle-point methods.} Mirror--Prox \citep{nemirovski2004prox} and its stochastic version \citep{juditsky2011solving,nemirovski2009robust} solve
convex--concave saddle-point problems in $O(L/\eps)$ iterations plus $O(\sigma^2/\eps^2)$ samples. The primal--dual methods of \citet{chambolle2011first} with dual
extrapolation are the inner engine of the sliding methods of \citet{lan2016gradient,zhanglan2022}.

\subsection{Contributions}\label{sec:contrib}
All results in the main text are proved in full, with explicit constants.
\begin{enumerate}[label=\textbf{(C\arabic*)},leftmargin=*]
\item \textbf{Reduction to a smooth saddle-point problem (\Cref{sec:smoothing,sec:lagrangian}).} We smooth the objective \emph{and} the constraints and
 shift the constraint levels explicitly. The resulting feasible set contains the original one, and the Lagrangian is smooth in $x$. The certificate
 $\Cert_{\Rmult,\gamma}(x)=f_\gamma(x)-f_\gamma^\ast+\Rmult\|\pos{h(x)}\|_\infty$ bounds both the optimality gap and the violation for the original problem
 (\Cref{lem:transfer,lem:certificate}). Two-point estimators of the Lagrangian gradient have second moment $2d(M_0+\Rmult M_g)^2$ for any smoothing radius, and
 the constraint levels are estimated from single calls. We also prove a new bound on the $\ell_\infty$-norm of mini-batches of Jacobian samples applied to a
 fixed vector (\Cref{lem:action}). It removes the factor $d$ from the dual noise at the cost of a logarithm in $m$.
\item \textbf{Batched stochastic Mirror--Prox (\Cref{sec:mp}).} As a baseline we prove (\Cref{thm:mp})
 $\E\Cert\le \sqrt3 L_\alpha\cD_\alpha^2/N+5\sigma_\ast\cD_\alpha/\sqrt N$, which gives $\Nseq=O(\eps^{-1})$ rounds for smooth data with $O(\eps^{-2})$ calls.
\item \textbf{Zeroth-order primal--dual sliding (\Cref{sec:sliding}).} Our main method (\Cref{alg:sliding}) uses the outer recursion of ACGD-S
 \citep{zhanglan2022}. Its inner proximal steps use fresh mini-batches drawn at the same center. We prove (\Cref{thm:sliding})
 \[
   \E\,\Cert_{\Rmult,\gamma}(\Xr_N)\;\le\;\frac{2L\dA^2}{N(N+1)}+4\dA\sqrt{\frac{17dM_R^2}{b_pK_N}+\frac{\rs^2\sigma_y^2}{b_yK_N}},
 \]
 where $N$ is the number of rounds, $K_N=\sum_t S_t$ the number of inner steps, and $b_p$, $b_y$ the primal and dual mini-batch sizes. As a result,
 $\Nseq=O(\sqrt{(L_0+\Rmult L_g)D^2/\eps})$ for smooth data and $\Nseq=O(d^{1/4}\sqrt{M_R M}D/\eps)$ for nonsmooth data. The method uses
 $\Nf=O(\Nmv+dM_R^2D^2/\eps^2)$ calls for the objective and $\Ng=\Norc=O\bigl(\Nmv+[dM_R^2D^2+\Lam\Rmult^2(\sigma_g^2+\nuj M_g^2D^2)]/\eps^2\bigr)$ calls for the
 constraints, with $\Lam=O((1+\log m)\log(m+1))$ and $\nuj=O(\min\{d,1+\log m\})$ (\Cref{cor:sliding-smooth,cor:sliding-nonsmooth}). Up to the multiplier
 scale, these round complexities are the best known ones for the unconstrained two-point problem. The analysis rests on two inequalities that may be
 useful elsewhere: a \emph{sliding inequality} for arbitrary sequences (\Cref{lem:outer}) and an \emph{energy inequality} for the inner primal--dual loop that
 tracks the stochastic errors explicitly (\Cref{lem:inner}).
\item \textbf{Strongly convex objectives (\Cref{sec:restarts}).} Restarting on the certificate gives $\Nseq=O(\sqrt{L/\mu}\log(\mu D^2/\eps))$ rounds.
 Because the primal and dual batch sizes are separate, the objective is used in only $\Nf=O(\Nmv+dM_R^2/(\mu\eps))$ calls. The constraint calls carry an extra
 $\tilde O(\Rmult^2\sigma_g^2\log(m+1)/\eps^2)$ term for the level information (\Cref{thm:restart-sliding,thm:restart-mp}), and this term cannot be removed
 even for strongly convex objectives (\Cref{thm:lower-levels} and \Cref{rem:lower-levels-sc}).
\item \textbf{Exactly known affine constraints (\Cref{sec:affine}).} If $g(x)=Ax-c$ is known, one mini-batch per round is enough and the coupling costs only
 matrix--vector products: $\Nf=O(\sqrt{L_{0,\gamma}D^2/\eps}+dM_0^2D^2/\eps^2)$, $\Ng=0$ and
 $\Nmv=O(\sqrt{L_{0,\gamma}D^2/\eps}+\rowinf{A}\Rmult D\sqrt{\log(m+1)}/\eps)$ (\Cref{thm:affine}). This is the zeroth-order counterpart of ACGD-S.
\item \textbf{Lower bounds (\Cref{sec:lower}).} Any algorithm with $T$ two-point calls has error $\Omega(MD\sqrt{d/T})$ on a class satisfying our assumptions
 (\Cref{thm:lower-objective}). A scaled epigraph reduction gives $\Omega(M_gD\sqrt{d/T})$ for the constraint calls, with universal constants
 (\Cref{thm:lower-constraint}). Detecting the active constraint takes $\Omega(\sigma_g^2\log(1+m)/\eps^2)$ vector calls and $\Omega(m\sigma_g^2/\eps^2)$ scalar calls
 (\Cref{thm:lower-levels,thm:lower-scalar}); this is where the $\log(m+1)$ factors of the entropy dual geometry come from. The objective term is tight up to
 constants. The constraint-gradient term is tight in $d$, $M_g$, $D$ and $\eps$ but not in $\Rmult^2$, and the level term is tight up to $\kap\Rmult^2$.
 For the nonsmooth depth we only match the best known bound (\Cref{tab:lower}).
\item \textbf{Experiments (\Cref{sec:experiments}).} On a smooth and a nonsmooth instance with $d=24$ and $m=8$, using the schedule of \Cref{thm:sliding}, we see
 the predicted gap between rounds and calls. We also test the restart scheme at equal budgets and illustrate two information-theoretic effects: the
 common-sample variance and the logarithmic dependence on $m$.
\end{enumerate}
Open questions are collected in \Cref{sec:future}. \Cref{tab:overview} summarizes the rates, and \Cref{tab:comparison} compares them with the literature.

\begin{table}[t]
\centering\footnotesize\setlength{\tabcolsep}{3pt}
\caption{Overview of the upper bounds proved in this paper (constants and $\log$ factors in $m$ omitted; see \Cref{sec:summary} for the exact
statements). Here $M_R=M_0+\Rmult M_g$, $H=L_0+\Rmult L_g$, $\sigma_g^2$ is the variance proxy of the constraint values, $\Lam=(1+\log 2m)\log(m+1)$,
$\nuj=\min\{4d,92\kap\}$, $\Nseq$ counts sequential oracle rounds, $\Nf$ and $\Ng$ count the vector calls (point evaluations of the whole vector $(F,G)$)
in which the objective, resp.\ the constraint component is used, and $\Nmv$ counts inner steps that require no new round. In all methods every call carries the
constraint vector, so $\Norc=\Ng$ (and $\Ng=0$, $\Norc=\Nf$ for known affine constraints). The terms of the call counts of the order of $\Nmv$ are omitted.}
\label{tab:overview}
\resizebox{\textwidth}{!}{%
\begin{tabular}{@{}lllll@{}}
\toprule
Regime & $\Nseq$ & $\Nf$ & $\Ng=\Norc$ & $\Nmv$ \\
\midrule
\multicolumn{5}{@{}l}{\emph{Mirror--Prox baseline (\Cref{thm:mp})}}\\
smooth $f$, smooth $g$ & $\dfrac{HD^2+M_g\Rmult D\sqrt{\log m}}{\eps}$ & $\dfrac{dM_R^2D^2+\Lam\Rmult^2\sigma_g^2}{\eps^2}$ & $=\Nf$ & $=\Nseq$\\[6pt]
nonsmooth & $\dfrac{\sqrt d M_RMD^2}{\eps^2}$ & $\dfrac{dM_R^2D^2+\Lam\Rmult^2\sigma_g^2}{\eps^2}$ & $=\Nf$ & $=\Nseq$\\
\midrule
\multicolumn{5}{@{}l}{\emph{Zeroth-order sliding (\Cref{thm:sliding})}}\\
smooth $f$, smooth $g$ & $\sqrt{\dfrac{HD^2}{\eps}}$ & $\dfrac{dM_R^2D^2}{\eps^2}$ & $\dfrac{dM_R^2D^2+\Lam\Rmult^2(\sigma_g^2+\nuj M_g^2D^2)}{\eps^2}$ & $\Nseq+\dfrac{M_g\Rmult D\sqrt{\log m}}{\eps}$\\[6pt]
nonsmooth & $\dfrac{d^{1/4}\sqrt{M_RM}D}{\eps}$ & $\dfrac{dM_R^2D^2}{\eps^2}$ & $\dfrac{dM_R^2D^2+\Lam\Rmult^2(\sigma_g^2+\nuj M_g^2D^2)}{\eps^2}$ & $\Nseq+\dfrac{M_g\Rmult D\sqrt{\log m}}{\eps}$\\[6pt]
$\mu$-s.c., smooth & $\sqrt{\dfrac{H}{\mu}}\log\dfrac{\mu D^2}{\eps}$ & $\dfrac{dM_R^2}{\mu\eps}$ & $\dfrac{dM_R^2}{\mu\eps}+\dfrac{\Lam\Rmult^2(\sigma_g^2+\nuj M_g^2D^2)}{\eps^2}$ & $\Nseq+M_g\Rmult\sqrt{\dfrac{\log m}{\mu\eps}}$\\[6pt]
$\mu$-s.c., nonsmooth & $d^{1/4}\sqrt{\dfrac{M_RM}{\mu\eps}}\log\dfrac{\mu D^2}{\eps}$ & $\dfrac{dM_R^2}{\mu\eps}$ & $\dfrac{dM_R^2}{\mu\eps}+\dfrac{\Lam\Rmult^2(\sigma_g^2+\nuj M_g^2D^2)}{\eps^2}$ & $\Nseq+M_g\Rmult\sqrt{\dfrac{\log m}{\mu\eps}}$\\
\midrule
\multicolumn{5}{@{}l}{\emph{Exactly known affine constraints (\Cref{thm:affine}), $\Ng=0$}}\\
smooth $f$ & $\sqrt{\dfrac{L_0D^2}{\eps}}$ & $\dfrac{dM_0^2D^2}{\eps^2}$ & $0$ & $\Nseq+\dfrac{\rowinf{A}\Rmult D\sqrt{\log m}}{\eps}$\\[6pt]
nonsmooth $f$ & $\dfrac{d^{1/4}M_0D}{\eps}$ & $\dfrac{dM_0^2D^2}{\eps^2}$ & $0$ & $\Nseq+\dfrac{\rowinf{A}\Rmult D\sqrt{\log m}}{\eps}$\\
\bottomrule
\end{tabular}}
\end{table}

\begin{table}[t]
\centering\footnotesize\setlength{\tabcolsep}{4pt}
\caption{Comparison with related results (constants omitted). ``FO'' means stochastic first-order information, ``ZO'' zeroth-order information,
``pairs'' two-point calls, $d$ is the dimension. Complexities are for the convex nonsmooth case unless stated otherwise;
$\Nseq$ is listed only for papers that account for it. The information models differ: in \citet{nguyen2023stochastic} a call returns one function,
the $m+1$ gradients are estimated from independent samples and the violation is measured in $\ell_2$; in this paper one call returns all $m+1$ values
with a common sample and the violation is measured in $\ell_\infty$, so part of the factor $m+1$ reflects the oracle model rather than the algorithm.}
\label{tab:comparison}
\begin{tabular}{@{}p{3.0cm}p{2.4cm}>{\raggedright\arraybackslash}p{4.9cm}>{\raggedright\arraybackslash}p{5.0cm}@{}}
\toprule
Method & Setting & Calls & Rounds / remarks\\
\midrule
CSA \citep{lan2020algorithms} & FO, stochastic $g$ & $1/\eps^2$ & no dual variables\\
ConEx \citep{boob2023stochastic} & FO, stochastic $g$ & $1/\eps^2$ & $\sqrt{1/\eps}$ gradients for smooth deterministic data\\
ACGD-S \citep{zhanglan2022} & deterministic FO, smooth $f,g$ & $\sqrt{L/\eps}$ gradients & $\rowinf{\nabla g}\Rmult D/\eps$ matrix--vector products\\
SZO-ConEx \citep{nguyen2023stochastic} & ZO, stochastic $g$ & $(m+1)d/\eps^2$ & violation in $\ell_2$; independent estimates per function\\
\citet{gasnikov2022power} & ZO pairs, unconstrained & $dM^2D^2/\eps^2$ & $\Nseq=d^{1/4}MD/\eps$\\
\midrule
\Cref{thm:mp} & ZO pairs, stochastic $g$ & $(dM_R^2D^2+\Lam\Rmult^2\sigma_g^2)/\eps^2$ & $\Nseq=\sqrt dM_RMD^2/\eps^2$ (nonsmooth), $HD^2/\eps$ (smooth)\\
\Cref{thm:sliding} & ZO pairs, stochastic $g$ & $\Nf=dM_R^2D^2/\eps^2$; $\Ng=dM_R^2D^2/\eps^2$\newline $+\,\Lam\Rmult^2(\sigma_g^2+\nuj M_g^2D^2)/\eps^2$ & $\Nseq=d^{1/4}\sqrt{M_RM}D/\eps$ (nonsmooth), $\sqrt{HD^2/\eps}$ (smooth)\\
\Cref{thm:affine} & ZO pairs, known affine $g$ & $\Nf=dM_0^2D^2/\eps^2$, $\Ng=0$ & $\Nmv=\rowinf A\Rmult D\sqrt{\log m}/\eps$\\
\Cref{thm:lower-objective,thm:lower-constraint} & ZO pairs & $\Omega(dM_0^2D^2/\eps^2)$ for $f$, $\Omega(dM_g^2D^2/\eps^2)$ for $g$ & lower bounds (unit multiplier)\\
\Cref{thm:lower-levels,thm:lower-scalar} & ZO, constraint levels & $\Omega(\sigma_g^2\log(1+m)/\eps^2)$ vector calls, $\Omega(m\sigma_g^2/\eps^2)$ scalar calls & lower bounds\\
\bottomrule
\end{tabular}
\end{table}

\subsection{Organization}
\Cref{sec:setting} introduces notation, assumptions, the oracle model and the complexity measures. \Cref{sec:smoothing} covers randomized smoothing,
the two-point estimators, the $\ell_\infty$ mini-batch lemma and the martingale-supremum lemma. \Cref{sec:lagrangian} defines the smoothed shifted Lagrangian,
bounds the multipliers and introduces the certificate. \Cref{sec:mp} analyzes the Mirror--Prox baseline. \Cref{sec:sliding} is the core of the paper: the
sliding method, its deterministic skeleton and the stochastic theorem. \Cref{sec:restarts} treats strongly convex objectives, \Cref{sec:affine} known affine
constraints and \Cref{sec:lower} lower bounds. \Cref{sec:summary} collects the rates, \Cref{sec:experiments} reports experiments, and \Cref{sec:conclusion} lists
open problems. Long proofs are in the appendices.

\section{Setting, oracle model and complexity measures}\label{sec:setting}

\subsection{Notation}
$\norm{\cdot}$ denotes the Euclidean norm on $\R^d$ and $\inner{\cdot}{\cdot}$ the Euclidean inner product; $\norm{\cdot}_1$ and
$\norm{\cdot}_\infty$ are the $\ell_1$- and $\ell_\infty$-norms on $\R^m$. For a matrix $J\in\R^{m\times d}$ with rows $J_1,\dots,J_m$ we write
$\rowinf{J}:=\max_i\norm{J_i}$; then
\begin{equation}\label{eq:rowinf}
\norm{Jw}_\infty\le \rowinf{J}\norm{w}\quad(w\in\R^d),\qquad \norm{J^\top y}\le \rowinf{J}\norm{y}_1\quad(y\in\R^m).
\end{equation}
$\Ball=\{v\in\R^d:\norm v\le 1\}$ is the unit ball, $\Sphere=\{u:\norm u=1\}$ the unit sphere, $U(\Ball)$ and $U(\Sphere)$ the uniform distributions
on them. $\pos{z}=\max\{z,0\}$ componentwise, and $\viol(x):=\norm{\pos{g(x)}}_\infty=\max_i\pos{g_i(x)}$ is the maximal constraint violation.
$\Q\subset\R^d$ is convex and compact, $D:=\max_{x,x'\in\Q}\norm{x-x'}$ its diameter and $\Q_{\bar\gamma}:=\Q+\bar\gamma\Ball$ its
$\bar\gamma$-neighbourhood, where $\bar\gamma>0$ is a fixed upper bound for the smoothing radii used below. Throughout,
\begin{equation}\label{eq:kappa}
 \kap:=1+\log(2m),\qquad \Lam:=\kap\log(m+1),
\end{equation}
and $\log$ is the natural logarithm. \Cref{tab:notation} lists the recurring symbols.

\begin{table}[t]
\centering\footnotesize
\caption{Recurring notation.}\label{tab:notation}
\begin{tabularx}{\textwidth}{@{}lXl@{}}
\toprule
Symbol & Meaning & Defined in\\
\midrule
$d,\ m$ & dimension, number of constraints & \eqref{eq:P}\\
$\Q,\ D,\ \Q_{\bar\gamma}$ & feasible set, its diameter, its $\bar\gamma$-neighbourhood & \Cref{sec:setting}\\
$M_0,M_g,\ M:=\max\{M_0,M_g\}$ & Lipschitz moduli (mean square) of the samples & \Cref{ass:lip}\\
$L_0,\ L_g$ & Lipschitz constants of $\nabla f$, $\nabla g_i$ (when smooth) & \Cref{ass:smooth}\\
$\sigma_g^2,\ \sigma_h^2$ & variance proxies of the constraint values & \Cref{ass:levels}, \eqref{eq:sigmah}\\
$\rho,\ x^{\mathrm s}$ & Slater margin and Slater point & \Cref{ass:slater}\\
$\mu$ & strong convexity parameter of $f$ & \Cref{ass:sc}\\
$\gamma,\ f_\gamma,\ g_{i,\gamma}$ & smoothing radius and smoothed functions & \eqref{eq:smoothing}\\
$L_{0,\gamma},\ L_{g,\gamma},\ H,\ L$ & smoothness constants after smoothing & \Cref{lem:smoothing}, \eqref{eq:HL}\\
$b_0,\ b_i,\ h_i=g_{i,\gamma}-b_i$ & bias bounds and shifted smoothed constraints & \eqref{eq:shift}\\
$\Rmult,\ Y_\Rmult,\ \tilde Y_\Rmult,\ M_R$ & multiplier bound, dual sets, $M_R:=M_0+\Rmult M_g$ & \eqref{eq:R}, \eqref{eq:Ysets}\\
$\LagS,\ \Cert_{\Rmult,\gamma}$ & smoothed Lagrangian and certificate & \eqref{eq:Lag}, \eqref{eq:cert}\\
$x^\ast_\gamma,\ f^\ast_\gamma,\ y^\ast_\gamma$ & solution, value and multipliers of the smoothed problem & \Cref{lem:multipliers}\\
$\cA,\ \cB$ & primal and dual prox radii: $\cA=\tfrac12\norm{x_0-x^\ast_\gamma}^2$ (unknown; replaced by $D^2/2$ in parameter choices), $\cB=\Rmult^2\log(m+1)$ & \eqref{eq:radii}\\
$\Vx,\ \Vy$ & Euclidean and entropic Bregman distances & \eqref{eq:bregman}\\
$\hat g,\ \hat J,\ \hat h$ & two-point gradient, Jacobian and one-point level estimators & \eqref{eq:estimators}\\
$\nuj$ & constant in the Jacobian-action bound: $\nuj=\min\{4d,92\kap\}$ & \Cref{lem:action}\\
$\Norc,\ \Nf,\ \Ng,\ \Nseq,\ \Nmv$ & vector calls, calls used for $f$ and for $g$, rounds, inner steps & \Cref{def:measures}\\
$b,\ \lambda,\ \rs$ & mini-batch size, stabilizer, dual scale & \Cref{alg:sliding}\\
$S_t,\ a_t,\ \eta_t,\ \beta_t,\ K_N,\ A_N,\ W_N$ & inner schedule of the sliding method & \eqref{eq:schedule}\\
$\dA,\ \Sb,\ \sigma_y$ & radius and noise constants of the sliding method & \eqref{eq:DA}, \eqref{eq:SigmaA}\\
$b_p,\ b_y$ & sizes of the primal and dual mini-batches of \Cref{alg:sliding} & \Cref{alg:sliding}\\
\bottomrule
\end{tabularx}
\end{table}

\subsection{Assumptions}
The functions $F(\cdot,\xi)$ and $G_i(\cdot,\xi)$ are defined on $\R^d$ for every $\xi$, and all functions of $x$ below are assumed measurable in
$\xi$ and integrable, so that the expectations exist. We say that a differentiable function $\varphi$ is $L$-smooth if $\nabla\varphi$ is
$L$-Lipschitz, and $\mu$-strongly convex if $\varphi-\frac\mu2\norm{\cdot}^2$ is convex.

\begin{assumption}[convexity]\label{ass:convex}
$f,g_1,\dots,g_m:\R^d\to\R$ are convex.
\end{assumption}

\begin{assumption}[Lipschitz samples]\label{ass:lip}
For every $\xi$ the function $F(\cdot,\xi)$ is $M_0(\xi)$-Lipschitz and $G_i(\cdot,\xi)$ is $M_i(\xi)$-Lipschitz on the neighbourhood $\Q_{\bar\gamma}$ of $\Q$,
$i=1,\dots,m$, where $M_g(\xi):=\max_iM_i(\xi)$ and
\[
\E\,M_0(\xi)^2\le M_0^2,\qquad \E\,M_g(\xi)^2\le M_g^2 ,
\]
with constants $M_0,M_g>0$. We put $M:=\max\{M_0,M_g\}$.
\end{assumption}

By Jensen's inequality, \Cref{ass:lip} implies that $f$ is $M_0$-Lipschitz and every $g_i$ is $M_g$-Lipschitz on $\Q_{\bar\gamma}$.
Note that no bound on the absolute values $F(x,\xi)$ or on their variance is required: the methods below never use the value of the objective
at a single point, only differences at pairs of points evaluated with the same $\xi$. The Lipschitz property is required only on $\Q_{\bar\gamma}$, which
contains every point the algorithms query; \Cref{lem:smoothing}(e) below extends the expected functions from $\Q_{\bar\gamma}$ to $\R^d$ whenever a global
property is convenient.

\begin{assumption}[noise of the constraint values]\label{ass:levels}
$\E\norm{G(x,\xi)-g(x)}_\infty^2\le\sigma_g^2$ for all $x\in\Q_{\bar\gamma}$, where $G=(G_1,\dots,G_m)$.
\end{assumption}

\begin{assumption}[smoothness, used when stated]\label{ass:smooth}
$f$ is $L_0$-smooth and each $g_i$ is $L_g$-smooth on $\R^d$.
\end{assumption}

\begin{assumption}[Slater condition]\label{ass:slater}
There exist $x^{\mathrm s}\in\Q$ and $\rho>0$ with $g_i(x^{\mathrm s})\le-\rho$ for all $i$.
\end{assumption}

\begin{assumption}[strong convexity, used when stated]\label{ass:sc}
$f$ is $\mu$-strongly convex on $\Q_{\bar\gamma}$, \ie\ $f-\frac\mu2\norm\cdot^2$ is convex on $\Q_{\bar\gamma}$.
\end{assumption}

\begin{remark}[on the standing assumptions]\label{rem:assumptions}
(i) \emph{Local versus global properties.} The Lipschitz property (\Cref{ass:lip}) and the strong convexity (\Cref{ass:sc}) are required only on the
neighbourhood $\Q_{\bar\gamma}$ of the compact set $\Q$. This is not a mere convenience: a function cannot be both $M$-Lipschitz and $\mu$-strongly convex on
all of $\R^d$ (for a unit vector $v$, strong convexity gives $\varphi(x+tv)+\varphi(x-tv)-2\varphi(x)\ge\mu t^2$, the Lipschitz property gives $\le2Mt$, and the
two are incompatible for $t>2M/\mu$), so global versions of \Cref{ass:lip,ass:sc} would define an empty class. On $\Q_{\bar\gamma}$ the two assumptions are
compatible (a strongly convex quadratic is Lipschitz on every bounded set), and $\Q_{\bar\gamma}$ is all the algorithms ever see: the Lipschitz property of
the \emph{samples} enters only through the moment bounds of \Cref{lem:twopoint,lem:action,lem:levels} and the bias bounds of \Cref{lem:smoothing}(a), all
evaluated at query points in $\Q_{\bar\gamma}$; the Lipschitz property of the \emph{expected} functions is used on $\Q$ in \Cref{lem:multipliers}; and strong
convexity is used on $\Q$ in \Cref{lem:certificate}. The only place where a global property is needed is the sliding inequality (\Cref{lem:outer}), which requires
the smoothed Lagrangian $\LagS(\cdot,y)$ to be convex and $L$-smooth on $\R^d$. In the nonsmooth regime this is obtained by the extension device of
\Cref{lem:smoothing}(e): a convex function that is $M$-Lipschitz on $\Q_{\bar\gamma}$ has a convex $M$-Lipschitz extension to $\R^d$ that coincides with it on
$\Q_{\bar\gamma}$, hence has the same smoothing on $\Q$ and produces the same oracle answers; the extension is used only for the smoothness of the Lagrangian,
never for its strong convexity (\Cref{conv:extension}). In the smooth regime \Cref{ass:smooth} is stated on $\R^d$; it is the only global requirement of
the paper, it is compatible with \Cref{ass:lip,ass:sc} (quadratics satisfy all three), and it may be read as the assumption that the data admit convex
$L_0$-, $L_g$-smooth extensions, since a function that is convex and $L$-smooth only on a neighbourhood of $\Q$ need not have such an extension with the same
constant. The test instances of \Cref{sec:experiments} satisfy \Cref{ass:lip,ass:smooth,ass:sc} in exactly this form (\Cref{sec:exp-setup}).
(ii) \Cref{ass:levels} is a genuine additional assumption: \Cref{ass:lip} bounds the differences $G_i(x,\xi)-G_i(x',\xi)$ but says nothing about
$G_i(x,\xi)-g_i(x)$, which may have infinite variance (a random offset independent of $x$). The constraint levels enter the dual updates and their
noise does not cancel.
(iii) \Cref{ass:slater} is used only to bound the Lagrange multipliers; any known upper bound on $\norm{y^\ast_\gamma}_1$ can replace it
(see \eqref{eq:R}).
(iv) The constants $M_0,M_g$ are taken positive and, in the smooth regime, $\max\{L_0,L_g\}>0$ (if $L_0=L_g=0$ all data are affine and \Cref{sec:affine}
applies); this only excludes degenerate cases in which the corresponding terms of the bounds are absent, and avoids divisions by zero in the parameter rules.
\end{remark}

\subsection{Oracle model}\label{sec:oracle}
The unit of oracle cost is a \emph{vector call}: one evaluation at a point $x\in\Q_{\bar\gamma}$ that returns $(F(x,\xi),G(x,\xi))\in\R^{m+1}$ for one
realization $\xi$. A \emph{paired query}
(or two-point call) at $(x^+,x^-)$ consists of \emph{two} vector calls with one common fresh $\xi\sim P$ independent of everything else, as in
\eqref{eq:oracle-intro}. A \emph{single query} is one vector call with a fresh $\xi$. A method may use only the objective part or only the
constraint part of a call. These two kinds of information can have different prices, and we show that their complexities differ. So besides the total
number of calls we count the calls in which each part is used.

\begin{definition}[complexity measures]\label{def:measures}
For a run of an algorithm we denote by
\begin{itemize}
\item $\Norc$ the total number of vector calls (point evaluations) issued by the algorithm; $\Nf$ the number of those calls whose objective
 component $F(x,\xi)$ is used, and $\Ng$ the number of those calls whose constraint component $G(x,\xi)$ is used (one call returns all $m$
 components). A call may count in both, so $\max\{\Nf,\Ng\}\le\Norc\le\Nf+\Ng$. We compare methods by $\Norc$;
\item $\Nseq$ the number of \emph{sequential oracle rounds}: the run is split into rounds, and the locations of all queries of a round are
 measurable with respect to the information (oracle answers and internal randomness) available before the round;
\item $\Nmv$ the number of inner operations that require no new oracle round: proximal steps on $\Q$ and on the dual set, or matrix--vector
 products with a constraint Jacobian. In \Cref{alg:sliding} these steps consume oracle samples that were drawn in the current round (see
 \Cref{sec:sliding-discussion}); in \Cref{alg:affine} they are oracle-free.
\end{itemize}
\end{definition}

The lower bounds of \Cref{sec:lower} count either paired queries or single vector calls, as stated in each theorem ($T$ paired queries are $2T$ vector
calls). In the \emph{scalar} oracle model of \Cref{thm:lower-scalar}, a call returns one component $G_i(x,\xi)$ chosen by the algorithm. In the \emph{vector}
model, used by all our algorithms, all $m$ components come at once.

\subsection{Target accuracy}
For $\eps>0$ a random point $\hat x\in\Q$ is an \emph{$\eps$-solution in expectation} if \eqref{eq:goal} holds. All upper bounds use this criterion.
The lower bounds of \Cref{sec:lower} are for $\E[f(\hat x)-f^\ast]+\E\viol(\hat x)$, the expected \emph{sum} of the gap and the violation (or the gap alone when
the constraints are inactive). We use the sum and not the maximum because $f(\hat x)-f^\ast$ can be negative at infeasible points. Since \eqref{eq:goal}
implies that the sum is at most $2\eps$, upper and lower bounds are comparable up to a factor $2$ in $\eps$. Bounds in probability follow from independent
repetitions in the standard way; see \Cref{sec:future}.

\section{Randomized smoothing and two-point estimators}\label{sec:smoothing}

\subsection{Smoothing by averaging over a ball}
For a function $\varphi:\R^d\to\R$ and a radius $\gamma>0$ define
\begin{equation}\label{eq:smoothing}
 \varphi_\gamma(x):=\E_{v\sim U(\Ball)}\varphi(x+\gamma v),\qquad x\in\R^d .
\end{equation}
We apply \eqref{eq:smoothing} to $f$ and to every $g_i$ and write $f_\gamma$, $g_{i,\gamma}$, $g_\gamma=(g_{1,\gamma},\dots,g_{m,\gamma})$ and
$\Jg(x)\in\R^{m\times d}$ for the matrix with rows $\nabla g_{i,\gamma}(x)^\top$.

\begin{lemma}[properties of the smoothing]\label{lem:smoothing}
Let $\varphi:\R^d\to\R$ be convex and $\gamma>0$; for a convex set $C\subseteq\R^d$ write $C_\gamma:=C+\gamma\Ball$.
\begin{enumerate}[label=(\alph*)]
\item $\varphi_\gamma$ is convex. If $\varphi$ is $M$-Lipschitz on $C_\gamma$, then $\varphi_\gamma$ is $M$-Lipschitz on $C$ and
 $\varphi(x)\le\varphi_\gamma(x)\le\varphi(x)+\gamma M$ for all $x\in C$. If $\varphi$ is $L$-smooth on $\R^d$, then
 $\varphi_\gamma(x)\le\varphi(x)+\frac{L\gamma^2}{2}\frac{d}{d+2}\le\varphi(x)+\frac{L\gamma^2}2$ for all $x$. If $\varphi$ is affine, $\varphi_\gamma=\varphi$.
\item If $\varphi$ is $M$-Lipschitz on $\R^d$, or $L$-smooth on $\R^d$, then $\varphi_\gamma$ is continuously differentiable on $\R^d$ and, for $u\sim U(\Sphere)$,
 \begin{equation}\label{eq:gradformula}
 \nabla\varphi_\gamma(x)=\frac d\gamma\,\E\bigl[\varphi(x+\gamma u)\,u\bigr]=\frac{d}{2\gamma}\,\E\bigl[\bigl(\varphi(x+\gamma u)-\varphi(x-\gamma u)\bigr)u\bigr].
 \end{equation}
\item If $\varphi$ is $M$-Lipschitz on $\R^d$, then $\nabla\varphi_\gamma$ is $L_\gamma$-Lipschitz on $\R^d$ with $L_\gamma:=\sqrt d\,M/\gamma$. If $\varphi$ is
 $L$-smooth on $\R^d$, then $\nabla\varphi_\gamma(x)=\E\nabla\varphi(x+\gamma v)$ and $\nabla\varphi_\gamma$ is $L$-Lipschitz on $\R^d$.
\item If $\varphi$ is $\mu$-strongly convex on $C_\gamma$, then $\varphi_\gamma$ is $\mu$-strongly convex on $C$.
\item (extension) If $\varphi$ is $M$-Lipschitz on $C_{\bar\gamma}$ for some $\bar\gamma\ge\gamma$, then
 \begin{equation}\label{eq:extension}
 \tilde\varphi(x):=\inf_{z\in C_{\bar\gamma}}\bigl\{\varphi(z)+M\norm{x-z}\bigr\},\qquad x\in\R^d,
 \end{equation}
 is convex and $M$-Lipschitz on $\R^d$, coincides with $\varphi$ on $C_{\bar\gamma}$, and $\tilde\varphi_\gamma=\varphi_\gamma$ on $C$.
\end{enumerate}
\end{lemma}
The proof is given in \Cref{app:smoothing}. The representation \eqref{eq:gradformula} is classical \citep{flaxman2005online,nesterov2017random};
we include a proof because the functions involved are merely Lipschitz, and the constant $\sqrt d M/\gamma$ in (c) is the one obtained in
\citet{gasnikov2022power}. \Cref{fig:smoothing} illustrates (a) and the level shift used in \Cref{sec:lagrangian}.

\begin{convention}[which functions are smoothed]\label{conv:extension}
Under \Cref{ass:lip} the expected functions $f,g_i$ are Lipschitz on $\Q_{\bar\gamma}$ only. For each of them we fix once and for all the function
to which the smoothing \eqref{eq:smoothing} is applied: if the function is $L$-smooth on $\R^d$ (\Cref{ass:smooth}, or an affine constraint), it is
smoothed as it is; otherwise it is replaced by its extension \eqref{eq:extension} from $\Q_{\bar\gamma}$ (with $C=\Q$ and $M=M_0$, resp.\ $M_g$), which is
convex and Lipschitz on $\R^d$. Throughout, $f_\gamma$ and $g_{i,\gamma}$ denote the smoothings of the functions so fixed. Nothing that the algorithms or
the guarantees can see depends on this choice: the oracle is only queried in $\Q_{\bar\gamma}$, where the extension coincides with the original function
(\Cref{lem:smoothing}(e)), so the estimators of \Cref{sec:estimators} are unchanged; by \Cref{lem:smoothing}(e) the smoothed functions on $\Q$, hence the
smoothed problem \eqref{eq:Pgamma}, its value, its solutions and its multipliers, are unchanged; and the transfer to the original problem (\Cref{lem:transfer}) only
involves points of $\Q$. What the convention buys is that each $f_\gamma$, $g_{i,\gamma}$ is convex and smooth on all of $\R^d$
(\Cref{lem:smoothing}(c)), with constant $L_0$, resp.\ $L_g$, in the smooth case and $\sqrt dM_0/\gamma$, resp.\ $\sqrt dM_g/\gamma$, in the Lipschitz case;
this is the property needed in \Cref{lem:outer}. The extension is never used for strong convexity: \Cref{ass:sc} concerns $f$ on $\Q_{\bar\gamma}$, where the
extension equals $f$, and \Cref{lem:smoothing}(d) with $C=\Q$ gives the strong convexity of $f_\gamma$ on $\Q$, which is where \Cref{lem:certificate} uses it.
\end{convention}

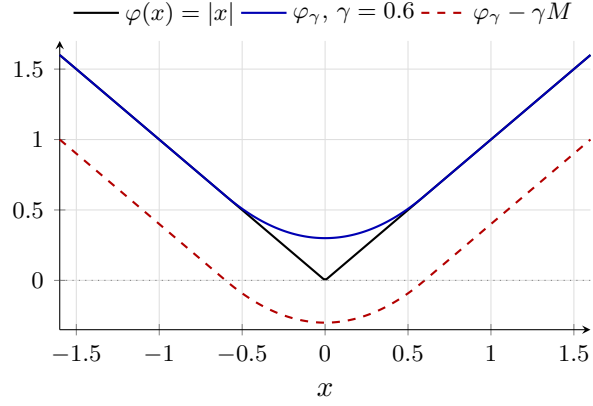
\begin{figure}[t]
\centering
\begin{tikzpicture}
\begin{axis}[width=8.6cm,height=5.4cm,xlabel={$x$},ylabel={},xmin=-1.6,xmax=1.6,ymin=-0.35,ymax=1.7,
 legend style={at={(0.5,1.02)},anchor=south,legend columns=3,font=\footnotesize,draw=none},
 axis lines=left, grid=major, grid style={gray!25}, tick label style={font=\footnotesize}]
\addplot[black,thick,domain=-1.6:1.6,samples=200] {abs(x)};
\addlegendentry{$\varphi(x)=|x|$}
\addplot[blue!70!black,thick,domain=-1.6:1.6,samples=400] {ifthenelse(abs(x)<=0.6, (x*x+0.36)/(1.2), abs(x))};
\addlegendentry{$\varphi_\gamma$, $\gamma=0.6$}
\addplot[red!70!black,thick,dashed,domain=-1.6:1.6,samples=400] {ifthenelse(abs(x)<=0.6, (x*x+0.36)/(1.2), abs(x))-0.6};
\addlegendentry{$\varphi_\gamma-\gamma M$}
\addplot[gray,dotted] coordinates {(-1.6,0) (1.6,0)};
\end{axis}
\end{tikzpicture}
\caption{One-dimensional illustration of \Cref{lem:smoothing}(a) for $\varphi(x)=|x|$ ($M=1$, $d=1$): $\varphi\le\varphi_\gamma\le\varphi+\gamma M$, and the
shifted function $h=\varphi_\gamma-\gamma M$ satisfies $h\le\varphi$, so that the set $\{h\le 0\}$ contains $\{\varphi\le0\}$. In the constrained
problem this inclusion guarantees that the optimal value of the smoothed problem does not exceed the original one (\Cref{lem:transfer}).}
\label{fig:smoothing}
\end{figure}

\subsection{Estimators}\label{sec:estimators}
Let $u\sim U(\Sphere)$ and $v\sim U(\Ball)$ be drawn independently of $\xi\sim P$. For $x\in\Q$, $y\in\R^m$ and a vector of shifts
$b=(b_1,\dots,b_m)$ (fixed in \eqref{eq:shift} below) we use the estimators
\begin{equation}\label{eq:estimators}
\begin{aligned}
 \hat g_0(x)&:=\frac{d}{2\gamma}\bigl[F(x+\gamma u,\xi)-F(x-\gamma u,\xi)\bigr]u, &\quad
 \hat J(x)&:=\frac{d}{2\gamma}\bigl[G(x+\gamma u,\xi)-G(x-\gamma u,\xi)\bigr]u^\top,\\
 \hat g_y(x)&:=\hat g_0(x)+\hat J(x)^\top y, &\quad
 \hat h(x)&:=G(x+\gamma v,\xi)-b\in\R^m ,
\end{aligned}
\end{equation}
with $\hat J(x)\in\R^{m\times d}$.
$\hat g_0$ and $\hat J$ are computed from one paired query at $(x+\gamma u,x-\gamma u)$; $\hat h$ from one single query. The following lemma
is the reason why the smoothing radius may be taken arbitrarily small when the data are smooth: the second moment of the two-point estimator does not depend on $\gamma$.

\begin{lemma}[two-point estimator]\label{lem:twopoint}
Let $x\in\Q$ and $\gamma\le\bar\gamma$. Let $\Phi(\cdot,\xi)$ be $\Lambda(\xi)$-Lipschitz on $\Q_{\bar\gamma}$ for every $\xi$, with $\E\Lambda(\xi)^2\le\Lambda^2$, let
$\phi:=\E\Phi(\cdot,\xi)$, smoothed according to \Cref{conv:extension} (so that $\phi_\gamma$ is differentiable on $\R^d$), and
$\hat\phi(x):=\frac{d}{2\gamma}[\Phi(x+\gamma u,\xi)-\Phi(x-\gamma u,\xi)]u$ with $u\sim U(\Sphere)$ independent of $\xi$. Then
\[
 \E\hat\phi(x)=\nabla\phi_\gamma(x),\qquad
 \E\norm{\hat\phi(x)}^2\le\frac{d^2}{d-1}\Lambda^2\le 2d\Lambda^2\quad(d\ge2),\qquad \E\norm{\hat\phi(x)}^2\le\Lambda^2\quad(d=1).
\]
In particular $\E\norm{\hat\phi(x)-\nabla\phi_\gamma(x)}^2\le2d\Lambda^2$ for every $d\ge1$.
\end{lemma}
\begin{proof}
Unbiasedness: by Fubini, $\E\hat\phi(x)=\frac{d}{2\gamma}\E_u[(\phi(x+\gamma u)-\phi(x-\gamma u))u]$, where the integrability follows from
$|\Phi(x+\gamma u,\xi)-\Phi(x-\gamma u,\xi)|\le2\gamma\Lambda(\xi)$ (the points $x\pm\gamma u$ lie in $\Q_{\bar\gamma}$). The points $x\pm\gamma u$ also lie in the set on
which $\phi$ coincides with the function smoothed according to \Cref{conv:extension}, so the right-hand side equals $\nabla\phi_\gamma(x)$ by
\Cref{lem:smoothing}(b), applied to the smooth, resp.\ globally Lipschitz, function that is actually smoothed. For the second moment fix $\xi$ and put
$q_\xi(u):=\frac1{2\gamma}[\Phi(x+\gamma u,\xi)-\Phi(x-\gamma u,\xi)]$. Then $q_\xi$ is odd, so $\E_uq_\xi(u)=0$, and it is $\Lambda(\xi)$-Lipschitz on
$\Sphere$ with respect to the Euclidean distance. For $d\ge2$ the Poincar\'e inequality on the sphere (\Cref{fact:poincare}) gives
$\E_uq_\xi(u)^2\le\Lambda(\xi)^2/(d-1)$, whence $\E\norm{\hat\phi(x)}^2=d^2\,\E_\xi\E_uq_\xi(u)^2\le d^2\Lambda^2/(d-1)\le2d\Lambda^2$.
For $d=1$, $|q_\xi(u)|\le\Lambda(\xi)$ and $\norm{\hat\phi}=|q_\xi|$. Finally $\E\norm{\hat\phi-\E\hat\phi}^2\le\E\norm{\hat\phi}^2$.
\end{proof}

\begin{lemma}[Jacobian estimator and its action]\label{lem:action}
Under \Cref{ass:lip}, for fixed $x\in\Q$, $y\in\R^m$ and $w\in\R^d$:
\begin{enumerate}[label=(\alph*)]
\item $\E\hat J(x)=\Jg(x)$, $\E\hat g_y(x)=\nabla f_\gamma(x)+\Jg(x)^\top y$ and
 $\E\norm{\hat g_y(x)-\nabla f_\gamma(x)-\Jg(x)^\top y}^2\le 2d\,(M_0+\norm y_1M_g)^2$.
\item $\rowinf{\Jg(x)}\le M_g$, hence $\norm{\Jg(x)w}_\infty\le M_g\norm w$.
\item $\E\norm{\hat J(x)w}_\infty^2\le d\,M_g^2\norm w^2$ and $\E\norm{(\hat J(x)-\Jg(x))w}_\infty^2\le 4d\,M_g^2\norm w^2$.
\item $\E\norm{\hat J(x)w}_\infty^2\le 45\,\kap M_g^2\norm w^2$ and
 $\E\norm{(\hat J(x)-\Jg(x))w}_\infty^2\le 92\,\kap M_g^2\norm w^2$.
\end{enumerate}
We write $\nuj:=\min\{4d,92\kap\}$, so that
$\E\norm{(\hat J(x)-\Jg(x))w}_\infty^2\le\nuj M_g^2\norm w^2$.
\end{lemma}
\begin{proof}
(a) The $i$-th row of $\hat J(x)$ is the two-point estimator of $G_i(\cdot,\xi)$, so $\E\hat J(x)=\Jg(x)$ by \Cref{lem:twopoint}. The vector
$\hat g_y(x)=\frac{d}{2\gamma}[\Phi(x+\gamma u,\xi)-\Phi(x-\gamma u,\xi)]u$ is the two-point estimator of $\Phi(\cdot,\xi):=F(\cdot,\xi)+\inner{y}{G(\cdot,\xi)}$,
which is $(M_0(\xi)+\norm y_1M_g(\xi))$-Lipschitz; by Minkowski's inequality $\E(M_0(\xi)+\norm y_1M_g(\xi))^2\le(M_0+\norm y_1M_g)^2$, and
\Cref{lem:twopoint} applies. (b) $g_{i,\gamma}$ is $M_g$-Lipschitz by \Cref{lem:smoothing}(a), so $\norm{\nabla g_{i,\gamma}(x)}\le M_g$; use \eqref{eq:rowinf}.
(c) With $q_i(u):=\frac1{2\gamma}[G_i(x+\gamma u,\xi)-G_i(x-\gamma u,\xi)]$ we have $|q_i(u)|\le M_i(\xi)\le M_g(\xi)$ and
$(\hat J(x)w)_i=d\,q_i(u)\inner{u}{w}$, hence $\E\max_i(\hat J(x)w)_i^2\le d^2\,\E[M_g(\xi)^2]\,\E\inner uw^2=dM_g^2\norm w^2$ because $u$ and $\xi$ are
independent and $\E\inner uw^2=\norm w^2/d$. Then $\E\norm{(\hat J-\Jg)w}_\infty^2\le2\E\norm{\hat Jw}_\infty^2+2\norm{\Jg w}_\infty^2\le(2d+2)M_g^2\norm w^2\le4dM_g^2\norm w^2$
by (b). Part (d) is proved in \Cref{app:smoothing}.
\end{proof}

\begin{lemma}[level estimator]\label{lem:levels}
Under \Cref{ass:lip,ass:levels}, for $x\in\Q$ and $\gamma\le\bar\gamma$, $\E\hat h(x)=g_\gamma(x)-b=:h(x)$ and
\begin{equation}\label{eq:sigmah}
 \E\norm{\hat h(x)-h(x)}_\infty^2\le\sigma_h^2:=2\sigma_g^2+8\gamma^2M_g^2 .
\end{equation}
\end{lemma}
\begin{proof}
$\E_\xi G(x+\gamma v,\xi)=g(x+\gamma v)$ and $\E_vg(x+\gamma v)=g_\gamma(x)$ give unbiasedness. Moreover
$\hat h(x)-h(x)=[G(x+\gamma v,\xi)-g(x+\gamma v)]+[g(x+\gamma v)-g_\gamma(x)]$; the first bracket has $\E\norm\cdot_\infty^2\le\sigma_g^2$ because
$x+\gamma v\in\Q_{\bar\gamma}$, and $|g_i(x+\gamma v)-g_{i,\gamma}(x)|\le|g_i(x+\gamma v)-g_i(x)|+|g_i(x)-g_{i,\gamma}(x)|\le2\gamma M_g$ by \Cref{lem:smoothing}(a).
Use $\norm{a+c}_\infty^2\le2\norm a_\infty^2+2\norm c_\infty^2$.
\end{proof}

\subsection{Mini-batches in the \texorpdfstring{$\ell_\infty$}{l-infinity}-norm}
Averaging $b$ independent copies of a centered random vector divides its \emph{Euclidean} second moment by $b$. In the $\ell_\infty$-norm, which is
the norm relevant for the dual (entropy) geometry, the same is true up to a logarithmic factor in $m$, but only after the batch size exceeds that factor.

\begin{lemma}[$\ell_\infty$ mini-batch lemma]\label{lem:batch}
Let $Z_1,\dots,Z_b\in\R^m$ be independent with $\E Z_j=0$ and $\E\norm{Z_j}_\infty^2\le\mathsf H$. Then
\[
\E\Bigl\|\frac1b\sum_{j=1}^bZ_j\Bigr\|_\infty^2\le\min\Bigl\{1,\frac{8\kap}{b}\Bigr\}\mathsf H,\qquad \kap=1+\log(2m).
\]
The same holds conditionally on a $\sigma$-field $\cG$ if the $Z_j$ are independent, centered and bounded as above conditionally on $\cG$.
\end{lemma}
The proof (\Cref{app:smoothing}) uses symmetrization and the sub-Gaussian maximal inequality; the factor $\min\{1,\cdot\}$ is Jensen's inequality.
For $b\le 8\kap$ the bound of the lemma is the trivial one, $\mathsf H$: the worst-case analysis does not certify any reduction of the $\ell_\infty$
noise before the batch size reaches the logarithmic scale (for particular distributions the actual variance may of course decrease earlier). This is why the
dual noise appears with the factor $\kap$ in all bounds below and why the parameter choices never rely on dual mini-batches smaller than $8\kap$.

\subsection{Suprema of martingale transforms over the dual set}
The certificate of \Cref{sec:lagrangian} requires a bound that holds uniformly over the dual variable $y$. Stochastic errors enter the analysis in
the form $\sum_k a_k\inner{e_k}{y-\zeta_k}$ with predictable $\zeta_k$; the supremum over $y$ of such a sum is controlled by the following
classical device \citep{nemirovski2009robust,juditsky2011solving}, which we state for a general Bregman setup.

\begin{lemma}[martingale supremum]\label{lem:martsup}
Let $Z\subseteq Z'$ be convex sets in a finite-dimensional normed space $(\mathbb E,\norm\cdot_Z)$ with dual norm $\norm\cdot_{Z,\ast}$, $Z$ compact, let
$\omega$ be a distance-generating function on $Z'$ (differentiable on $Z'$, or on its relative interior, and $1$-strongly convex with respect to $\norm\cdot_Z$
on $Z'$), $V(z,z'):=\omega(z')-\omega(z)-\inner{\nabla\omega(z)}{z'-z}$ the associated Bregman distance, $z_0\in Z'$ a point at which $\omega$ is differentiable
(not necessarily in $Z$) and $\Theta\ge\sup_{z\in Z}V(z_0,z)$. Let $(\cF_k)_{k\ge0}$ be a filtration, $e_k$ an $\cF_k$-measurable
random vector with $\E[e_k\mid\cF_{k-1}]=0$ and $\E\norm{e_k}_{Z,\ast}^2\le v_k^2$, $\zeta_k$ an $\cF_{k-1}$-measurable bounded random vector in $\mathbb E$
(not necessarily in $Z$) and $a_k>0$ deterministic, $k=1,\dots,K$. Then for every $\lambda>0$
\[
 \E\sup_{z\in Z}\sum_{k=1}^Ka_k\inner{e_k}{z-\zeta_k}\le\lambda\Theta+\frac1{2\lambda}\sum_{k=1}^Ka_k^2v_k^2,
\]
and consequently, optimizing over $\lambda$,
\[
 \E\sup_{z\in Z}\sum_{k=1}^Ka_k\inner{e_k}{z-\zeta_k}\le\Bigl(2\Theta\sum_{k=1}^Ka_k^2v_k^2\Bigr)^{1/2}.
\]
\end{lemma}
\begin{proof}
Define the auxiliary sequence $w_0:=z_0$, $w_k:=\argmin_{w\in Z}\{a_k\inner{-e_k}{w}+\lambda V(w_{k-1},w)\}$; the minimizers exist since $Z$ is compact
and lie in $Z\subseteq Z'$, so that all Bregman distances below are defined ($w_0=z_0$ may lie outside $Z$, which is why the lemma is stated with the
ambient set $Z'$; \Cref{fact:threepoint} covers this case). By the three-point inequality (\Cref{fact:threepoint}), for every $z\in Z$: $a_k\inner{-e_k}{w_k-z}\le\lambda[V(w_{k-1},z)-V(w_k,z)-V(w_{k-1},w_k)]$. Hence
\begin{align*}
 a_k\inner{e_k}{z-w_{k-1}}&=a_k\inner{e_k}{z-w_k}+a_k\inner{e_k}{w_k-w_{k-1}}\\
 &\le\lambda[V(w_{k-1},z)-V(w_k,z)]-\lambda V(w_{k-1},w_k)+a_k\norm{e_k}_{Z,\ast}\norm{w_k-w_{k-1}}_Z .
\end{align*}
By Young's inequality and $V(w_{k-1},w_k)\ge\frac12\norm{w_k-w_{k-1}}_Z^2$,
$a_k\norm{e_k}_{Z,\ast}\norm{w_k-w_{k-1}}_Z\le\frac{a_k^2}{2\lambda}\norm{e_k}_{Z,\ast}^2+\lambda V(w_{k-1},w_k)$. Summing over $k$ and telescoping,
\[
 \sum_ka_k\inner{e_k}{z-w_{k-1}}\le\lambda V(w_0,z)+\frac1{2\lambda}\sum_ka_k^2\norm{e_k}_{Z,\ast}^2\le\lambda\Theta+\frac1{2\lambda}\sum_ka_k^2\norm{e_k}_{Z,\ast}^2
\]
for all $z\in Z$ simultaneously. Finally
\[
 \sum_ka_k\inner{e_k}{z-\zeta_k}=\sum_ka_k\inner{e_k}{z-w_{k-1}}+\sum_ka_k\inner{e_k}{w_{k-1}-\zeta_k},
\]
where $w_{k-1}$ and $\zeta_k$ are $\cF_{k-1}$-measurable and bounded, so the last sum has zero expectation. Taking the supremum over $z$ in the first sum and then expectations
proves the first bound; the second follows with $\lambda=\bigl(\sum_ka_k^2v_k^2/(2\Theta)\bigr)^{1/2}$.
\end{proof}

\section{The smoothed Lagrangian and the certificate}\label{sec:lagrangian}

\subsection{The smoothed and shifted problem}
Fix $\gamma\in(0,\bar\gamma]$, and let $f_\gamma$, $g_{i,\gamma}$ be the smoothed functions of \Cref{conv:extension}. By \Cref{lem:smoothing}(a) with $C=\Q$, the
smoothed constraints satisfy $g_i\le g_{i,\gamma}\le g_i+b_i$ on $\Q$ with the \emph{known} bias bounds
\begin{equation}\label{eq:shift}
 b_i:=\begin{cases}\gamma M_g,&\text{in general},\\ L_g\gamma^2/2,&\text{if $g_i$ is $L_g$-smooth},\\ 0,&\text{if $g_i$ is affine},\end{cases}
 \qquad
 b_0:=\begin{cases}\gamma M_0,&\text{in general},\\ L_0\gamma^2/2,&\text{if $f$ is $L_0$-smooth}.\end{cases}
\end{equation}
We consider the \emph{smoothed shifted problem}
\begin{equation}\label{eq:Pgamma}
 \min_{x\in\Q}\ f_\gamma(x)\qquad\text{s.t.}\qquad h_i(x):=g_{i,\gamma}(x)-b_i\le0,\quad i=1,\dots,m,
\end{equation}
whose data are smooth on all of $\R^d$: writing $L_{0,\gamma}$ and $L_{g,\gamma}$ for the smoothness constants of $f_\gamma$ and of the $g_{i,\gamma}$ given by
\Cref{lem:smoothing}(c) and \Cref{conv:extension} ($L_{0,\gamma}=L_0$ if $f$ is $L_0$-smooth on $\R^d$ and $L_{0,\gamma}=\sqrt dM_0/\gamma$ otherwise; $L_{g,\gamma}$ is the
largest of the corresponding constants of the $g_{i,\gamma}$, \ie\ $L_g$ if all constraints are $L_g$-smooth, $0$ if they are affine, and $\sqrt dM_g/\gamma$ if some of
them are merely Lipschitz), we put
\begin{equation}\label{eq:HL}
 H:=L_{0,\gamma}+\Rmult L_{g,\gamma},\qquad L\ge H\quad\text{(any upper bound used by an algorithm)},
\end{equation}
where $\Rmult$ is the multiplier bound \eqref{eq:R} below. We denote by $f_\gamma^\ast$ the optimal value of \eqref{eq:Pgamma}, by $x^\ast_\gamma$ an
optimal solution and by $b_g:=\max_ib_i$.

\begin{lemma}[transfer between the original and the smoothed problem]\label{lem:transfer}
Under \Cref{ass:convex,ass:lip}:
\begin{enumerate}[label=(\alph*)]
\item every feasible point of \eqref{eq:P} is feasible for \eqref{eq:Pgamma}, and $f^\ast_\gamma\le f^\ast+b_0$;
\item \Cref{ass:slater} for \eqref{eq:P} implies $h_i(x^{\mathrm s})\le-\rho$ for all $i$, \ie\ the Slater condition for \eqref{eq:Pgamma} with the same point and margin;
\item if $\hat x\in\Q$ satisfies $f_\gamma(\hat x)-f^\ast_\gamma\le\eps'$ and $\norm{\pos{h(\hat x)}}_\infty\le\eps'$, then
 $f(\hat x)-f^\ast\le\eps'+b_0$ and $\viol(\hat x)=\norm{\pos{g(\hat x)}}_\infty\le\eps'+b_g$.
\end{enumerate}
\end{lemma}
\begin{proof}
(a) If $g(x)\le0$ then $h_i(x)=g_{i,\gamma}(x)-b_i\le g_i(x)\le0$. Hence the feasible set of \eqref{eq:Pgamma} contains that of \eqref{eq:P} and
$f^\ast_\gamma\le\min\{f_\gamma(x):g(x)\le0,x\in\Q\}\le\min\{f(x)+b_0:g(x)\le0,x\in\Q\}=f^\ast+b_0$ by \Cref{lem:smoothing}(a).
(b) $h_i(x^{\mathrm s})\le g_i(x^{\mathrm s})\le-\rho$.
(c) $f(\hat x)-f^\ast\le f_\gamma(\hat x)-f^\ast\le(f_\gamma(\hat x)-f^\ast_\gamma)+(f^\ast_\gamma-f^\ast)\le\eps'+b_0$ by (a). Moreover
$g_i(\hat x)\le g_{i,\gamma}(\hat x)=h_i(\hat x)+b_i\le\pos{h_i(\hat x)}+b_g$.
\end{proof}

\subsection{Lagrange multipliers}
\begin{lemma}[multipliers]\label{lem:multipliers}
Under \Cref{ass:convex,ass:lip,ass:slater}, problem \eqref{eq:Pgamma} has an optimal solution $x^\ast_\gamma\in\Q$ and a vector
$y^\ast_\gamma\in\R^m_+$ such that
\[
 x^\ast_\gamma\in\argmin_{x\in\Q}\bigl\{f_\gamma(x)+\inner{y^\ast_\gamma}{h(x)}\bigr\},\qquad \inner{y^\ast_\gamma}{h(x^\ast_\gamma)}=0,\qquad
 \norm{y^\ast_\gamma}_1\le\frac{f_\gamma(x^{\mathrm s})-f^\ast_\gamma}{\rho}\le\frac{M_0D}{\rho}.
\]
\end{lemma}
\begin{proof}
$f_\gamma$ and $h_i$ are finite convex functions on $\R^d$ (\Cref{lem:smoothing}(a) and \Cref{conv:extension}), $\Q$ is compact and the Slater condition holds by \Cref{lem:transfer}(b); the existence
of $x^\ast_\gamma$ follows from continuity and compactness, the existence of Kuhn--Tucker multipliers with the saddle-point property and complementary
slackness is the classical Kuhn--Tucker theorem for convex programs under Slater's condition (\eg\ \citealp[Thm.~28.2 and Cor.~28.3.1]{rockafellar1970convex}).
For the bound, the saddle-point property gives
$f^\ast_\gamma=\min_{x\in\Q}\{f_\gamma(x)+\inner{y^\ast_\gamma}{h(x)}\}\le f_\gamma(x^{\mathrm s})+\inner{y^\ast_\gamma}{h(x^{\mathrm s})}\le f_\gamma(x^{\mathrm s})-\rho\norm{y^\ast_\gamma}_1$,
and $f_\gamma(x^{\mathrm s})-f^\ast_\gamma=f_\gamma(x^{\mathrm s})-f_\gamma(x^\ast_\gamma)\le M_0\norm{x^{\mathrm s}-x^\ast_\gamma}\le M_0D$ since $f_\gamma$ is $M_0$-Lipschitz on $\Q$
(\Cref{lem:smoothing}(a) with $C=\Q$).
\end{proof}

Throughout we fix a constant
\begin{equation}\label{eq:R}
 \Rmult\ \ge\ \norm{y^\ast_\gamma}_1+1,\qquad\text{for instance}\qquad \Rmult:=1+\frac{M_0D}{\rho},
\end{equation}
and we write $M_R:=M_0+\Rmult M_g$. The dual variable ranges over the truncated orthant and its lifting to a simplex,
\begin{equation}\label{eq:Ysets}
 Y_\Rmult:=\{y\in\R^m_+:\norm y_1\le\Rmult\},\qquad
 \tilde Y_\Rmult:=\Bigl\{\tilde y=(\tilde y_0,y)\in\R^{m+1}_+:\ \tilde y_0+\textstyle\sum_{i=1}^my_i=\Rmult\Bigr\},
\end{equation}
where the map $\tilde y\mapsto y$ (dropping the slack coordinate $\tilde y_0$) is a bijection from $\tilde Y_\Rmult$ onto $Y_\Rmult$.

\subsection{Lagrangian, restricted gap and certificate}
The smoothed Lagrangian and the certificate are
\begin{equation}\label{eq:Lag}
 \LagS(x,y):=f_\gamma(x)+\inner{y}{h(x)},\qquad (x,y)\in\Q\times Y_\Rmult,
\end{equation}
\begin{equation}\label{eq:cert}
 \Cert_{\Rmult,\gamma}(x):=\sup_{y\in Y_\Rmult}\LagS(x,y)-f^\ast_\gamma=f_\gamma(x)-f^\ast_\gamma+\Rmult\,\norm{\pos{h(x)}}_\infty,\qquad x\in\Q .
\end{equation}
The identity in \eqref{eq:cert} holds because $\sup_{y\in Y_\Rmult}\inner{y}{h}=\Rmult\max\{0,\max_ih_i\}=\Rmult\norm{\pos{h}}_\infty$.
For $y\in Y_\Rmult$ the function $\LagS(\cdot,y)$ is convex and $L$-smooth on $\R^d$ with $L$ from \eqref{eq:HL}: indeed $\nabla_x\LagS(x,y)=\nabla f_\gamma(x)+\Jg(x)^\top y$
and $\norm{(\Jg(x)-\Jg(x'))^\top y}\le\sum_iy_i\norm{\nabla g_{i,\gamma}(x)-\nabla g_{i,\gamma}(x')}\le\Rmult L_{g,\gamma}\norm{x-x'}$.

\begin{lemma}[certificate]\label{lem:certificate}
Under \Cref{ass:convex,ass:lip,ass:slater} and \eqref{eq:R}, for every $x\in\Q$,
\begin{equation}\label{eq:certbounds}
 \Cert_{\Rmult,\gamma}(x)\ \ge\ \max\bigl\{f_\gamma(x)-f^\ast_\gamma,\ \norm{\pos{h(x)}}_\infty\bigr\}\ \ge\ 0,
\end{equation}
and for every $\hat y\in Y_\Rmult$
\begin{equation}\label{eq:restrictedgap}
 \LagS(x^\ast_\gamma,\hat y)\le f^\ast_\gamma,\qquad\text{hence}\qquad \Cert_{\Rmult,\gamma}(x)\le\sup_{y\in Y_\Rmult}\LagS(x,y)-\LagS(x^\ast_\gamma,\hat y).
\end{equation}
If in addition \Cref{ass:sc} holds, then $\frac\mu2\norm{x-x^\ast_\gamma}^2\le\Cert_{\Rmult,\gamma}(x)$ for all $x\in\Q$.
\end{lemma}
\begin{proof}
By the saddle-point property of \Cref{lem:multipliers}, for all $x\in\Q$,
$f_\gamma(x)+\inner{y^\ast_\gamma}{h(x)}\ge f_\gamma(x^\ast_\gamma)+\inner{y^\ast_\gamma}{h(x^\ast_\gamma)}=f^\ast_\gamma$, hence
$f_\gamma(x)-f^\ast_\gamma\ge-\inner{y^\ast_\gamma}{h(x)}\ge-\inner{y^\ast_\gamma}{\pos{h(x)}}\ge-\norm{y^\ast_\gamma}_1\norm{\pos{h(x)}}_\infty\ge-(\Rmult-1)\norm{\pos{h(x)}}_\infty$.
Adding $\Rmult\norm{\pos{h(x)}}_\infty$ yields $\Cert_{\Rmult,\gamma}(x)\ge\norm{\pos{h(x)}}_\infty\ge0$; and $\Cert_{\Rmult,\gamma}(x)\ge f_\gamma(x)-f^\ast_\gamma$ is
obvious. For \eqref{eq:restrictedgap}, $h(x^\ast_\gamma)\le0$ and $\hat y\ge0$ give $\LagS(x^\ast_\gamma,\hat y)\le f_\gamma(x^\ast_\gamma)=f^\ast_\gamma$.
Under \Cref{ass:sc}, $f_\gamma$ is $\mu$-strongly convex on $\Q$ (\Cref{lem:smoothing}(d) with $C=\Q$, $C_\gamma\subseteq\Q_{\bar\gamma}$; the function smoothed under
\Cref{conv:extension} coincides with $f$ on $\Q_{\bar\gamma}$), so $\LagS(\cdot,y^\ast_\gamma)$ is $\mu$-strongly convex on $\Q$ and minimized over $\Q$ at $x^\ast_\gamma$; thus $\LagS(x,y^\ast_\gamma)-f^\ast_\gamma\ge\frac\mu2\norm{x-x^\ast_\gamma}^2$ (strong convexity plus the first-order optimality condition
on $\Q$), and $\LagS(x,y^\ast_\gamma)\le\sup_{y\in Y_\Rmult}\LagS(x,y)$ because $y^\ast_\gamma\in Y_\Rmult$.
\end{proof}

\Cref{lem:certificate,lem:transfer} reduce the task to bounding $\E\sup_{y\in Y_\Rmult}\LagS(\hat x,y)-\LagS(x^\ast_\gamma,\hat y)$ for the output $\hat x$ and an
arbitrary dual output $\hat y$ of a method: an expected bound $\eps'$ on this \emph{restricted duality gap} gives an $(\eps'+b_0,\eps'+b_g)$-solution
of the original problem in expectation. Note that the supremum is over $y$ only; the primal comparator is the fixed point $x^\ast_\gamma$. This
asymmetry is exploited in the stochastic analysis: primal errors need only be centered, dual errors must be controlled uniformly (\Cref{lem:martsup}).

\subsection{Proximal geometry}
The primal space carries the Euclidean distance and the dual space the entropy on the lifted simplex $\tilde Y_\Rmult$:
\begin{equation}\label{eq:bregman}
 \Vx(x,x'):=\tfrac12\norm{x-x'}^2,\qquad
 \Vy(\tilde y,\tilde y'):=\Rmult\sum_{i=0}^m\tilde y'_i\log\frac{\tilde y'_i}{\tilde y_i}\quad(\tilde y\in\operatorname{ri}\tilde Y_\Rmult,\ \tilde y'\in\tilde Y_\Rmult).
\end{equation}
Both are Bregman distances of distance-generating functions ($\frac12\norm x^2$ and $\Rmult\sum_i\tilde y_i\log\tilde y_i$). We use the uniform
dual center $\tilde y^{\mathrm u}:=\frac{\Rmult}{m+1}(1,\dots,1)$ and the radii
\begin{equation}\label{eq:radii}
 \cA:=\Vx(x_0,x^\ast_\gamma)=\tfrac12\norm{x_0-x^\ast_\gamma}^2\le\tfrac12D^2,\qquad
 \cB:=\max_{\tilde y\in\tilde Y_\Rmult}\Vy(\tilde y^{\mathrm u},\tilde y)=\Rmult^2\log(m+1).
\end{equation}

\begin{lemma}[facts about the proximal steps]\label{lem:prox}
\begin{enumerate}[label=(\alph*)]
\item (Pinsker) For all $\tilde y\in\operatorname{ri}\tilde Y_\Rmult$, $\tilde y'\in\tilde Y_\Rmult$ with last $m$ coordinates $y,y'$:
 \[ \Vy(\tilde y,\tilde y')\ge\tfrac12\norm{\tilde y'-\tilde y}_1^2\ge\tfrac12\norm{y'-y}_1^2 . \]
\item (radius) $\Vy(\tilde y^{\mathrm u},\tilde y)\le\Rmult^2\log(m+1)=\cB$ for all $\tilde y\in\tilde Y_\Rmult$.
\item (dual step) For $\tilde y\in\operatorname{ri}\tilde Y_\Rmult$, $q\in\R^m$, $\tau>0$ and $\tilde q:=(0,q)\in\R^{m+1}$, let
 $\tilde y^+:=\argmax_{\tilde y'\in\tilde Y_\Rmult}\{\inner{\tilde q}{\tilde y'}-\tau\Vy(\tilde y,\tilde y')\}$. Then
 $\tilde y^+_i=\Rmult\,\tilde y_i e^{\tilde q_i/(\tau\Rmult)}/\sum_{l=0}^m\tilde y_le^{\tilde q_l/(\tau\Rmult)}$ and $\tilde y^+\in\operatorname{ri}\tilde Y_\Rmult$.
 Moreover, for all $\tilde y'\in\tilde Y_\Rmult$,
 \[
 \inner{q}{y'-y^+}=\inner{\tilde q}{\tilde y'-\tilde y^+}\le\tau\bigl[\Vy(\tilde y,\tilde y')-\Vy(\tilde y^+,\tilde y')-\Vy(\tilde y,\tilde y^+)\bigr].
 \]
\item (primal step) For $x_{\mathrm c},p\in\Q$, $r\in\R^d$ and $\eta,\beta\ge0$ with $\eta+\beta>0$, the point
 $p^+:=\argmin_{x\in\Q}\{\inner rx+\eta\Vx(x_{\mathrm c},x)+\beta\Vx(p,x)\}$ is the Euclidean projection onto $\Q$ of $(\eta x_{\mathrm c}+\beta p-r)/(\eta+\beta)$ and satisfies for all $x\in\Q$
 \[
 \inner{r}{p^+-x}\le\eta\bigl[\Vx(x_{\mathrm c},x)-\Vx(x_{\mathrm c},p^+)-\Vx(p^+,x)\bigr]+\beta\bigl[\Vx(p,x)-\Vx(p,p^+)-\Vx(p^+,x)\bigr].
 \]
\end{enumerate}
\end{lemma}
\begin{proof}
(a) With $p:=\tilde y'/\Rmult$, $q:=\tilde y/\Rmult$ probability vectors, $\Vy(\tilde y,\tilde y')=\Rmult^2\KL(p\|q)\ge\frac{\Rmult^2}{2}\norm{p-q}_1^2=\frac12\norm{\tilde y'-\tilde y}_1^2$ by
Pinsker's inequality (\Cref{fact:pinsker}); dropping the coordinate $0$ does not increase the $\ell_1$-norm.
(b) $\KL(p\|\text{uniform})=\log(m+1)-\mathrm{Ent}(p)\le\log(m+1)$.
(c) The objective $\tilde y'\mapsto\inner{\tilde q}{\tilde y'}-\tau\Vy(\tilde y,\tilde y')$ is strictly concave on the simplex and tends to $-\infty$ in
no direction; the Lagrange conditions $\tilde q_i-\tau\Rmult(\log\tilde y'_i-\log\tilde y_i+1)=\nu$ give the stated formula, which has positive coordinates.
The inequality is the three-point inequality (\Cref{fact:threepoint}) for the concave maximization problem with the Bregman distance $\Vy$.
(d) The objective equals $\frac{\eta+\beta}{2}\norm{x-c}^2+\text{const}$ with $c=(\eta x_{\mathrm c}+\beta p-r)/(\eta+\beta)$, so $p^+$ is the projection of $c$.
The inequality follows from the optimality condition $\inner{r+\eta(p^+-x_{\mathrm c})+\beta(p^+-p)}{x-p^+}\ge0$ and the identity
$2\inner{p^+-c'}{x-p^+}=\norm{x-c'}^2-\norm{p^+-c'}^2-\norm{x-p^+}^2$ applied with $c'=x_{\mathrm c}$ and $c'=p$.
\end{proof}

\section{Baseline: batched stochastic Mirror--Prox}\label{sec:mp}

The saddle-point problem $\min_{x\in\Q}\max_{y\in Y_\Rmult}\LagS(x,y)$ has a convex--concave objective whose partial gradients are Lipschitz. The natural
baseline is therefore the stochastic Mirror--Prox (extragradient) method of \citet{nemirovski2004prox,juditsky2011solving} applied to the monotone
operator
\begin{equation}\label{eq:operator}
 \cG(z):=\bigl(\nabla_x\LagS(x,y),\,-\nabla_y\LagS(x,y)\bigr)=\bigl(\nabla f_\gamma(x)+\Jg(x)^\top y,\ -h(x)\bigr),\qquad z=(x,y)\in\cZ:=\Q\times Y_\Rmult,
\end{equation}
with the estimator $\hat\cG(z)=(\hat g_y(x),-\hat h(x))$ of \eqref{eq:estimators}, averaged over mini-batches. We use the joint geometry
\begin{equation}\label{eq:jointgeometry}
\begin{gathered}
 V(z,z'):=\alpha\Vx(x,x')+\alpha^{-1}\Vy(\tilde y,\tilde y'),\\
 \norm{z}_\alpha^2:=\alpha\norm x^2+\alpha^{-1}\norm{\tilde y}_1^2,\qquad
 \norm{(p,q)}_{\alpha,\ast}^2:=\alpha^{-1}\norm p^2+\alpha\norm q_\infty^2 ,
\end{gathered}
\end{equation}
where $\alpha>0$ balances the two blocks; $V$ is $1$-strongly convex with respect to $\norm\cdot_\alpha$ by \Cref{lem:prox}(a), and $\norm\cdot_{\alpha,\ast}$
is the dual norm. The constants that enter the analysis are
\begin{equation}\label{eq:Lalpha}
 L_\alpha:=\frac12\Bigl(\frac H\alpha+\sqrt{\frac{H^2}{\alpha^2}+4M_g^2}\Bigr)\le\frac H\alpha+M_g,\quad
 \sigma_\ast^2:=\frac1b\Bigl(\frac{2dM_R^2}{\alpha}+\alpha\min\{b,8\kap\}\sigma_h^2\Bigr),\quad
 \cD_\alpha^2:=\alpha\cA+\frac{\cB}{\alpha}.
\end{equation}

\begin{algorithm}[t]
\caption{\BSMP: batched stochastic Mirror--Prox for the smoothed Lagrangian}\label{alg:mp}
\begin{algorithmic}[1]
\Require $N\ge1$, mini-batch $b\ge1$, balance $\alpha>0$, step $\eta\in(0,1/(\sqrt3L_\alpha)]$, radius $\gamma$, shifts $b_i$, bound $\Rmult$; $x_1\in\Q$, $\tilde y_1:=\tilde y^{\mathrm u}$.
\For{$k=1,\dots,N$}
 \State \textbf{Round $2k-1$:} at $x_k$ issue $b$ paired queries and $b$ single queries; form $\hat g_k:=\frac1b\sum_{j}\hat g_{y_k,j}(x_k)$ and $\hat h_k:=\frac1b\sum_j\hat h_j(x_k)$.
 \State $x^w_k:=\Pi_\Q\bigl(x_k-\tfrac{\eta}{\alpha}\hat g_k\bigr)$;\quad $\tilde y^w_k:=$ dual step of \Cref{lem:prox}(c) from $\tilde y_k$ with $q=\eta\hat h_k$ and $\tau=1/\alpha$.
 \State \textbf{Round $2k$:} at $x^w_k$ issue $b$ paired and $b$ single queries; form $\hat g'_k:=\frac1b\sum_j\hat g_{y^w_k,j}(x^w_k)$ and $\hat h'_k:=\frac1b\sum_j\hat h_j(x^w_k)$.
 \State $x_{k+1}:=\Pi_\Q\bigl(x_k-\tfrac{\eta}{\alpha}\hat g'_k\bigr)$;\quad $\tilde y_{k+1}:=$ dual step from $\tilde y_k$ with $q=\eta\hat h'_k$ and $\tau=1/\alpha$.
\EndFor
\Ensure $\Xr_N:=\frac1N\sum_{k=1}^Nx^w_k$ (and $\yr_N:=\frac1N\sum_ky^w_k$).
\end{algorithmic}
\end{algorithm}

Each iteration of \Cref{alg:mp} consists of two sequential rounds and issues $6b$ vector calls ($2b$ paired and $b$ single queries in each round); every
call is used for the constraint vector, and the $4b$ calls of the paired queries are also used for the objective. Hence
\[
 \Nseq=2N,\qquad \Norc=\Ng=6bN,\qquad \Nf=4bN,\qquad \Nmv=\Nseq .
\]

\begin{theorem}[batched stochastic Mirror--Prox]\label{thm:mp}
Let \Cref{ass:convex,ass:lip,ass:levels,ass:slater} hold, $\gamma\le\bar\gamma$, $\Rmult$ as in \eqref{eq:R}, $H$ as in \eqref{eq:HL}, and $\eta\le1/(\sqrt3L_\alpha)$. Then the output of
\Cref{alg:mp} satisfies
\begin{equation}\label{eq:mp-bound}
 \E\,\Cert_{\Rmult,\gamma}(\Xr_N)\le\frac{\cD_\alpha^2}{\eta N}+3\eta\sigma_\ast^2+\sigma_\ast\cD_\alpha\sqrt{\frac2N}.
\end{equation}
With $\eta:=\min\{1/(\sqrt3L_\alpha),\ \cD_\alpha/(\sigma_\ast\sqrt{3N})\}$,
\begin{equation}\label{eq:mp-rate}
 \E\,\Cert_{\Rmult,\gamma}(\Xr_N)\le\frac{\sqrt3L_\alpha\cD_\alpha^2}{N}+\frac{5\sigma_\ast\cD_\alpha}{\sqrt N}.
\end{equation}
\end{theorem}
The proof is in \Cref{app:mp}. It is the classical analysis of \citet{juditsky2011solving}, written for the restricted gap of \Cref{lem:certificate}:
the stochastic term $\sigma_\ast\cD_\alpha\sqrt{2/N}$ comes from \Cref{lem:martsup} applied to the dual block, and $3\eta\sigma_\ast^2$ from the variance
of the two mini-batches of an iteration. The constant $L_\alpha$ is the spectral norm of the $2\times2$ matrix $\bigl(\begin{smallmatrix}H/\alpha&M_g\\M_g&0\end{smallmatrix}\bigr)$
that dominates the block Lipschitz structure of $\cG$.

\begin{corollary}[complexity of \BSMP]\label{cor:mp-complexity}
In the setting of \Cref{thm:mp}, let $\eps_{\mathrm s}>0$, choose $\alpha:=\sqrt{2\cB}/D$, the step $\eta$ of \eqref{eq:mp-rate},
\begin{equation}\label{eq:mp-params}
 N:=\Bigl\lceil\frac{2\sqrt3\,\bigl(HD^2+M_gD\sqrt{2\cB}\bigr)}{\eps_{\mathrm s}}\Bigr\rceil,\qquad
 b:=\max\Bigl\{1,\ \Bigl\lceil\frac{200\,(dM_R^2D^2+8\kap\cB\sigma_h^2)}{N\eps_{\mathrm s}^2}\Bigr\rceil\Bigr\}.
\end{equation}
Then $\E\Cert_{\Rmult,\gamma}(\Xr_N)\le\eps_{\mathrm s}$, and
\begin{gather*}
 \Nseq=2N,\qquad \Nf=4bN\le4N+\frac{800\,(dM_R^2D^2+8\kap\cB\sigma_h^2)}{\eps_{\mathrm s}^2},\\
 \Norc=\Ng=6bN\le6N+\frac{1200\,(dM_R^2D^2+8\kap\cB\sigma_h^2)}{\eps_{\mathrm s}^2}.
\end{gather*}
\end{corollary}
\begin{proof}
With $\alpha=\sqrt{2\cB}/D$ and $\cA\le D^2/2$ we get $\cD_\alpha^2\le\sqrt{2\cB}\,D$, $L_\alpha\cD_\alpha^2\le(H/\alpha+M_g)\cD_\alpha^2\le HD^2+M_gD\sqrt{2\cB}$ and
$\sigma_\ast^2\cD_\alpha^2\le\frac1b\bigl(2dM_R^2D^2+2\cB\min\{b,8\kap\}\sigma_h^2\bigr)\le\frac2b\,(dM_R^2D^2+8\kap\cB\sigma_h^2)$. By \eqref{eq:mp-rate} the two terms are at most
$\eps_{\mathrm s}/2$ each: the first because $N\ge2\sqrt3(HD^2+M_gD\sqrt{2\cB})/\eps_{\mathrm s}$, the second because
$25\sigma_\ast^2\cD_\alpha^2/N\le 50(dM_R^2D^2+8\kap\cB\sigma_h^2)/(bN)\le\eps_{\mathrm s}^2/4$. The counts follow from $b\le1+200(\cdots)/(N\eps_{\mathrm s}^2)$.
\end{proof}

\begin{corollary}[\BSMP\ for smooth and for nonsmooth data]\label{cor:mp-cases}
Let $\eps>0$ and put $\eps_{\mathrm s}:=\eps/2$.
\begin{enumerate}[label=(\alph*)]
\item (smooth data) Under \Cref{ass:smooth}, take $\gamma\le\min\{\bar\gamma,\sqrt{\eps/(2\max\{L_0,L_g\})},\eps/(4M_g)\}$, $b_0=L_0\gamma^2/2$, $b_i=L_g\gamma^2/2$ and $H=L_0+\Rmult L_g$.
 Then the output is an $\eps$-solution in expectation with
 $\Nseq=O\bigl((L_0+\Rmult L_g)D^2\eps^{-1}+M_g\Rmult D\sqrt{\log(m+1)}\,\eps^{-1}\bigr)$ and $\Norc=O\bigl(\Nseq+(dM_R^2D^2+\Lam\Rmult^2\sigma_g^2)\eps^{-2}\bigr)$.
\item (nonsmooth data) Take $\gamma:=\eps/(4M)$ (assuming $\eps\le4M\bar\gamma$), $b_0=\gamma M_0$, $b_i=\gamma M_g$ and $H=\sqrt dM_R/\gamma=4\sqrt dM_RM/\eps$. Then the output is an
 $\eps$-solution in expectation with $\Nseq=O\bigl(\sqrt dM_RMD^2\eps^{-2}+M_g\Rmult D\sqrt{\log(m+1)}\,\eps^{-1}\bigr)$ and
 $\Norc=O\bigl(\Nseq+(dM_R^2D^2+\Lam\Rmult^2\sigma_g^2)\eps^{-2}\bigr)$.
\end{enumerate}
\end{corollary}
\begin{proof}
In both cases $b_0,b_g\le\eps/4$, so by \Cref{lem:transfer}(c) and \Cref{lem:certificate} the output with $\E\Cert_{\Rmult,\gamma}\le\eps/2$ satisfies
$\E[f(\Xr_N)-f^\ast]\le\eps/2+\eps/4$ and $\E\viol(\Xr_N)\le\eps/2+\eps/4$. In both cases $\gamma\le\eps/(4M_g)$, so $8\gamma^2M_g^2\le\eps^2/2$ and
$\sigma_h^2\le2\sigma_g^2+\eps^2/2$ by \eqref{eq:sigmah}; the contribution of $\eps^2/2$ to $\Norc$ is $O(\Lam\Rmult^2)$, a constant. Insert $H$,
$\cB=\Rmult^2\log(m+1)$ and $\eps_{\mathrm s}=\eps/2$ into \Cref{cor:mp-complexity}.
\end{proof}

\begin{remark}[what the baseline shows]
The number of rounds of \BSMP\ is proportional to $H/\eps$, where $H=L_{0,\gamma}+\Rmult L_{g,\gamma}$ is the smoothness of the Lagrangian. For nonsmooth
data $H\propto\sqrt d/\eps$, so the rounds scale as $\sqrt d/\eps^2$: better than the $d/\eps^2$ rounds of plain stochastic mirror descent with
two-point estimates (which cannot be parallelized, since its rate $MD/\sqrt N$ does not improve with the batch size), but far from the
$d^{1/4}/\eps$ of the accelerated unconstrained methods of \citet{gasnikov2022power}. The next section closes this gap.
\end{remark}

\section{Zeroth-order primal--dual sliding}\label{sec:sliding}

\subsection{Idea}
For smooth data the round complexity $O(H D^2/\eps)$ of Mirror--Prox is not optimal: for the \emph{unconstrained} problem accelerated methods with
mini-batches need only $O(\sqrt{L_0D^2/\eps})$ rounds. The obstacle is the coupling term $\inner y{h(x)}$, whose Lipschitz constant $M_g$ does
not improve under acceleration \citep{ouyang2021lower}. \citet{zhanglan2022} resolved this in the deterministic first-order setting by
\emph{sliding}: in phase $t$ the constraints are linearized at a center $u_t$, and the resulting bilinear saddle-point subproblem is solved by $S_t$
cheap primal--dual steps that only multiply vectors by the fixed Jacobian $\nabla g(u_t)$; the gradients themselves are evaluated once per phase.

In the stochastic zeroth-order setting the Jacobian at $u_t$ is not available; it must be estimated from paired queries at $u_t$. Reusing a single
estimate $\hat J_t$ during the whole inner loop is not admissible: the inner dual iterates depend on $\hat J_t$, and the error terms
$\inner{\bar y_t}{(\hat J_t-J_t)(x^\ast_\gamma-u_t)}$ are not centered: a bias of order $\Rmult M_gD\sqrt{\nuj/b}$ per phase (\Cref{lem:action}) would have to
be brought below $\eps$, which forces a mini-batch of size $b=\Omega(\Rmult^2M_g^2D^2\nuj/\eps^2)$ in each of the $N=O(\eps^{-1/2})$ phases, i.e.\ a total of
$O(\eps^{-5/2})$ calls. Instead we draw, \emph{in a single round}, $2S_t+1$ independent mini-batches at the same center $u_t$ and assign a fresh
mini-batch to each inner step: one for the primal direction (bank $A$) and one for the dual direction (bank $B$). Because the locations of all
queries of the round are known in advance, the round complexity is the number of phases, $N$, while the estimation errors of the inner steps remain
conditionally centered and are absorbed by the increments of the inner iterates. The price is that every inner step consumes oracle calls, so that
the number of calls is proportional to the number of inner steps $K_N=\sum_tS_t$ rather than to $N$; since $K_N=O(\eps^{-1})$ and each step uses
$b=O(\eps^{-1})$ samples, the total $O(\eps^{-2})$ is unchanged.

\subsection{The method}
Fix the parameters $L\ge H$ (see \eqref{eq:HL}), $M_g$, $\Rmult$, $\gamma$, the shifts $b_i$, a dual scale $\rs>0$, a stabilizer $\lambda\ge0$, two mini-batch
sizes $b_p,b_y\ge1$ (for the primal banks and the reserve batch, resp.\ for the dual banks; the two sizes are allowed to differ because the two kinds of
banks carry different information at different prices, see \Cref{cor:sliding-complexity}) and the number of phases $N$. The schedule is
\begin{equation}\label{eq:schedule}
\begin{gathered}
 \tau_t=\frac{t-1}{2},\qquad \theta_t=\frac{t-1}{t},\qquad S_t=\max\Bigl\{1,\Bigl\lceil\frac{M_g\rs\,t}{L}\Bigr\rceil\Bigr\},\qquad a_t=\frac t{S_t},\\
 \eta_t=\frac{2L}t,\qquad \beta_t=\frac{L+\lambda}{a_t},\qquad
 W_N=\frac{N(N+1)}2,\qquad K_N=\sum_{t=1}^NS_t,\qquad A_N=\sum_{t=1}^N\frac{t^2}{S_t}.
\end{gathered}
\end{equation}
Inner steps are indexed globally by $k=1,\dots,K_N$; slot $k$ belongs to phase $t(k)$ and the inner steps of phase $t$ are $k=K_{t-1}+1,\dots,K_t$.
For a mini-batch $A$ of paired queries at a point $u$ we write $\hat g_0^A(u)$ and $\hat J^A(u)$ for the averages of the estimators \eqref{eq:estimators} over
the batch, and for a mini-batch $B$ of single queries we write $\hat h^B(u)$.

\begin{algorithm}[t]
\caption{\ZOS: zeroth-order primal--dual sliding with fresh mini-batches}\label{alg:sliding}
\begin{algorithmic}[1]
\Require parameters as in \eqref{eq:schedule}; $x_0\in\Q$; set $x_{-1}:=u_0:=p_0:=x_0$, $\tilde y_0=\tilde y_{-1}:=\tilde y^{\mathrm u}$, $k:=0$.
\For{$t=1,\dots,N$}
 \State $u_t:=\dfrac{\tau_tu_{t-1}+x_{t-1}+\theta_t(x_{t-1}-x_{t-2})}{1+\tau_t}$.\hfill(center; $u_1=x_0$)
 \State \textbf{Round $t$.} Issue at $u_t$: $S_t$ mini-batches $A_{k}$ ($b_p$ paired queries each, $k=K_{t-1}+1,\dots,K_t$) and
 \Statex\hspace{\algorithmicindent} $S_t$ mini-batches $B_k$ ($b_y$ paired and $b_y$ single queries each); if $t\ge2$, issue at $u_{t-1}$ one reserve
 \Statex\hspace{\algorithmicindent} mini-batch $A^-_t$ of $b_p$ paired queries (constraint values only).
 \For{$s=1,\dots,S_t$}
  \State $k:=k+1$.
  \If{$s=1$} $\hat r_k:=\hat g_0^{A_k}(u_t)+\hat J^{A_k}(u_t)^\top y_{k-1}+\dfrac{a_{t-1}}{a_t}\,\hat J^{A^-_t}(u_{t-1})^\top(y_{k-1}-y_{k-2})$
  \Statex\hspace{2\algorithmicindent} (the last term is absent for $t=1$)
  \Else\ $\hat r_k:=\hat g_0^{A_k}(u_t)+\hat J^{A_k}(u_t)^\top(2y_{k-1}-y_{k-2})$
  \EndIf
  \State $p_k:=\argmin_{x\in\Q}\bigl\{\inner{\hat r_k}{x}+\eta_t\Vx(x_{t-1},x)+\beta_t\Vx(p_{k-1},x)\bigr\}$\hfill(\Cref{lem:prox}(d))
  \State $\hat q_k:=\hat h^{B_k}(u_t)+\hat J^{B_k}(u_t)\,(p_k-u_t)\in\R^m$
  \State $\tilde y_k:=\argmax_{\tilde y\in\tilde Y_\Rmult}\bigl\{\inner{(0,\hat q_k)}{\tilde y}-\rs^{-2}\beta_t\Vy(\tilde y_{k-1},\tilde y)\bigr\}$\hfill(\Cref{lem:prox}(c))
 \EndFor
 \State $x_t:=\frac1{S_t}\sum_{k=K_{t-1}+1}^{K_t}p_k$,\qquad $\bar y_t:=\frac1{S_t}\sum_{k=K_{t-1}+1}^{K_t}y_k$.
\EndFor
\Ensure $\Xr_N:=\frac1{W_N}\sum_{t=1}^Nt\,x_t$.
\end{algorithmic}
\end{algorithm}

In round $t$ the algorithm issues $(2b_p+3b_y)S_t+2b_p\cdot\1_{t\ge2}$ vector calls: $2b_pS_t$ from the banks $A_k$, $3b_yS_t$ from the banks $B_k$ and $2b_p$ from the
reserve mini-batch. Every call is used for the constraint vector, and the $2b_pS_t$ calls of the banks $A_k$ (whose pairs deliver both $F$ and $G$) are also
used for the objective; therefore
\begin{equation}\label{eq:sliding-counts}
 \Nseq=N,\qquad \Norc=\Ng=(2b_p+3b_y)K_N+2b_p(N-1),\qquad \Nf=2b_pK_N,\qquad \Nmv=K_N .
\end{equation}
With a common size $b_p=b_y=b$ this reads $\Norc=\Ng=5bK_N+2b(N-1)$ and $\Nf=2bK_N$.
All query locations of round $t$ ($u_t$ and $u_{t-1}$) are determined before the round; only the \emph{assignment} of the already sampled mini-batches to
the inner steps depends on the inner iterates. The centers $u_t$ are convex combinations of the previous phase outputs (\Cref{lem:centers}), so all queries are
located in $\Q_{\gamma}$.

\subsection{The deterministic skeleton}\label{sec:skeleton}
Throughout, $J_t:=\Jg(u_t)$, and for $(x,y)\in\R^d\times\R^m$
\begin{equation}\label{eq:linearization}
 \ell_t(x,y):=f_\gamma(u_t)+\inner{\nabla f_\gamma(u_t)}{x-u_t}+\inner{y}{h(u_t)+J_t(x-u_t)}
\end{equation}
is the linearization of $\LagS(\cdot,y)$ at the center; by convexity of $f_\gamma$ and $h_i$, $\ell_t(x,y)\le\LagS(x,y)$ for $y\ge0$. The exact (noise-free)
inner directions are
\begin{equation}\label{eq:exactdirections}
 r_k:=\nabla f_\gamma(u_t)+J_t^\top y_{k-1}+\rho_kJ_{t'(k)}^\top(y_{k-1}-y_{k-2}),\qquad q_k:=h(u_t)+J_t(p_k-u_t),
\end{equation}
where, for a slot $k$ of phase $t$ that is not the first one, $\rho_k:=1$ and $t'(k):=t$; for the first slot of phase $t\ge2$, $\rho_k:=a_{t-1}/a_t$ and
$t'(k):=t-1$; and for $k=1$ the last term vanishes because $y_0=y_{-1}$. The algorithm uses $\hat r_k=r_k+e^p_k$ and $\hat q_k=q_k+e^y_k$ with the estimation errors
\begin{equation}\label{eq:errors}
 e^p_k:=\hat r_k-r_k,\qquad e^y_k:=\hat q_k-q_k .
\end{equation}

\begin{lemma}[schedule]\label{lem:schedule}
Let $c:=M_g\rs/L$. For all $t\ge1$: (i) $a_tM_g\rs\le L$; (ii) $a_{t-1}/a_t\le2$ for $t\ge2$; (iii) $\max\{N,\frac c2N(N+1)\}\le K_N\le N+\frac c2N(N+1)$;
(iv) $W_N^2/K_N\le A_N\le\frac53\,W_N^2/K_N$.
\end{lemma}

\begin{lemma}[centers]\label{lem:centers}
For arbitrary $x_{-1}=x_0,x_1,\dots$ and $u_t$ defined by the recursion of \Cref{alg:sliding}, $u_1=x_0$ and for $t\ge2$
\begin{equation}\label{eq:barycenter}
 u_t=\frac1{t(t+1)}\Bigl(\sum_{j=1}^{t-2}2j\,x_j+(4t-2)\,x_{t-1}\Bigr),
\end{equation}
a convex combination of $x_1,\dots,x_{t-1}$. Moreover, with $d_t:=x_t-x_{t-1}$,
\begin{equation}\label{eq:xminusu}
 x_t-u_t=d_t-\theta_td_{t-1}+\tau_t(u_t-u_{t-1}).
\end{equation}
\end{lemma}

The next lemma is the analytic heart of the acceleration. It is a statement about the extrapolation rule alone, valid for \emph{arbitrary} points
$x_t$: whatever the inner loop produces, the weighted linearizations at the centers control the function value at the weighted average up to a
quadratic penalty on the increments. It is a reformulation, with explicit constants, of the conjugate-based argument of \citet{zhanglan2022}.

\begin{lemma}[sliding inequality]\label{lem:outer}
Let $\mathsf F:\R^d\to\R$ be convex and $L$-smooth, let $x_{-1}=x_0,x_1,\dots,x_N\in\R^d$ be arbitrary, let $u_t$ be given by the recursion of
\Cref{alg:sliding} with $\tau_t=(t-1)/2$, $\theta_t=(t-1)/t$, let $\ell^{\mathsf F}_t(x):=\mathsf F(u_t)+\inner{\nabla\mathsf F(u_t)}{x-u_t}$ and $\Xr_N:=W_N^{-1}\sum_{t=1}^Ntx_t$. Then
\begin{equation}\label{eq:outer}
 W_N\,\mathsf F(\Xr_N)\le\sum_{t=1}^Nt\,\ell^{\mathsf F}_t(x_t)+L\sum_{t=1}^N\norm{x_t-x_{t-1}}^2 .
\end{equation}
\end{lemma}

\begin{lemma}[energy inequality of the inner loop]\label{lem:inner}
Let \Cref{ass:convex,ass:lip,ass:slater} hold, $L\ge H$, and let the iterates be generated by \Cref{alg:sliding} with the directions
$\hat r_k=r_k+e^p_k$, $\hat q_k=q_k+e^y_k$ for arbitrary vectors $e^p_k\in\R^d$, $e^y_k\in\R^m$. Then, pathwise, for every $y\in Y_\Rmult$,
\begin{multline}\label{eq:energy}
 \sum_{t=1}^Nt\bigl[\ell_t(x_t,y)-\ell_t(x^\ast_\gamma,\bar y_t)\bigr]+L\sum_{t=1}^N\norm{x_t-x_{t-1}}^2
 +\frac\lambda2\sum_{k=1}^{K_N}\Bigl[\norm{p_k-p_{k-1}}^2+\frac{\norm{y_k-y_{k-1}}_1^2}{\rs^2}\Bigr]\\
 \le(L+\lambda)\dA^2-\sum_{k=1}^{K_N}a_k\inner{e^p_k}{p_k-x^\ast_\gamma}+\sum_{k=1}^{K_N}a_k\inner{e^y_k}{y_k-y},
\end{multline}
where $a_k:=a_{t(k)}$ and
\begin{equation}\label{eq:DA}
 \dA^2:=3\cA+\frac{\cB}{\rs^2},\qquad \cA=\tfrac12\norm{x_0-x^\ast_\gamma}^2,\quad \cB=\Rmult^2\log(m+1).
\end{equation}
\end{lemma}
The proofs of \Cref{lem:schedule,lem:centers,lem:outer,lem:inner} are given in \Cref{app:sliding}. The proof of \Cref{lem:inner} sums three-point inequalities,
and three groups of terms telescope. (i) The prox terms with coefficient $\eta_t$ telescope over phases and give the term $-L\sum\norm{d_t}^2$ that cancels the
penalty in \eqref{eq:outer}. (ii) The prox terms with coefficient $\beta_t$ telescope over slots, because $a_k\beta_k=L+\lambda$ is constant. (iii) The bilinear
cross terms from the dual extrapolation telescope across phase boundaries; this uses the reserve mini-batch and the ratio $a_{t-1}/a_t$.
What remains is absorbed by Young's inequality with $a_tM_g\rs\le L$.

\subsection{Main result}
\begin{theorem}[\ZOS]\label{thm:sliding}
Let \Cref{ass:convex,ass:lip,ass:levels,ass:slater} hold, $\gamma\le\bar\gamma$, $\Rmult$ as in \eqref{eq:R}, $L\ge H$ as in \eqref{eq:HL}, and let the starting
point $x_0\in\Q$ be deterministic or independent of the oracle samples. For every $\lambda>0$, $\rs>0$, $b_p,b_y\ge1$ and $N\ge1$ the output of \Cref{alg:sliding} satisfies
\begin{equation}\label{eq:sliding-main}
 W_N\,\E\,\Cert_{\Rmult,\gamma}(\Xr_N)\le(L+2\lambda)\,\E\dA^2+\frac{A_N\Sb^2}{\lambda},
\end{equation}
where
\begin{equation}\label{eq:SigmaA}
 \Sb^2:=\frac{17\,dM_R^2}{b_p}+\frac{\rs^2\sigma_y^2}{b_y},\qquad \sigma_y^2:=16\kap\bigl(\sigma_h^2+\nuj M_g^2D^2\bigr),
\end{equation}
with $\sigma_h$ from \eqref{eq:sigmah} and $\nuj=\min\{4d,92\kap\}$ from \Cref{lem:action}. The first term of $\Sb^2$ is the noise of the primal banks
(gradient information), the second that of the dual banks (level and linearization information); with a common batch size $b_p=b_y=b$, $\Sb^2=\SA^2/b$ with
$\SA^2:=17dM_R^2+\rs^2\sigma_y^2$.
\end{theorem}

\begin{proof}
\emph{Step 1 (filtration and centering).} Order the mini-batches of a round in the order in which the inner loop consumes them: $A^-_t$ (if $t\ge2$),
then $A_k,B_k$ for $k=K_{t-1}+1,\dots,K_t$. Let $\cF^A_k$ be the $\sigma$-field generated by $x_0$ and by all mini-batches consumed up to and including $A_k$,
and $\cF^B_k$ the one that also includes $B_k$; $\cF^B_0$ is generated by $x_0$. Then $p_{k-1},y_{k-1},y_{k-2}$ and $u_t,u_{t-1}$ are $\cF^B_{k-1}$-measurable, the
batches $A_k$ and $A^-_t$ are independent of $\cF^B_{k-1}$, $p_k$ is $\cF^A_k$-measurable and $B_k$ is independent of $\cF^A_k$. By
\Cref{lem:action}(a) and \Cref{lem:levels}, applied conditionally with the (measurable) weights $y_{k-1}$, $\rho_k(y_{k-1}-y_{k-2})$ and $w_k:=p_k-u_t$,
\begin{equation}\label{eq:centered}
 \E[e^p_k\mid\cF^B_{k-1}]=0,\qquad \E[e^y_k\mid\cF^A_k]=0 .
\end{equation}

\emph{Step 2 (second moments).} For a slot $k$ that is not the first of its phase, $e^p_k$ is the average over the $b$ pairs of $A_k$ of centered estimation
errors of the Lagrangian gradient with weight vector $2y_{k-1}-y_{k-2}$, whose $\ell_1$-norm is at most $3\Rmult$; by \Cref{lem:action}(a),
$\E[\norm{e^p_k}^2\mid\cF^B_{k-1}]\le2d(M_0+3\Rmult M_g)^2/b_p\le18dM_R^2/b_p$. For the first slot of a phase $t\ge2$, $e^p_k$ is the sum of two conditionally
independent centered terms, from $A_k$ (weights $y_{k-1}$, $\ell_1$-norm $\le\Rmult$) and from $A^-_t$ (weights $\rho_k(y_{k-1}-y_{k-2})$, $\ell_1$-norm $\le2\cdot2\Rmult$ by
\Cref{lem:schedule}(ii)), so $\E[\norm{e^p_k}^2\mid\cF^B_{k-1}]\le[2dM_R^2+2d(4\Rmult M_g)^2]/b_p\le34dM_R^2/b_p$. Hence in all cases
\begin{equation}\label{eq:sigp}
 \E\bigl[\norm{e^p_k}^2\mid\cF^B_{k-1}\bigr]\le\frac{34\,dM_R^2}{b_p}.
\end{equation}
For the dual error, $e^y_k=[\hat h^{B_k}(u_t)-h(u_t)]+[\hat J^{B_k}(u_t)-J_t]w_k$ with $\norm{w_k}=\norm{p_k-u_t}\le D$ (both $p_k$ and $u_t$ lie in $\Q$ by
\Cref{lem:centers}); both brackets are averages of $b_y$ conditionally
i.i.d.\ centered vectors with $\ell_\infty$ second moments bounded by $\sigma_h^2$ (\Cref{lem:levels}) and $\nuj M_g^2D^2$ (\Cref{lem:action}(c),(d)) respectively.
By \Cref{lem:batch} and $\norm{a+c}_\infty^2\le2\norm a_\infty^2+2\norm c_\infty^2$,
\begin{equation}\label{eq:sigy}
 \E\bigl[\norm{e^y_k}_\infty^2\mid\cF^A_k\bigr]\le\frac{16\kap}{b_y}\bigl(\sigma_h^2+\nuj M_g^2D^2\bigr)=\frac{\sigma_y^2}{b_y}.
\end{equation}

\emph{Step 3 (splitting the error terms).} Write $p_k-x^\ast_\gamma=(p_k-p_{k-1})+(p_{k-1}-x^\ast_\gamma)$ and $y_k-y=(y_k-y_{k-1})+(y_{k-1}-y)$ in the right-hand
side of \eqref{eq:energy}. By Young's inequality,
\begin{gather*}
 -a_k\inner{e^p_k}{p_k-p_{k-1}}\le\frac{a_k^2}{2\lambda}\norm{e^p_k}^2+\frac\lambda2\norm{p_k-p_{k-1}}^2,\\
 a_k\inner{e^y_k}{y_k-y_{k-1}}\le\frac{a_k^2\rs^2}{2\lambda}\norm{e^y_k}_\infty^2+\frac{\lambda}{2\rs^2}\norm{y_k-y_{k-1}}_1^2 ,
\end{gather*}
and the two quadratic increments are cancelled by the left-hand side of \eqref{eq:energy}. Consequently, pathwise for all $y\in Y_\Rmult$,
\begin{multline}\label{eq:energy2}
 \sum_tt\bigl[\ell_t(x_t,y)-\ell_t(x^\ast_\gamma,\bar y_t)\bigr]+L\sum_t\norm{d_t}^2\\
 \le(L+\lambda)\dA^2+\sum_k\frac{a_k^2}{2\lambda}\bigl[\norm{e^p_k}^2+\rs^2\norm{e^y_k}_\infty^2\bigr]
 -\sum_ka_k\inner{e^p_k}{p_{k-1}-x^\ast_\gamma}+\sum_ka_k\inner{e^y_k}{y_{k-1}-y}.
\end{multline}

\emph{Step 4 (from linearizations to the certificate).} Fix $y\in Y_\Rmult$ and apply \Cref{lem:outer} to $\mathsf F:=\LagS(\cdot,y)$, which is convex and
$L$-smooth on all of $\R^d$ by \Cref{conv:extension} and \eqref{eq:HL} (this is the only place where a property outside $\Q_{\bar\gamma}$ is used; the
points $u_t,x_t,\Xr_N$ at which the inequality is evaluated lie in $\Q$), with $\ell^{\mathsf F}_t=\ell_t(\cdot,y)$: $W_N\LagS(\Xr_N,y)\le\sum_tt\,\ell_t(x_t,y)+L\sum_t\norm{d_t}^2$. Moreover
$\ell_t(x^\ast_\gamma,\bar y_t)\le\LagS(x^\ast_\gamma,\bar y_t)\le f^\ast_\gamma$ by convexity and \eqref{eq:restrictedgap}, so $\sum_tt\,\ell_t(x^\ast_\gamma,\bar y_t)\le W_Nf^\ast_\gamma$.
Together with \eqref{eq:energy2},
\begin{multline*}
 W_N\bigl[\LagS(\Xr_N,y)-f^\ast_\gamma\bigr]\le(L+\lambda)\dA^2+\sum_k\frac{a_k^2}{2\lambda}\bigl[\norm{e^p_k}^2+\rs^2\norm{e^y_k}_\infty^2\bigr]\\
 -\sum_ka_k\inner{e^p_k}{p_{k-1}-x^\ast_\gamma}+\sum_ka_k\inner{e^y_k}{y_{k-1}-y}
\end{multline*}
for every $y\in Y_\Rmult$. Taking the supremum over $y$ on both sides and recalling \eqref{eq:cert},
\begin{multline}\label{eq:pathwise-cert}
 W_N\,\Cert_{\Rmult,\gamma}(\Xr_N)\le(L+\lambda)\dA^2+\sum_k\frac{a_k^2}{2\lambda}\bigl[\norm{e^p_k}^2+\rs^2\norm{e^y_k}_\infty^2\bigr]\\
 -\sum_ka_k\inner{e^p_k}{p_{k-1}-x^\ast_\gamma}+\sup_{y\in Y_\Rmult}\sum_ka_k\inner{e^y_k}{y_{k-1}-y}.
\end{multline}

\emph{Step 5 (expectations).} By \eqref{eq:centered} and the measurability of $p_{k-1}$, $\E\inner{e^p_k}{p_{k-1}-x^\ast_\gamma}=0$ (all quantities are bounded, since
$\Q$ is compact and the second moments are finite). Next we apply \Cref{lem:martsup} with
\begin{itemize}
\item $Z=\tilde Y_\Rmult$, $\omega=\Rmult\sum_i\tilde y_i\log\tilde y_i$, the norm $\norm\cdot_1$ (dual norm $\norm\cdot_\infty$), $\Theta=\cB$, $z_0=\tilde y^{\mathrm u}$;
\item the filtration $(\cF^B_k)_k$, the points $\zeta_k=\tilde y_{k-1}$ and the vectors $e_k:=-(0,e^y_k)$, so that $a_k\inner{e_k}{\tilde y-\zeta_k}=a_k\inner{e^y_k}{y_{k-1}-y}$.
\end{itemize}
The lemma applies because $e^y_k$ is $\cF^B_k$-measurable and $\E[e^y_k\mid\cF^B_{k-1}]=0$; the latter follows from \eqref{eq:centered} and the tower property,
since $\cF^B_{k-1}\subseteq\cF^A_k$. We obtain, for every $\lambda'>0$,
\[
 \E\sup_{y\in Y_\Rmult}\sum_ka_k\inner{e^y_k}{y_{k-1}-y}\le\lambda'\cB+\frac1{2\lambda'}\sum_ka_k^2\,\E\norm{e^y_k}_\infty^2 .
\]
Choosing $\lambda'=\lambda/\rs^2$ and using \eqref{eq:sigp}, \eqref{eq:sigy} and $\sum_ka_k^2=\sum_tS_ta_t^2=A_N$ in \eqref{eq:pathwise-cert},
\[
 W_N\,\E\Cert_{\Rmult,\gamma}(\Xr_N)\le(L+\lambda)\E\dA^2+\frac{\lambda\cB}{\rs^2}+\frac{A_N}{2\lambda}\Bigl(\frac{34dM_R^2}{b_p}+\frac{\rs^2\sigma_y^2}{b_y}\Bigr)+\frac{A_N\rs^2\sigma_y^2}{2\lambda b_y}.
\]
Since $\cB/\rs^2\le\dA^2$, the right-hand side is at most $(L+2\lambda)\E\dA^2+A_N\Sb^2/\lambda$, which is \eqref{eq:sliding-main}.
(When $x_0$ is random, the whole argument is carried out conditionally on $x_0$ and then integrated.)
\end{proof}

\subsection{Choice of the parameters and complexity}
\begin{corollary}[optimized stabilizer]\label{cor:sliding-lambda}
In the setting of \Cref{thm:sliding} with deterministic $x_0$, the choice $\lambda:=\Sb\sqrt{A_N/2}/\dA$ gives
\begin{equation}\label{eq:sliding-rate}
 \E\,\Cert_{\Rmult,\gamma}(\Xr_N)\le\frac{2L\dA^2}{N(N+1)}+\frac{2}{W_N}\sqrt{2A_N\Sb^2\dA^2}
 \le\frac{2L\dA^2}{N(N+1)}+4\,\dA\sqrt{\frac{17dM_R^2}{b_pK_N}+\frac{\rs^2\sigma_y^2}{b_yK_N}} .
\end{equation}
\end{corollary}
\begin{proof}
Minimizing $2\lambda\dA^2+A_N\Sb^2/\lambda$ over $\lambda>0$ gives the stated $\lambda$ and the value $2\sqrt{2A_N\Sb^2\dA^2}$; divide \eqref{eq:sliding-main} by $W_N$.
By \Cref{lem:schedule}(iv), $A_N\le\frac53W_N^2/K_N$, and $2\sqrt{10/3}\le4$.
\end{proof}

The bound \eqref{eq:sliding-rate} has the structure of the optimal rate for smooth stochastic convex optimization with mini-batches,
$L\dA^2/N^2+\Sigma\dA/\sqrt{\text{number of samples}}$, in which $b_pK_N$ and $b_yK_N$ are the numbers of primal and dual mini-batches consumed and $N$
the number of rounds; with a common batch size the second term is $4\SA\dA/\sqrt{bK_N}$.

\begin{corollary}[complexity of \ZOS]\label{cor:sliding-complexity}
In the setting of \Cref{cor:sliding-lambda}, let $\eps_{\mathrm s}>0$ with $\eps_{\mathrm s}\le L\dA^2$, and choose
\begin{equation}\label{eq:sliding-params}
 N:=\Bigl\lceil2\dA\sqrt{L/\eps_{\mathrm s}}\Bigr\rceil,\qquad b_p:=\max\Bigl\{1,\Bigl\lceil\frac{2176\,dM_R^2\dA^2}{K_N\eps_{\mathrm s}^2}\Bigr\rceil\Bigr\},\qquad
 b_y:=\max\Bigl\{1,\Bigl\lceil\frac{128\,\rs^2\sigma_y^2\dA^2}{K_N\eps_{\mathrm s}^2}\Bigr\rceil\Bigr\}.
\end{equation}
Then $\E\Cert_{\Rmult,\gamma}(\Xr_N)\le\eps_{\mathrm s}$ and
\begin{equation}\label{eq:sliding-counts2}
\begin{gathered}
 \Nseq=N\le2\dA\sqrt{\frac L{\eps_{\mathrm s}}}+1,\qquad \Nmv=K_N\le N+\frac{6M_g\rs\dA^2}{\eps_{\mathrm s}},\\
 \Nf\le2K_N+\frac{4352\,dM_R^2\dA^2}{\eps_{\mathrm s}^2},\qquad \Norc=\Ng\le7K_N+\frac{8704\,dM_R^2\dA^2}{\eps_{\mathrm s}^2}+\frac{384\,\rs^2\sigma_y^2\dA^2}{\eps_{\mathrm s}^2}.
\end{gathered}
\end{equation}
The objective calls thus pay only for the gradient information ($dM_R^2$), the constraint calls for the gradient and for the level information.
\end{corollary}
\begin{proof}
By \eqref{eq:sliding-rate}, $2L\dA^2/(N(N+1))\le2L\dA^2/N^2\le\eps_{\mathrm s}/2$, and the second term of \eqref{eq:sliding-rate} is at most $\eps_{\mathrm s}/2$ because
$17dM_R^2/(b_pK_N)\le\eps_{\mathrm s}^2/(128\dA^2)$ and $\rs^2\sigma_y^2/(b_yK_N)\le\eps_{\mathrm s}^2/(128\dA^2)$, so that the square root is at most $\eps_{\mathrm s}/(8\dA)$.
Put $\vartheta:=\dA\sqrt{L/\eps_{\mathrm s}}\ge1$; then $N\le2\vartheta+1\le3\vartheta$ and $N+1\le4\vartheta$, so $N(N+1)\le12\vartheta^2=12L\dA^2/\eps_{\mathrm s}$ and
\Cref{lem:schedule}(iii) gives $K_N\le N+\frac{M_g\rs}{2L}N(N+1)\le N+6M_g\rs\dA^2/\eps_{\mathrm s}$. Finally $b_p\le1+2176dM_R^2\dA^2/(K_N\eps_{\mathrm s}^2)$,
$b_y\le1+128\rs^2\sigma_y^2\dA^2/(K_N\eps_{\mathrm s}^2)$, $\Nf=2b_pK_N$ and $\Norc=\Ng=(2b_p+3b_y)K_N+2b_p(N-1)\le4b_pK_N+3b_yK_N$ by \eqref{eq:sliding-counts} and $N\le K_N$.
\end{proof}

The dual scale $\rs$ trades the size of $\dA^2=3\cA+\cB/\rs^2$ against the dual noise $\rs^2\sigma_y^2$ and the number of inner steps.
A choice that depends only on known quantities and balances the two prox radii is $\rs^2=2\cB/D^2$:

\begin{corollary}[explicit rates for $\rs^2=2\cB/D^2$]\label{cor:sliding-explicit}
Choose $\rs:=\sqrt{2\cB}/D$ in \Cref{cor:sliding-complexity}. Then $\dA^2\le2D^2$ and, with $\Lam=\kap\log(m+1)$,
\begin{equation}\label{eq:explicit-constants}
\begin{gathered}
 N\le2D\sqrt{\frac{2L}{\eps_{\mathrm s}}}+1,\qquad K_N\le N+\frac{17\,M_g\Rmult D\sqrt{\log(m+1)}}{\eps_{\mathrm s}},\qquad
 \Nf\le2K_N+\frac{8704\,dM_R^2D^2}{\eps_{\mathrm s}^2},\\
 \Norc=\Ng\le7K_N+\frac{17408\,dM_R^2D^2+24576\,\Lam\Rmult^2\bigl(\sigma_h^2+\nuj M_g^2D^2\bigr)}{\eps_{\mathrm s}^2}.
\end{gathered}
\end{equation}
Consequently $\Nseq=O(\sqrt{LD^2/\eps_{\mathrm s}})$, $\Nmv=O(\Nseq+M_g\Rmult D\sqrt{\log(m+1)}/\eps_{\mathrm s})$ and
\[
 \Nf=O\Bigl(\Nmv+\frac{dM_R^2D^2}{\eps_{\mathrm s}^2}\Bigr),\qquad
 \Norc=\Ng=O\Bigl(\Nmv+\frac{dM_R^2D^2+\Lam\Rmult^2(\sigma_h^2+\nuj M_g^2D^2)}{\eps_{\mathrm s}^2}\Bigr);
\]
since $\nuj\le4d$ and $\Rmult M_g\le M_R$, the last expression is in turn $O\bigl(\Nmv+\Lam(dM_R^2D^2+\Rmult^2\sigma_h^2)/\eps_{\mathrm s}^2\bigr)$, a cruder form in
which the factor $\Lam$ also multiplies the primal term. The term $\Lam\Rmult^2\nuj M_g^2D^2$ is dimension-free ($\nuj\le92\kap$) and is dominated by $dM_R^2D^2$
as soon as $d\ge92\kap^2\log(m+1)$.
\end{corollary}
\begin{proof}
$\dA^2=3\cA+\cB/\rs^2\le\frac32D^2+\frac12D^2=2D^2$, so $dM_R^2\dA^2\le2dM_R^2D^2$ and $\rs^2\sigma_y^2\dA^2\le(2\cB/D^2)\cdot16\kap(\sigma_h^2+\nuj M_g^2D^2)\cdot2D^2
=64\Lam\Rmult^2(\sigma_h^2+\nuj M_g^2D^2)$; insert this into \eqref{eq:sliding-counts2}. $6M_g\rs\dA^2\le6M_g(\sqrt{2\cB}/D)(2D^2)=12\sqrt2M_g\Rmult D\sqrt{\log(m+1)}$ and $12\sqrt2\le17$.
\end{proof}

\begin{corollary}[smooth data]\label{cor:sliding-smooth}
Let \Cref{ass:convex,ass:lip,ass:levels,ass:slater,ass:smooth} hold and $\eps\in(0,\min\{2\max\{L_0,L_g\}\bar\gamma^2,\ (L_0+\Rmult L_g)D^2\}]$. Take
$\gamma:=\min\{\sqrt{\eps/(2\max\{L_0,L_g\})},\ \eps/(4M_g)\}$, $b_0=L_0\gamma^2/2$, $b_i=L_g\gamma^2/2$,
$L:=L_0+\Rmult L_g$ and run \Cref{alg:sliding} with the parameters of \Cref{cor:sliding-explicit} for $\eps_{\mathrm s}=\eps/2$. Then $\Xr_N$ is an $\eps$-solution
in expectation and
\begin{gather*}
 \Nseq\le4D\sqrt{\frac{L_0+\Rmult L_g}{\eps}}+1,\qquad \Nmv=O\Bigl(\Nseq+\frac{M_g\Rmult D\sqrt{\log(m+1)}}{\eps}\Bigr),\\
 \Nf=O\Bigl(\Nmv+\frac{dM_R^2D^2}{\eps^2}\Bigr),\qquad \Norc=\Ng=O\Bigl(\Nmv+\frac{dM_R^2D^2+\Lam\Rmult^2(\sigma_g^2+\nuj M_g^2D^2)}{\eps^2}\Bigr).
\end{gather*}
\end{corollary}
\begin{proof}
By \Cref{lem:smoothing}(c), $L_{0,\gamma}=L_0$ and $L_{g,\gamma}=L_g$ for every $\gamma>0$, so $L=H$; the second moment of the two-point estimators does not
depend on $\gamma$ (\Cref{lem:twopoint}), so the radius may be taken as small as convenient. The biases satisfy $b_0,b_g\le\max\{L_0,L_g\}\gamma^2/2\le\eps/4$; by
\Cref{lem:transfer}(c) and \Cref{lem:certificate}, $\E\Cert\le\eps/2$ gives $\E[f-f^\ast]\le3\eps/4$ and $\E\viol\le3\eps/4$. Since $\gamma\le\eps/(4M_g)$,
$8\gamma^2M_g^2\le\eps^2/2$ and $\sigma_h^2\le2\sigma_g^2+\eps^2/2$, so that the contribution of the smoothing part of $\sigma_h^2$ to $\Norc$ is $O(\Lam\Rmult^2)$, a
constant. The hypothesis $\eps\le LD^2$ guarantees $\eps_{\mathrm s}=\eps/2\le L\dA^2$, because $\dA^2\ge\cB/\rs^2=D^2/2$. Insert $\eps_{\mathrm s}=\eps/2$ into
\eqref{eq:explicit-constants}.
\end{proof}

\begin{corollary}[nonsmooth data]\label{cor:sliding-nonsmooth}
Let \Cref{ass:convex,ass:lip,ass:levels,ass:slater} hold and $\eps\in(0,\min\{4M\bar\gamma,\ 2d^{1/4}\sqrt{M_RM}D\}]$. Take $\gamma:=\eps/(4M)$, $b_0=\gamma M_0$, $b_i=\gamma M_g$,
$L:=\sqrt dM_R/\gamma=4\sqrt dM_RM/\eps$, and
run \Cref{alg:sliding} with the parameters of \Cref{cor:sliding-explicit} for $\eps_{\mathrm s}=\eps/2$. Then $\Xr_N$ is an $\eps$-solution in expectation and
\begin{gather*}
 \Nseq\le\frac{8\,d^{1/4}\sqrt{M_RM}\,D}{\eps}+1,\qquad \Nmv=O\Bigl(\Nseq+\frac{M_g\Rmult D\sqrt{\log(m+1)}}{\eps}\Bigr),\\
 \Nf=O\Bigl(\Nmv+\frac{dM_R^2D^2}{\eps^2}\Bigr),\qquad \Norc=\Ng=O\Bigl(\Nmv+\frac{dM_R^2D^2+\Lam\Rmult^2(\sigma_g^2+\nuj M_g^2D^2)}{\eps^2}\Bigr).
\end{gather*}
\end{corollary}
\begin{proof}
$H=L_{0,\gamma}+\Rmult L_{g,\gamma}=\sqrt d(M_0+\Rmult M_g)/\gamma=L$ by \Cref{lem:smoothing}(c) and \Cref{conv:extension}. The biases are $b_0\le\gamma M=\eps/4$ and $b_g\le\eps/4$, and
$N\le2D\sqrt{2L/\eps_{\mathrm s}}+1=2D\sqrt{16\sqrt dM_RM/\eps^2}+1=8d^{1/4}\sqrt{M_RM}D/\eps+1$. Also $8\gamma^2M_g^2\le\eps^2/2$, so $\sigma_h^2\le2\sigma_g^2+\eps^2/2$ and the
contribution of $\eps^2/2$ to $\Norc$ is $O(\Lam\Rmult^2)$, a constant. The hypothesis $\eps\le2d^{1/4}\sqrt{M_RM}D$ is equivalent to $\eps/2\le L D^2/2$ and
guarantees $\eps_{\mathrm s}\le L\dA^2$ because $\dA^2\ge D^2/2$.
\end{proof}

\begin{remark}[separate smoothing radii]\label{rem:two-radii}
The objective and the constraints can be smoothed with different radii. Smooth $f$ with $\gamma_0:=\eps/(4M_0)$ and every $g_i$ with $\gamma_g:=\eps/(4M_g)$,
let each estimator in \eqref{eq:estimators} use its own radius, and apply \Cref{lem:levels} with $\gamma_g$. The biases are still $b_0,b_g\le\eps/4$, and the
smoothness constant of the Lagrangian becomes
\[
 H=\frac{\sqrt dM_0}{\gamma_0}+\frac{\Rmult\sqrt dM_g}{\gamma_g}=\frac{4\sqrt d\,(M_0^2+\Rmult M_g^2)}{\eps}\le\frac{4\sqrt d\,M_RM}{\eps},
\]
since $M_0^2\le M_0M$ and $\Rmult M_g^2\le\Rmult M_gM$. All statements of \Cref{sec:smoothing,sec:lagrangian} are per-function and remain valid, so
\Cref{cor:sliding-nonsmooth} holds with $\Nseq\le8d^{1/4}\sqrt{M_0^2+\Rmult M_g^2}\,D/\eps+1$; when $M_0\gg\sqrt{\Rmult}M_g$ this is within a constant of the
unconstrained depth $d^{1/4}M_0D/\eps$. We keep a common radius in the main text to lighten the notation.
\end{remark}

The round complexities of \Cref{cor:sliding-smooth,cor:sliding-nonsmooth} coincide, up to the multiplier factor $\Rmult$ in $M_R$ and $H$, with the
iteration complexities of the accelerated batched methods for \emph{unconstrained} zeroth-order problems \citep{gasnikov2022power}:
$\sqrt{L_0D^2/\eps}$ for smooth and $d^{1/4}MD/\eps$ for nonsmooth objectives. These are the best known upper bounds for the unconstrained two-point
problem; they are minimax optimal in the smooth case, while in the nonsmooth case the known depth lower bounds are weaker (see \Cref{sec:depth}).

\subsection{Discussion}\label{sec:sliding-discussion}
\paragraph{What is paid for the depth reduction.} Compared with the Mirror--Prox baseline, whose dual noise is only the level noise $\sigma_h^2$, the sliding
method has the additional dual noise term $\nuj M_g^2D^2$ in \eqref{eq:SigmaA}. It comes from the estimate $\hat J^{B_k}(u_t)(p_k-u_t)$ of the linearized
constraint value at the inner point $p_k$: the exact value $h(p_k)$ could be estimated by a single query at $p_k$, but $p_k$ is not known before the round,
and querying it would create a new sequential round. The linearization at the center replaces an \emph{adaptive} query by a matrix--vector product with a
fresh Jacobian sample; the variance of this product is $\nuj M_g^2\norm{p_k-u_t}^2$. By \Cref{lem:action}(d), $\nuj\le92\kap$, so the extra term is dimension-free
and is dominated by the primal term $dM_R^2D^2$ as soon as $d\gtrsim\kap^2\log(m+1)$; in low dimensions the crude bound $\nuj\le4d$ is the better one, and the extra
term is then of the same order $dM_R^2D^2$ as the primal one, up to the factor $\Lam$.

\paragraph{Why the mini-batches must be fresh.} If the same Jacobian estimate were reused within a phase, the inner dual iterates would be correlated with the
estimation error and the term $\inner{\bar y_t}{(\hat J_t-J_t)(x^\ast_\gamma-u_t)}$ in the gap could only be bounded by $\Rmult\,\E\norm{(\hat J_t-J_t)(x^\ast_\gamma-u_t)}_\infty
=O(\Rmult M_gD\sqrt{\nuj/b_y})$, a bias that does not average out over the phases. The fresh mini-batches make the errors martingale differences with respect to the
consumption order, which is exactly what \eqref{eq:centered} and \Cref{lem:martsup} require. This is the reason why the counts $\Norc,\Nf,\Ng$ in
\eqref{eq:sliding-counts} are proportional to $K_N$ rather than to $N$; the depth $\Nseq=N$ is unaffected.

\paragraph{Inner steps, calls and query points.} Unlike the inner steps of ACGD-S, the inner steps of \Cref{alg:sliding} are not free in oracle calls: each of them
consumes one primal and one dual mini-batch that were drawn in the current round. What they save is \emph{rounds}: $\Nmv=K_N$ counts operations that need
no new oracle round, and the $O(\eps^{-2})$ calls of the method are located at only $N=O(\eps^{-1/2})$ distinct centers (plus the previous center for the
reserve mini-batch). For simulation oracles this concentration of the queries is an advantage of its own (state of the simulator, common random numbers)
which is not reflected in the three complexity measures. In the affine case of \Cref{sec:affine} the inner steps are oracle-free in the literal sense.

\paragraph{Relation to ACGD-S.} With exact gradients ($e^p_k=e^y_k=0$, $\lambda=0$) \Cref{lem:outer,lem:inner} give
$\Cert_{\Rmult,\gamma}(\Xr_N)\le2L\dA^2/(N(N+1))$, which is the deterministic guarantee of ACGD-S \citep[Thm.~5]{zhanglan2022} in our notation, with
$K_N=O(N+M_g\rs N^2/L)$ inner matrix--vector products. Our schedule \eqref{eq:schedule} is theirs ($\tau_t$, $\theta_t$, $\eta_t=L/\tau_{t+1}$, weights $\omega_t=t$).
The new ingredients are the stochastic zeroth-order estimators, the fresh-mini-batch design, the stabilizer $\lambda$, the explicit treatment of the dual noise
through \Cref{lem:batch,lem:martsup}, and the transfer to the original constrained problem through the smoothing shifts.

\paragraph{Parameters.} The method needs upper bounds $L$, $M_g$ and $\Rmult$ and the constants entering $\Sb$. The stabilizer of \Cref{cor:sliding-lambda} is the
worst-case choice. In the experiments of \Cref{sec:experiments} a much smaller one works better.

\paragraph{Two batch sizes.} The primal banks $A_k$ and the reserve batch carry the gradient information about the Lagrangian. Their noise, $17dM_R^2/b_p$, contains the
dimension factor $d$. The dual banks $B_k$ carry the level and linearization information. Their noise, $\rs^2\sigma_y^2/b_y$, is dimension-free but does not
shrink with the accuracy in the strongly convex case. The objective is used only in the banks $A_k$, so $\Nf=2b_pK_N$, and the two batch sizes can be tuned to the
two error sources separately (\Cref{cor:sliding-complexity}). As a result, in \Cref{thm:restart-sliding} the objective calls $\Nf$ of \Cref{def:measures} scale as $1/(\mu\eps)$, as in
strongly convex optimization, while $\Ng$ keeps the $\eps^{-2}$ level term. The experiments use a common size $b_p=b_y$.

\section{Strongly convex objectives: restarts}\label{sec:restarts}

Under \Cref{ass:sc} the certificate controls the distance to the solution of the smoothed problem, $\frac\mu2\norm{x-x^\ast_\gamma}^2\le\Cert_{\Rmult,\gamma}(x)$
(\Cref{lem:certificate}). Both methods of the previous sections have guarantees that depend on the starting point only through
$\cA=\frac12\norm{x_0-x^\ast_\gamma}^2$ and that are nondecreasing and concave in $\cA$. This is exactly what is needed for the classical restart
technique: halving the target accuracy at each stage halves the admissible bound on $\cA$, so that the number of rounds per stage stays constant.
The dual variable is reset to the uniform point at every stage; the dual radius $\cB$ does not shrink, and the dual scale $\rs$ (or the balance $\alpha$) is
increased from stage to stage to keep the two prox radii balanced. \Cref{fig:restart} shows the scheme.

\begin{figure}[t]
\centering
\resizebox{\textwidth}{!}{%
\begin{tikzpicture}[>=Latex,font=\small,node distance=6mm and 9mm,
 blk/.style={draw,rounded corners=2pt,align=center,inner sep=4pt,minimum height=8mm}]
 \node[blk,fill=gray!10] (init) {$\bar x^{(0)}:=x_0$, $\eps_0:=\mu D^2/2$,\\ $\bar\cA_1:=D^2/2$};
 \node[blk,fill=orange!12,right=of init] (par) {stage $k$: target $\eps_k=\eps_0 2^{-k}$,\\ radius bound $\bar\cA_k=\eps_{k-1}/\mu$,\\ scale $\rs_k^2=\cB/\bar\cA_k$};
 \node[blk,fill=blue!8,right=of par] (run) {run the base method from $\bar x^{(k-1)}$\\ with parameters $P_k(\bar\cA_k,\eps_k)$\\ (dual reset to $\tilde y^{\mathrm u}$)};
 \node[blk,fill=green!10,right=of run] (out) {$\bar x^{(k)}$:\\ $\E\Cert(\bar x^{(k)})\le\eps_k$};
 \node[blk,fill=gray!10,below=of out] (dist) {$\E\tfrac12\|\bar x^{(k)}-x^\ast_\gamma\|^2\le\eps_k/\mu=\bar\cA_{k+1}$};
 \draw[->] (init)--(par); \draw[->] (par)--(run); \draw[->] (run)--(out);
 \draw[->] (out)--(dist) node[midway,right] {\Cref{lem:certificate}};
 \draw[->] (dist.west) -| node[pos=.75,left] {$k\leftarrow k+1$} (par.south);
\end{tikzpicture}}
\caption{The restart scheme (\Cref{alg:restart}). The base method is either \ZOS\ (\Cref{alg:sliding}) or \BSMP\ (\Cref{alg:mp}); its
parameters at stage $k$ are chosen from the bound $\bar\cA_k$ on the expected squared distance of the starting point to $x^\ast_\gamma$.}
\label{fig:restart}
\end{figure}
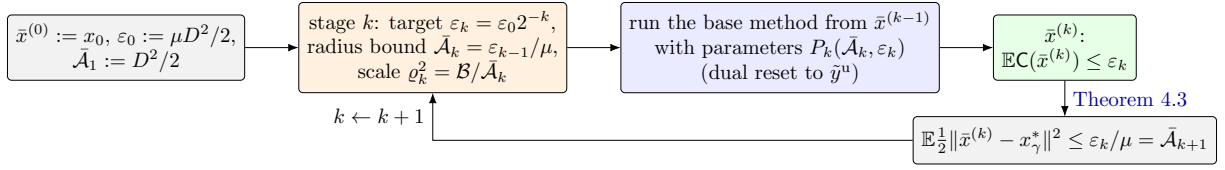

\begin{algorithm}[t]
\caption{\RZOS\ / R-\BSMP: restarts for strongly convex objectives}\label{alg:restart}
\begin{algorithmic}[1]
\Require $\mu$, $D$, target $\eps_{\mathrm s}\in(0,\mu D^2/2)$, a base method with a parameter rule $P(\bar\cA,\eps)$;
 \Statex\hspace{\algorithmicindent} starting point $\bar x^{(0)}:=x_0\in\Q$.
\State $\eps_0:=\mu D^2/2$, $K:=\lceil\log_2(\eps_0/\eps_{\mathrm s})\rceil$.
\For{$k=1,\dots,K$}
 \State $\eps_k:=\eps_02^{-k}$, $\bar\cA_k:=\eps_{k-1}/\mu$.
 \State Run the base method from $\bar x^{(k-1)}$ with the dual variable reset to $\tilde y^{\mathrm u}$ and parameters $P(\bar\cA_k,\eps_k)$; let $\bar x^{(k)}$ be its output.
\EndFor
\Ensure $\bar x^{(K)}$.
\end{algorithmic}
\end{algorithm}

\begin{lemma}[restart lemma]\label{lem:restart}
Let \Cref{ass:convex,ass:lip,ass:slater,ass:sc} hold. Suppose that for every $\bar\cA>0$ and $\eps>0$ the base method with parameters $P(\bar\cA,\eps)$, started at a
random point $x_0\in\Q$ independent of the oracle samples of the run, returns $\hat x$ with
\begin{equation}\label{eq:restart-hyp}
 \E\bigl[\Cert_{\Rmult,\gamma}(\hat x)\mid x_0\bigr]\le\Psi_{P}\bigl(\tfrac12\norm{x_0-x^\ast_\gamma}^2\bigr),
\end{equation}
where $\Psi_P:[0,\infty)\to[0,\infty)$ is nondecreasing and concave with $\Psi_{P(\bar\cA,\eps)}(\bar\cA)\le\eps$. Then the output of \Cref{alg:restart} satisfies
$\E\Cert_{\Rmult,\gamma}(\bar x^{(K)})\le\eps_K\le\eps_{\mathrm s}$, and $\eps_K>\eps_{\mathrm s}/2$.
\end{lemma}
\begin{proof}
We show by induction that $\E\frac12\norm{\bar x^{(k-1)}-x^\ast_\gamma}^2\le\bar\cA_k$ and $\E\Cert(\bar x^{(k)})\le\eps_k$ for $k=1,\dots,K$. For $k=1$,
$\frac12\norm{x_0-x^\ast_\gamma}^2\le D^2/2=\eps_0/\mu=\bar\cA_1$. If the first claim holds for $k$, then by \eqref{eq:restart-hyp}, Jensen's inequality (concavity)
and monotonicity, $\E\Cert(\bar x^{(k)})=\E\,\E[\Cert(\bar x^{(k)})\mid\bar x^{(k-1)}]\le\E\Psi(\tfrac12\norm{\bar x^{(k-1)}-x^\ast_\gamma}^2)\le\Psi(\bar\cA_k)\le\eps_k$;
here $\bar x^{(k-1)}$ is independent of the fresh samples of stage $k$. Then \Cref{lem:certificate} gives $\E\frac12\norm{\bar x^{(k)}-x^\ast_\gamma}^2\le\E\Cert(\bar x^{(k)})/\mu\le\eps_k/\mu=\bar\cA_{k+1}$.
Finally $\eps_K=\eps_02^{-K}\in(\eps_{\mathrm s}/2,\eps_{\mathrm s}]$ by the definition of $K$.
\end{proof}

\subsection{Restarted sliding}
\begin{theorem}[\RZOS]\label{thm:restart-sliding}
Let \Cref{ass:convex,ass:lip,ass:levels,ass:slater,ass:sc} hold, $\gamma\le\bar\gamma$, $L\ge H$, $\eps_{\mathrm s}\in(0,\mu D^2/2)$, and run \Cref{alg:restart} with the base method
\Cref{alg:sliding} and the stage-$k$ parameters
\begin{equation}\label{eq:restart-sliding-params}
\begin{gathered}
 \rs_k^2:=\frac{\cB}{\bar\cA_k},\qquad D_k^2:=4\bar\cA_k,\qquad N_k:=\Bigl\lceil\sqrt{32L/\mu}\Bigr\rceil,\\
 b_{p,k}:=\max\Bigl\{1,\Bigl\lceil\frac{2176\,dM_R^2D_k^2}{K^{(k)}\eps_k^2}\Bigr\rceil\Bigr\},\qquad
 b_{y,k}:=\max\Bigl\{1,\Bigl\lceil\frac{128\,\rs_k^2\sigma_y^2D_k^2}{K^{(k)}\eps_k^2}\Bigr\rceil\Bigr\},\qquad
 \lambda_k:=\frac{\Sigma_{b,k}}{D_k}\sqrt{\frac{A^{(k)}}{2}},
\end{gathered}
\end{equation}
where $\Sigma_{b,k}^2:=17dM_R^2/b_{p,k}+\rs_k^2\sigma_y^2/b_{y,k}$ and $K^{(k)}$, $A^{(k)}$ are the quantities $K_N$, $A_N$ of \eqref{eq:schedule} for $N=N_k$ and
$\rs=\rs_k$. Then $\E\Cert_{\Rmult,\gamma}(\bar x^{(K)})\le\eps_{\mathrm s}$ and, with $K=\lceil\log_2(\mu D^2/(2\eps_{\mathrm s}))\rceil$,
\begin{align}
 \Nseq&=\sum_{k=1}^KN_k\le K\Bigl(\sqrt{32L/\mu}+1\Bigr),\qquad
 \Nmv=\sum_{k=1}^KK^{(k)}\le\Nseq+167\,M_g\sqrt{\frac{\cB}{\mu\eps_{\mathrm s}}},\label{eq:restart-sliding-depth}\\
 \Nf&\le2\Nmv+\frac{1.4\cdot10^5\,dM_R^2}{\mu\eps_{\mathrm s}},\qquad
 \Norc=\Ng\le7\Nmv+\frac{2.8\cdot10^5\,dM_R^2}{\mu\eps_{\mathrm s}}+\frac{8.2\cdot10^3\,\cB\sigma_y^2}{\eps_{\mathrm s}^2},\label{eq:restart-sliding-calls}
\end{align}
where $\cB\sigma_y^2=16\Lam\Rmult^2(\sigma_h^2+\nuj M_g^2D^2)$. In $O$-notation,
\begin{gather*}
 \Nseq=O\Bigl(\sqrt{\tfrac L\mu}\log\tfrac{\mu D^2}{\eps_{\mathrm s}}\Bigr),\qquad
 \Nmv=O\Bigl(\Nseq+M_g\Rmult\sqrt{\tfrac{\log(m+1)}{\mu\eps_{\mathrm s}}}\Bigr),\\
 \Nf=O\Bigl(\Nmv+\frac{dM_R^2}{\mu\eps_{\mathrm s}}\Bigr),\qquad
 \Norc=\Ng=O\Bigl(\Nmv+\frac{dM_R^2}{\mu\eps_{\mathrm s}}+\frac{\Lam\Rmult^2(\sigma_h^2+\nuj M_g^2D^2)}{\eps_{\mathrm s}^2}\Bigr).
\end{gather*}
The number of calls in which the objective is used thus benefits fully from strong convexity, while the $\eps_{\mathrm s}^{-2}$ level term is charged to the
constraint calls only; this is what the two batch sizes of \Cref{alg:sliding} buy.
\end{theorem}
\begin{proof}
By \Cref{thm:sliding} (conditionally on the starting point), the base method with parameters $(N,b_p,b_y,\lambda,\rs)$ satisfies \eqref{eq:restart-hyp} with the
nondecreasing affine function $\Psi(\cA)=W_N^{-1}[(L+2\lambda)(3\cA+\cB/\rs^2)+A_N\Sb^2/\lambda]$. At stage $k$, $3\bar\cA_k+\cB/\rs_k^2=4\bar\cA_k=D_k^2$, and with
$\lambda_k$ as in \eqref{eq:restart-sliding-params} the proof of \Cref{cor:sliding-lambda} gives $\Psi(\bar\cA_k)\le2LD_k^2/(N_k(N_k+1))+4D_k\Sigma_{b,k}/\sqrt{K^{(k)}}$.
Since $D_k^2=4\eps_{k-1}/\mu=8\eps_k/\mu$, the choice $N_k^2\ge32L/\mu=4LD_k^2/\eps_k$ makes the first term at most $\eps_k/2$, and, exactly as in the proof of
\Cref{cor:sliding-complexity}, $b_{p,k}K^{(k)}\ge2176dM_R^2D_k^2/\eps_k^2$ and $b_{y,k}K^{(k)}\ge128\rs_k^2\sigma_y^2D_k^2/\eps_k^2$ make the second at most
$\eps_k/2$. Hence $\Psi(\bar\cA_k)\le\eps_k$ and \Cref{lem:restart} applies.

\emph{Counts.} $\Nseq=KN_k\le K(\sqrt{32L/\mu}+1)$. Since $LD_k^2=8L\eps_k/\mu\ge\eps_k$ (as $L\ge H\ge\mu$ because a $\mu$-strongly convex $L$-smooth function has $L\ge\mu$),
\Cref{cor:sliding-complexity} gives $K^{(k)}\le N_k+6M_g\rs_kD_k^2/\eps_k$ with $\rs_kD_k^2=4\sqrt{\cB\bar\cA_k}=4\sqrt{2\cB\eps_k/\mu}$, \ie\
$K^{(k)}\le N_k+24M_g\sqrt{2\cB/(\mu\eps_k)}$. Now $\eps_k=\eps_K2^{K-k}$ with $\eps_K>\eps_{\mathrm s}/2$, so
$\sum_k\eps_k^{-1/2}\le\eps_K^{-1/2}\sum_{j\ge0}2^{-j/2}\le\frac{\sqrt2}{1-2^{-1/2}}\eps_{\mathrm s}^{-1/2}\le4.83\,\eps_{\mathrm s}^{-1/2}$,
$\sum_k\eps_k^{-1}\le2\eps_K^{-1}\le4\eps_{\mathrm s}^{-1}$ and $\sum_k\eps_k^{-2}\le\frac43\eps_K^{-2}\le\frac{16}{3}\eps_{\mathrm s}^{-2}$. This gives
$\Nmv\le\Nseq+24\sqrt2\cdot4.83\,M_g\sqrt{\cB/(\mu\eps_{\mathrm s})}\le\Nseq+167M_g\sqrt{\cB/(\mu\eps_{\mathrm s})}$. For the calls, $dM_R^2D_k^2=8dM_R^2\eps_k/\mu$ and
$\rs_k^2\sigma_y^2D_k^2=(\cB/\bar\cA_k)\sigma_y^2\cdot4\bar\cA_k=4\cB\sigma_y^2$, so by \eqref{eq:sliding-counts2}
$\Nf^{(k)}\le2K^{(k)}+4352\cdot8\,dM_R^2/(\mu\eps_k)=2K^{(k)}+34816\,dM_R^2/(\mu\eps_k)$ and
$\Ng^{(k)}\le7K^{(k)}+8704\cdot8\,dM_R^2/(\mu\eps_k)+384\cdot4\,\cB\sigma_y^2/\eps_k^2=7K^{(k)}+69632\,dM_R^2/(\mu\eps_k)+1536\,\cB\sigma_y^2/\eps_k^2$. Summing with
$\sum_k\eps_k^{-1}\le4\eps_{\mathrm s}^{-1}$ and $\sum_k\eps_k^{-2}\le\frac{16}3\eps_{\mathrm s}^{-2}$ gives $34816\cdot4\le1.4\cdot10^5$, $69632\cdot4\le2.8\cdot10^5$ and
$1536\cdot\frac{16}3=8192\le8.2\cdot10^3$, which is \eqref{eq:restart-sliding-calls}.
\end{proof}

The term $\cB\sigma_y^2/\eps_{\mathrm s}^2$, which appears in $\Ng=\Norc$ but not in $\Nf$, does not benefit from strong convexity. Its first part, $\Lam\Rmult^2\sigma_h^2/\eps_{\mathrm s}^2$, is unavoidable up to
the factor $\kap\Rmult^2$ for any method, also when the objective is strongly convex (\Cref{thm:lower-levels} and \Cref{rem:lower-levels-sc}: even deciding which
constraint is active at a known point costs $\Omega(\sigma_g^2\log(1+m)/\eps^2)$ calls). Its second part,
$\Lam\Rmult^2\nuj M_g^2D^2/\eps_{\mathrm s}^2$, is the price of the linearization discussed in \Cref{sec:sliding-discussion}; the restarted Mirror--Prox method below does
not have it, at the cost of a larger number of rounds.

\subsection{Restarted Mirror--Prox}
\begin{theorem}[R-\BSMP]\label{thm:restart-mp}
Let \Cref{ass:convex,ass:lip,ass:levels,ass:slater,ass:sc} hold, $\gamma\le\bar\gamma$, $\eps_{\mathrm s}\in(0,\mu D^2/2)$, and run \Cref{alg:restart} with the base method \Cref{alg:mp}
and the stage-$k$ parameters
\begin{equation}\label{eq:restart-mp-params}
\begin{gathered}
 \alpha_k:=\sqrt{\cB/\bar\cA_k},\qquad N_k:=\Bigl\lceil2\sqrt3\Bigl(\frac{4H}{\mu}+2M_g\sqrt{\frac{2\cB}{\mu\eps_k}}\Bigr)\Bigr\rceil,\\
 b_k:=\max\Bigl\{1,\Bigl\lceil\frac{200}{N_k}\Bigl(\frac{4dM_R^2}{\mu\eps_k}+\frac{8\kap\sigma_h^2\cB}{\eps_k^2}\Bigr)\Bigr\rceil\Bigr\},
\end{gathered}
\end{equation}
and the step $\eta$ of \eqref{eq:mp-rate} with $\cD_{\alpha_k}^2$ replaced by $2\sqrt{\cB\bar\cA_k}$. Then $\E\Cert_{\Rmult,\gamma}(\bar x^{(K)})\le\eps_{\mathrm s}$ and, with $K=\lceil\log_2(\mu D^2/(2\eps_{\mathrm s}))\rceil$,
\begin{gather*}
 \Nseq=2\sum_kN_k\le K\Bigl(\frac{16\sqrt3H}{\mu}+2\Bigr)+95\,M_g\sqrt{\frac{\cB}{\mu\eps_{\mathrm s}}},\\
 \Norc\le3\Nseq+\frac{2\cdot10^4\,dM_R^2}{\mu\eps_{\mathrm s}}+\frac{5.2\cdot10^4\,\kap\Rmult^2\log(m+1)\,\sigma_h^2}{\eps_{\mathrm s}^2},
\end{gather*}
\ie\ $\Nseq=O\bigl(\frac H\mu\log\frac{\mu D^2}{\eps_{\mathrm s}}+M_g\Rmult\sqrt{\log(m+1)/(\mu\eps_{\mathrm s})}\bigr)$ and
$\Norc=O\bigl(\Nseq+\frac{dM_R^2}{\mu\eps_{\mathrm s}}+\frac{\Lam\Rmult^2\sigma_h^2}{\eps_{\mathrm s}^2}\bigr)$; no term proportional to $D^2/\eps_{\mathrm s}^2$ appears.
\end{theorem}
\begin{proof}
Fix a stage and write $\alpha=\alpha_k$, $\bar\cA=\bar\cA_k$, $\bar\cD^2:=2\sqrt{\cB\bar\cA}=\alpha\bar\cA+\cB/\alpha$. By \Cref{rem:mp-random-start}, for a starting point $x_0$
with $\cA=\frac12\norm{x_0-x^\ast_\gamma}^2$ and any deterministic step $\eta\le1/(\sqrt3L_\alpha)$,
\[
 \E[\Cert\mid x_0]\le\Psi(\cA):=\frac{\alpha\cA+\cB/\alpha}{\eta N}+3\eta\sigma_\ast^2+\sqrt{\frac{2(\alpha\cA+\cB/\alpha)\sigma_\ast^2}{N}},
\]
a nondecreasing concave function of $\cA$. With $\eta:=\min\{1/(\sqrt3L_\alpha),\bar\cD/(\sigma_\ast\sqrt{3N})\}$ the computation in the proof of
\Cref{thm:mp} (with $\cD_\alpha$ replaced by $\bar\cD$) gives $\Psi(\bar\cA)\le\sqrt3L_\alpha\bar\cD^2/N+5\sigma_\ast\bar\cD/\sqrt N$.
At stage $k$, $\bar\cA_k=2\eps_k/\mu$, hence $L_{\alpha_k}\bar\cD_k^2\le(H/\alpha_k+M_g)\cdot2\sqrt{\cB\bar\cA_k}=2H\bar\cA_k+2M_g\sqrt{\cB\bar\cA_k}=4H\eps_k/\mu+2M_g\sqrt{2\cB\eps_k/\mu}$,
and $\sigma_\ast^2\bar\cD_k^2\le\frac1{b_k}(2dM_R^2/\alpha_k+8\kap\alpha_k\sigma_h^2)\cdot2\sqrt{\cB\bar\cA_k}=\frac{4}{b_k}(dM_R^2\bar\cA_k+4\kap\sigma_h^2\cB)$ (using $\alpha_k\sqrt{\cB\bar\cA_k}=\cB$ and
$\sqrt{\cB\bar\cA_k}/\alpha_k=\bar\cA_k$). The choices \eqref{eq:restart-mp-params} make both terms at most $\eps_k/2$: $\sqrt3L_{\alpha_k}\bar\cD_k^2/N_k\le\eps_k/2$ iff
$N_k\ge2\sqrt3(4H/\mu+2M_g\sqrt{2\cB/(\mu\eps_k)})$, and $25\sigma_\ast^2\bar\cD_k^2/N_k\le\eps_k^2/4$ iff $b_kN_k\ge400(dM_R^2\bar\cA_k+4\kap\sigma_h^2\cB)/\eps_k^2=200(4dM_R^2/(\mu\eps_k)+8\kap\sigma_h^2\cB/\eps_k^2)$.
\Cref{lem:restart} gives the accuracy claim. The counts follow from $\Nseq=2\sum_kN_k$, $\Norc=6\sum_kb_kN_k\le6\sum_kN_k+1200\sum_k(4dM_R^2/(\mu\eps_k)+8\kap\sigma_h^2\cB/\eps_k^2)$
and the sums $\sum_k\eps_k^{-1/2}\le4.83\eps_{\mathrm s}^{-1/2}$, $\sum_k\eps_k^{-1}\le4\eps_{\mathrm s}^{-1}$, $\sum_k\eps_k^{-2}\le\frac{16}3\eps_{\mathrm s}^{-2}$ computed in the proof of
\Cref{thm:restart-sliding}: $2\sum_kN_k\le2K(8\sqrt3H/\mu+1)+8\sqrt6M_g\sqrt{\cB/\mu}\cdot4.83\,\eps_{\mathrm s}^{-1/2}$ with $8\sqrt6\cdot4.83\le95$; $6\sum_kN_k=3\Nseq$;
$1200\cdot4\cdot4\,dM_R^2/(\mu\eps_{\mathrm s})\le2\cdot10^4dM_R^2/(\mu\eps_{\mathrm s})$; and $1200\cdot8\cdot\frac{16}3\kap\sigma_h^2\cB=51200\,\kap\sigma_h^2\cB\le5.2\cdot10^4\kap\Rmult^2\log(m+1)\sigma_h^2$.
\end{proof}

\section{Exactly known affine constraints: matrix--vector products}\label{sec:affine}

When the constraints are affine and known exactly, $g(x)=Ax-c$ with $A\in\R^{m\times d}$ and $c\in\R^m$, the constraint part of the oracle is not needed:
$h(x)=Ax-c$ (no smoothing, no shift, $b_i=0$), $J_t=A$ for all $t$, and the inner dual directions $q_k=Ap_k-c$ are exact. The only stochastic zeroth-order
information concerns the objective. In this situation \Cref{alg:sliding} simplifies considerably: since the dual step has no noise, the inner stabilizer
is not needed, and since the Jacobian is exact, \emph{one} mini-batch of paired queries per phase suffices for the primal directions (the reuse of the
objective gradient estimate within a phase is harmless, as in accelerated stochastic gradient methods, because the primal error only has to be centered).
Each inner step then costs one multiplication by $A^\top$ and one by $A$, which is the matrix--vector product count of ACGD-S \citep{zhanglan2022}; here the
inner steps are oracle-free in the literal sense. In the accounting of \Cref{def:measures}, round $t$ issues $2b$ vector calls, all of them used for the objective
only, so that $\Norc=\Nf=2bN$ and $\Ng=0$. We write $M_A:=\rowinf A$.

\begin{algorithm}[t]
\caption{\AZOS: sliding with exactly known affine constraints}\label{alg:affine}
\begin{algorithmic}[1]
\Require $L\ge L_{0,\gamma}$, $M_A$, $\Rmult$, $\gamma$, $\rs>0$, $\lambda\ge0$, $b\ge1$, $N$; schedule $\tau_t,\theta_t$ as in \eqref{eq:schedule},
$S_t=\max\{1,\lceil M_A\rs t/L\rceil\}$, $a_t=t/S_t$, $\eta_t=(2L+\lambda)/t$, $\beta_t=L/a_t$; initialization as in \Cref{alg:sliding}.
\For{$t=1,\dots,N$}
 \State $u_t:=[\tau_tu_{t-1}+x_{t-1}+\theta_t(x_{t-1}-x_{t-2})]/(1+\tau_t)$.
 \State \textbf{Round $t$.} Issue $b$ paired queries of the objective at $u_t$ and form $\hat g_t:=\hat g_0^{A_t}(u_t)$.
 \For{$s=1,\dots,S_t$}
  \State $k:=k+1$;\quad $\hat r_k:=\hat g_t+A^\top y_{k-1}+\rho_kA^\top(y_{k-1}-y_{k-2})$
  \Statex\hspace{2\algorithmicindent} with $\rho_k=1$ for $s\ge2$, $\rho_k=a_{t-1}/a_t$ for $s=1$, $t\ge2$ (no last term for $k=1$).
  \State $p_k:=\argmin_{x\in\Q}\{\inner{\hat r_k}{x}+\eta_t\Vx(x_{t-1},x)+\beta_t\Vx(p_{k-1},x)\}$;\quad $q_k:=Ap_k-c$;
  \State $\tilde y_k:=\argmax_{\tilde y\in\tilde Y_\Rmult}\{\inner{(0,q_k)}{\tilde y}-\rs^{-2}\beta_t\Vy(\tilde y_{k-1},\tilde y)\}$.
 \EndFor
 \State $x_t:=\frac1{S_t}\sum_{k}p_k$ over the slots of phase $t$.
\EndFor
\Ensure $\Xr_N:=W_N^{-1}\sum_tt\,x_t$.
\end{algorithmic}
\end{algorithm}

\begin{theorem}[\AZOS]\label{thm:affine}
Let \Cref{ass:convex,ass:lip,ass:slater} hold for the problem with $g(x)=Ax-c$, let $\gamma\le\bar\gamma$, $\Rmult$ as in \eqref{eq:R}, $L\ge L_{0,\gamma}$ and
$\lambda>0$. Then the output of \Cref{alg:affine} satisfies
\begin{equation}\label{eq:affine-main}
 W_N\,\E\,\Cert_{\Rmult,\gamma}(\Xr_N)\le(3L+\lambda)\cA+\frac{L\cB}{\rs^2}+\frac{\sigma_f^2}{2\lambda b}\sum_{t=1}^Nt^2,\qquad \sigma_f^2:=2dM_0^2 ,
\end{equation}
and with $\lambda:=\sigma_f\sqrt{\sum_tt^2/(2b\cA)}$
\begin{equation}\label{eq:affine-rate}
 \E\,\Cert_{\Rmult,\gamma}(\Xr_N)\le\frac{2L\dA^2}{N(N+1)}+2\sigma_f\sqrt{\frac{\cA}{b(N+1)}},\qquad \dA^2=3\cA+\frac{\cB}{\rs^2}.
\end{equation}
\end{theorem}
\begin{proof}
We specialize the argument of \Cref{lem:inner} to the affine case; the details are given in \Cref{app:affine}. The dual error vanishes,
$e^y_k=0$, and the primal error of every slot of phase $t$ equals $\hat e_t:=\hat g_t-\nabla f_\gamma(u_t)$, which is independent of $\cF_{t-1}$ (the
information before round $t$) with $\E\hat e_t=0$ and $\E\norm{\hat e_t}^2\le\sigma_f^2/b$ (\Cref{lem:twopoint}). With $a_k\beta_k=L$ the inner primal terms
telescope to $L[\Vx(p_0,x)-\Vx(p_K,x)]-L\sum_k\Vx(p_{k-1},p_k)$, which is exactly what is needed to absorb the bilinear cross terms (they need $L/2$ per
increment, see \eqref{eq:cross-bound}); the dual terms telescope to $\rs^{-2}L[\Vy(\tilde y_0,\tilde y)-\Vy(\tilde y_K,\tilde y)-\sum_k\Vy(\tilde y_{k-1},\tilde y_k)]$; and the outer
terms, with $t\eta_t=2L+\lambda$, telescope to $(2L+\lambda)[\Vx(x_0,x)-\Vx(x_N,x)]-(L+\tfrac\lambda2)\sum_t\norm{d_t}^2$. Hence, pathwise for every $y\in Y_\Rmult$,
\[
 \sum_tt\bigl[\ell_t(x_t,y)-\ell_t(x^\ast_\gamma,\bar y_t)\bigr]+\Bigl(L+\frac\lambda2\Bigr)\sum_t\norm{d_t}^2\le(3L+\lambda)\cA+\frac{L\cB}{\rs^2}-\sum_ka_k\inner{\hat e_{t(k)}}{p_k-x^\ast_\gamma}.
\]
The error sum equals $\sum_tt\inner{\hat e_t}{x_t-x^\ast_\gamma}$ because $\sum_{k\in\text{phase }t}a_tp_k=tx_t$. Splitting $x_t-x^\ast_\gamma=d_t+(x_{t-1}-x^\ast_\gamma)$ and using
Young's inequality, $-t\inner{\hat e_t}{d_t}\le\frac{t^2}{2\lambda}\norm{\hat e_t}^2+\frac\lambda2\norm{d_t}^2$, where the increment term is cancelled by the left-hand side. The
remaining error terms $-t\inner{\hat e_t}{x_{t-1}-x^\ast_\gamma}$ have zero mean because $x_{t-1}$ is $\cF_{t-1}$-measurable. Since the right-hand side no longer depends on
$y$, taking the supremum over $y\in Y_\Rmult$, using \Cref{lem:outer} for $\mathsf F=\LagS(\cdot,y)$ (which is $L$-smooth because $L_{g,\gamma}=0$) and
\eqref{eq:restrictedgap} exactly as in Step 4 of the proof of \Cref{thm:sliding}, and taking expectations gives \eqref{eq:affine-main}. Minimizing
$\lambda\cA+\sigma_f^2\sum_tt^2/(2\lambda b)$ over $\lambda$ gives $\sigma_f\sqrt{2\cA\sum_tt^2/b}$; since $\sum_{t=1}^Nt^2=\frac{N(N+1)(2N+1)}6\le NW_N$ and $3L\cA+L\cB/\rs^2\le L\dA^2$,
dividing by $W_N$ yields \eqref{eq:affine-rate}.
\end{proof}

\begin{corollary}[complexity of \AZOS]\label{cor:affine}
In the setting of \Cref{thm:affine}, let $\eps_{\mathrm s}\in(0,L\dA^2]$ and choose $N:=\lceil2\dA\sqrt{L/\eps_{\mathrm s}}\rceil$,
$b:=\max\{1,\lceil16\sigma_f^2\cA/((N+1)\eps_{\mathrm s}^2)\rceil\}$ (with $\cA$ replaced by the known bound $D^2/2$ if necessary). Then $\E\Cert_{\Rmult,\gamma}(\Xr_N)\le\eps_{\mathrm s}$,
\begin{gather*}
 \Nseq=N\le2\dA\sqrt{\frac L{\eps_{\mathrm s}}}+1,\qquad \Norc=\Nf=2bN\le2N+\frac{64\,dM_0^2\cA}{\eps_{\mathrm s}^2},\\
 \Ng=0,\qquad \Nmv=2K_N\le2N+\frac{12\,M_A\rs\dA^2}{\eps_{\mathrm s}} .
\end{gather*}
With $\rs=\sqrt{2\cB}/D$: $\dA^2\le2D^2$ and $\Nmv\le2N+34\,M_A\Rmult D\sqrt{\log(m+1)}/\eps_{\mathrm s}$. For a nonsmooth objective ($\gamma=\eps/(4M_0)$, $L=\sqrt dM_0/\gamma$, $\eps_{\mathrm s}=\eps/2$,
$\eps\le2d^{1/4}M_0D$ so that $\eps_{\mathrm s}\le L\dA^2$) this gives $\Nseq\le8d^{1/4}M_0D/\eps+1$ and $\Nf=O(\Nseq+dM_0^2D^2/\eps^2)$; for an $L_0$-smooth objective
($\gamma$ arbitrarily small, $L=L_0$, $\eps\le L_0D^2$) $\Nseq\le4D\sqrt{L_0/\eps}+1$.
\end{corollary}
\begin{proof}
As in \Cref{cor:sliding-complexity}: the two terms of \eqref{eq:affine-rate} are at most $\eps_{\mathrm s}/2$ each, $K_N\le N+\frac{M_A\rs}{2L}N(N+1)\le N+6M_A\rs\dA^2/\eps_{\mathrm s}$
(each inner step performs two matrix--vector products), $2bN\le2N+32\sigma_f^2\cA N/((N+1)\eps_{\mathrm s}^2)\le2N+32\sigma_f^2\cA/\eps_{\mathrm s}^2$, and $6\sqrt2\cdot2\le17\cdot2$.
\end{proof}

The objective-call complexity $O(dM_0^2D^2/\eps^2)$ and the round complexities $O(\sqrt{L_0D^2/\eps})$ and $O(d^{1/4}M_0D/\eps)$ are those of the
\emph{unconstrained} problem \citep{gasnikov2022power}; the constraints cost only $O(M_A\Rmult D\sqrt{\log(m+1)}/\eps)$ matrix--vector products, which is the
ACGD-S accounting \citep[Cor.~3]{zhanglan2022} with the Euclidean dual norm replaced by the entropy geometry (whence $\sqrt{\log(m+1)}$ instead of $\sqrt m$
for the dual radius). The order $O(1/\eps)$ of $\Nmv$ cannot be improved by methods that access $A$ only through matrix--vector products
\citep{ouyang2021lower}; that lower bound concerns the exponent of $\eps$ and does not address the exact dependence on $M_A$, $\Rmult$, $D$ and $m$.

\section{Lower bounds}\label{sec:lower}

We now show that the $\eps^{-2}$ terms in the call complexities of \Cref{sec:sliding,sec:restarts} are unavoidable, and we identify which ingredients are
responsible for them: the dimension factor $d$ for the two-point information about gradients (objective \emph{and} constraints), and the factor
$\sigma_g^2\log(1+m)$ for the information about the constraint levels. All lower bounds are minimax statements over classes of problems satisfying
\Cref{ass:convex,ass:lip,ass:levels,ass:slater} (with the indicated constants); they hold for arbitrary, possibly randomized and adaptive, algorithms
whose output $\hat x$ is any measurable function of the observations. \Cref{thm:lower-objective,thm:lower-constraint} count paired queries (one paired query is
two vector calls, so the bounds hold for $\Norc$ with the constant halved), \Cref{thm:lower-levels} counts vector calls and \Cref{thm:lower-scalar} scalar calls.
The proofs are given in \Cref{app:lower}.

\subsection{Two-point calls for the objective}
\begin{theorem}[objective information]\label{thm:lower-objective}
Let $d\ge1$, $M>0$, $R_x>0$, $\Q=R_x\Ball$, $m\ge1$ and $T\ge1$. For every algorithm that makes at most $T$ paired queries and outputs $\hat x\in\Q$ there is a problem
in the class of linear objectives $F(x,\xi)=\inner{\xi}{x}$ with $\xi\sim\mathcal N(\theta,s^2I_d)$, $\norm\theta\le M/2$, $s=M/\sqrt{2d}$ (which satisfy \Cref{ass:lip} with $M_0=M$)
and inactive constraints $G_i\equiv-1$ (\Cref{ass:levels} with $\sigma_g=0$, \Cref{ass:slater} with $\rho=1$) such that
\[
 \E\bigl[f(\hat x)-f^\ast\bigr]\ \ge\ \frac{MR_x}{32}\,\min\Bigl\{1,\sqrt{\frac dT}\Bigr\}.
\]
Consequently, an $\eps$-solution in expectation with $\eps<MR_x/32$ requires $T\ge dM^2R_x^2/(1024\eps^2)$ paired queries.
\end{theorem}
The proof is Assouad's method with $d$ binary parameters and the chain rule for the Kullback--Leibler divergence; it is the two-point (and thus
noise-cancelling) version of the argument of \citet{duchi2015optimal,shamir2017optimal}: even though the difference $F(x^+,\xi)-F(x^-,\xi)=\inner\xi{x^+-x^-}$
removes the noise in the value, a paired query reveals only the projection of the unknown mean onto the span of the two query points, a linear measurement of
rank at most two, and $d$ coordinates of the mean have to be estimated to accuracy $\eps/R_x$. Together with \Cref{cor:affine} this shows that the term $dM_0^2D^2/\eps^2$ in the objective calls is optimal up to constants, and in the affine case
(\Cref{thm:affine}) the whole objective complexity is optimal up to the lower-order term $\Nseq$; in \Cref{cor:sliding-nonsmooth} the objective and the
constraint information are estimated jointly and enter through $M_R=M_0+\Rmult M_g$, so there the match is up to the multiplier factor.

\subsection{Two-point calls for the constraints}
\begin{theorem}[constraint information]\label{thm:lower-constraint}
Let $d\ge1$, $M>0$, $R_x>0$, $T\ge1$. Consider problems in the variables $z=(x,t)\in\Q':=R_x\Ball\times[-2R_x,2R_x]$ with the known deterministic objective
$f(z)=Mt$ and the single stochastic constraint $G(z,\xi)=\inner{\xi}{x}-Mt$, $\xi\sim\mathcal N(\theta,s^2I_d)$, $\norm\theta\le M/2$, $s=M/\sqrt{2d}$. The class has
diameter $D=2\sqrt5R_x$ and satisfies \Cref{ass:lip} with $M_0=M$ and $M_g=\frac{\sqrt7}2M$ (since $\E\norm{(\xi,-M)}^2\le\frac34M^2+M^2$),
\Cref{ass:levels} on $\Q'_{\bar\gamma}$ with $\sigma_g^2=s^2(R_x+\bar\gamma)^2$, and \Cref{ass:slater} with $\rho=MR_x$ at $z=(0,R_x)$; the active multiplier equals
$1$ in every problem of the class. For every algorithm that makes at most $T$ paired queries of the constraint and outputs $\hat z\in\Q'$ there is a problem in
this class with
\[
 \E\Bigl[\bigl(f(\hat z)-f^\ast\bigr)+\pos{g(\hat z)}\Bigr]\ \ge\ \frac{MR_x}{32}\,\min\Bigl\{1,\sqrt{\frac dT}\Bigr\}
 =\frac{M_gD}{32\sqrt{35}}\,\min\Bigl\{1,\sqrt{\frac dT}\Bigr\}.
\]
Consequently, an $\eps$-solution in expectation in the sense of \eqref{eq:goal} with $\eps<MR_x/64$ requires
$T\ge dM^2R_x^2/(4096\,\eps^2)=dM_g^2D^2/(143360\,\eps^2)$ paired queries.
\end{theorem}
The proof uses a scaled epigraph transformation, $\min_x\varphi(x)=\min\{Mt:\varphi(x)\le Mt\}$. An algorithm for the constrained problem is then an algorithm
for minimizing the unknown linear function $\varphi(x)=\inner\theta x$ with the same information, and $\varphi(\hat x)-\varphi^\ast\le(M\hat t-f^\ast)+\pos{g(\hat z)}$.
We scale the epigraph variable by $M$ so that $\Q'$ has diameter of order $R_x$ and the constraint is $O(M)$-Lipschitz. The bound can then be written in the
parameters $M_g$ and $D$ of the class. Without scaling, the set $\{\vartheta:\varphi(x)\le\vartheta\}$ would have diameter of order $MR_x$ and Lipschitz constant
$\sqrt{M^2+1}$, and the bound would not take the form $dM_g^2D^2/\eps^2$ uniformly in $M$. We bound the sum of gap and violation because
$f(\hat z)-f^\ast=M\hat t-\varphi^\ast$ can be negative at infeasible points; by \eqref{eq:goal} the expected sum is at most $2\eps$.

So two-point information about \emph{constraint} gradients costs $dM_g^2D^2/\eps^2$, just like information about the objective, with universal constants.
In particular, the term $d\Rmult^2M_g^2D^2/\eps^2$ of $M_R^2=(M_0+\Rmult M_g)^2$ in \Cref{cor:sliding-nonsmooth} cannot be removed in general, since upper and lower
bounds agree in $d$, $M_g$, $D$ and $\eps$. In this construction the active multiplier is $1$, so it says nothing about the factor $\Rmult^2$; see \Cref{op:multiplier}.

\subsection{Calls for the constraint levels}
The third source of $\eps^{-2}$ complexity is the noise in the constraint \emph{values}, which does not cancel in differences. The following theorems concern
the very simple one-dimensional problem of finding, among $m$ constraints of the form $x\le\beta_i$, the one with the smallest level. They show that the entropy
dual geometry, which produces the factors $\log(m+1)$ in \Cref{thm:sliding,thm:mp}, is not an artifact: $\log(1+m)$ vector calls are necessary, and if a call
returns only one component, $m$ calls are necessary.

\begin{theorem}[level information, vector oracle]\label{thm:lower-levels}
Let $m\ge1$, $\sigma>0$ and $0<\eps\le\min\{\sigma/16,\,1/16\}$. Consider the one-dimensional problems with $\Q=[-1,1]$ (so $D=2$), known objective $f(x)=-x$, and
constraints $G_i(x,\xi)=x-\xi_i$, where $\xi=(\xi_1,\dots,\xi_m)$ has independent coordinates taking the values $\pm\sigma$ with $\E\xi_i=\beta_i\in\{-a,a\}$, $a:=8\eps\le1/2$;
the class consists of the $m+1$ vectors $\beta^{(0)}=(a,\dots,a)$ and $\beta^{(j)}=\beta^{(0)}-2ae_j$, $j=1,\dots,m$ (\Cref{ass:lip} with $M_0=M_g=1$, \Cref{ass:levels} with
$\sigma_g=2\sigma$, \Cref{ass:slater} at $x^{\mathrm s}=-1$ with $\rho=1-a\ge1/2$; the active multiplier is $y^\ast=1$ in every problem of the class, so $\Rmult=2$ is
admissible in \eqref{eq:R}). Any algorithm that makes at most $T$ vector calls and outputs $\hat x$ with $\E[f(\hat x)-f^\ast]\le\eps$ and $\E\viol(\hat x)\le\eps$ for every
problem of the class satisfies
\[
 T\ \ge\ \frac{3\,\sigma^2\log(1+m)}{1024\,\eps^2}=\frac{3\,\sigma_g^2\log(1+m)}{4096\,\eps^2}.
\]
\end{theorem}

\begin{theorem}[level information, scalar oracle]\label{thm:lower-scalar}
In the setting of \Cref{thm:lower-levels}, suppose that each call returns a single component $G_i(x,\xi)$ for an index $i$ chosen by the algorithm (a fresh $\xi$ per
call). Then any algorithm with the same guarantee satisfies $\E_{0}[T]\ge25\,m\sigma^2/(8192\,\eps^2)\ge m\sigma^2/(328\,\eps^2)$, where $T$ is the (possibly random) total
number of calls and $\E_0$ is the expectation under $\beta^{(0)}$.
\end{theorem}

\begin{remark}[strongly convex objectives]\label{rem:lower-levels-sc}
The level information remains an $\eps^{-2}$ bottleneck when the objective is strongly convex. Replace $f(x)=-x$ in \Cref{thm:lower-levels} by
$f(x)=-x+\frac\mu2x^2$ with $\mu\in(0,1]$, which is $\mu$-strongly convex and decreasing on $[-1,1]$, and take $a:=16\eps$ with $0<\eps\le\min\{\sigma/32,1/32\}$. Then
every algorithm with the guarantee of \Cref{thm:lower-levels} satisfies $T\ge3\sigma^2\log(1+m)/(4096\,\eps^2)$, and in the scalar model of \Cref{thm:lower-scalar}
$\E_0[T]\ge25m\sigma^2/(32768\,\eps^2)$. The proof is given in \Cref{app:lower}.
\end{remark}

\Cref{thm:lower-levels} is proved with a mixture argument ($\chi^2$-divergence between the uniform mixture of the $m$ alternatives and the null
hypothesis); \Cref{thm:lower-scalar} with the chain rule for the Kullback--Leibler divergence applied to the number of observations of each coordinate.
In both cases the optimization guarantee is converted into a statistical test: a solution with gap and violation below $\eps$ reveals the sign of the
smallest level.

\subsection{Depth}\label{sec:depth}
The lower bounds above concern the number of calls. Lower bounds on the number of \emph{rounds} of adaptive queries are much harder to obtain and are
known only for the unconstrained problem: for nonsmooth $M$-Lipschitz convex minimization on the unit ball with polynomially many gradient (or value)
queries per round, no algorithm can guarantee accuracy $\eps$ in fewer than $\tilde\Omega(\min\{\eps^{-2},d^{1/3}\eps^{-2/3}\})$ rounds
\citep{nemirovski1994parallel,balkanski2018parallelization,bubeck2019complexity,diakonikolas2020lower}. The best known upper bounds are
$O(d^{1/4}\eps^{-1})$ by smoothing \citep{duchi2012randomized,gasnikov2022power} and, in the regime $\eps\lesssim d^{-1/4}$, $\tilde O(d^{1/3}\eps^{-2/3})$ by ball
acceleration \citep{bubeck2019complexity,carmon2023resqueing}. Ball acceleration uses (stochastic) gradient queries, and it is open whether it can be
run with a two-point value oracle at the same depth. An unconstrained problem is a special case of \eqref{eq:P} with inactive constraints, so these lower bounds
apply to $\Nseq$ in our setting. \Cref{cor:sliding-nonsmooth} matches the best known depth for the unconstrained two-point problem up to the factor $\sqrt{M_R/M}$.
For smooth problems, $\Omega(\sqrt{L_0D^2/\eps})$ rounds are necessary even with exact gradients
\citep{nemirovski1983problem}, so the depth of \Cref{cor:sliding-smooth} is optimal up to the factor $\sqrt{(L_0+\Rmult L_g)/L_0}$. The smoothness of the
constraints cannot simply be dropped from this factor: for the scaled epigraph problem $\min\{Mt:\varphi(x)-Mt\le0,\ x\in R_x\Ball,\ |t|\le2R_x\}$ of
\Cref{thm:lower-constraint} with an $M$-Lipschitz, $L_g$-smooth convex $\varphi$ (here $L_0=0$, the active multiplier is $1$, and the diameter is $2\sqrt5R_x=\Theta(D)$),
an $\eps$-solution in the sense of \eqref{eq:goal} is a $2\eps$-minimizer of $\varphi$ over $R_x\Ball$ in expectation, so
$\Omega(\sqrt{L_gD^2/\eps})$ rounds are necessary as well. Whether the coupling constant $M_g\Rmult$ must appear in the depth of any zeroth-order method,
as it does in $\Nmv$, is open (\Cref{op:depth}).

\section{Summary of the rates}\label{sec:summary}

\Cref{tab:rates} collects the complexity bounds of \Cref{sec:mp,sec:sliding,sec:restarts,sec:affine} for the target accuracy $\eps$ of the original
problem \eqref{eq:P} (the smoothed target is $\eps_{\mathrm s}=\eps/2$ and the biases are at most $\eps/4$). Constants are omitted; the exact
statements with constants are the theorems and corollaries referenced in the first column. \Cref{tab:lower} confronts the $\eps^{-2}$ terms with the
lower bounds of \Cref{sec:lower}.

\begin{table}[t]
\centering\footnotesize\setlength{\tabcolsep}{2.5pt}
\caption{Complexity bounds for an $\eps$-solution in expectation. $H=L_0+\Rmult L_g$, $M_R=M_0+\Rmult M_g$, $M=\max\{M_0,M_g\}$, $\Lam=(1+\log2m)\log(m+1)$;
$\sigma_g^2$ bounds the variance of the constraint values; $\nuj=\min\{4d,92\kap\}=O(\min\{d,1+\log m\})$. The three call channels of \Cref{def:measures} are
listed separately: $\Nf$ (calls in which the objective is used), $\Ng$ (calls in which the constraint vector is used) and $\Norc$ (all calls); the call columns
list the terms beyond $O(\Nmv)$; ``$\Nf+\dots$'' in the $\Ng$ column means the $\Nf$ term plus the listed one. For \Cref{alg:sliding,alg:mp} every call carries the
constraint vector, so $\Norc=\Ng$; for \Cref{alg:affine}, $\Ng=0$ and $\Norc=\Nf$.}
\label{tab:rates}
\resizebox{\textwidth}{!}{%
\begin{tabular}{@{}llllll@{}}
\toprule
Result & Data & $\Nseq$ & $\Nf$ & $\Ng$ ($=\Norc$) & $\Nmv$\\
\midrule
\Cref{cor:mp-cases}(a) & smooth & $\frac{HD^2}{\eps}+\frac{M_g\Rmult D\sqrt{\log m}}{\eps}$ & $\frac{dM_R^2D^2+\Lam\Rmult^2\sigma_g^2}{\eps^2}$ & $=\Nf$ & $\Nseq$\\[3pt]
\Cref{cor:mp-cases}(b) & nonsmooth & $\frac{\sqrt dM_RMD^2}{\eps^2}+\frac{M_g\Rmult D\sqrt{\log m}}{\eps}$ & $\frac{dM_R^2D^2+\Lam\Rmult^2\sigma_g^2}{\eps^2}$ & $=\Nf$ & $\Nseq$\\[3pt]
\Cref{cor:sliding-smooth} & smooth & $\sqrt{\frac{HD^2}{\eps}}$ & $\frac{dM_R^2D^2}{\eps^2}$ & $\Nf+\frac{\Lam\Rmult^2(\sigma_g^2+\nuj M_g^2D^2)}{\eps^2}$ & $\Nseq+\frac{M_g\Rmult D\sqrt{\log m}}{\eps}$\\[3pt]
\Cref{cor:sliding-nonsmooth} & nonsmooth & $\frac{d^{1/4}\sqrt{M_RM}D}{\eps}$ & $\frac{dM_R^2D^2}{\eps^2}$ & $\Nf+\frac{\Lam\Rmult^2(\sigma_g^2+\nuj M_g^2D^2)}{\eps^2}$ & $\Nseq+\frac{M_g\Rmult D\sqrt{\log m}}{\eps}$\\[3pt]
\Cref{thm:restart-sliding} & $\mu$-s.c., smooth & $\sqrt{\frac H\mu}\log\frac{\mu D^2}{\eps}$ & $\frac{dM_R^2}{\mu\eps}$ & $\Nf+\frac{\Lam\Rmult^2(\sigma_g^2+\nuj M_g^2D^2)}{\eps^2}$ & $\Nseq+M_g\Rmult\sqrt{\frac{\log m}{\mu\eps}}$\\[3pt]
\Cref{thm:restart-sliding} & $\mu$-s.c., nonsmooth & $d^{1/4}\sqrt{\frac{M_RM}{\mu\eps}}\log\frac{\mu D^2}{\eps}$ & $\frac{dM_R^2}{\mu\eps}$ & $\Nf+\frac{\Lam\Rmult^2(\sigma_g^2+\nuj M_g^2D^2)}{\eps^2}$ & $\Nseq+M_g\Rmult\sqrt{\frac{\log m}{\mu\eps}}$\\[3pt]
\Cref{thm:restart-mp} & $\mu$-s.c., smooth & $\frac H\mu\log\frac{\mu D^2}{\eps}+M_g\Rmult\sqrt{\frac{\log m}{\mu\eps}}$ & $\frac{dM_R^2}{\mu\eps}+\frac{\Lam\Rmult^2\sigma_g^2}{\eps^2}$ & $=\Nf$ & $\Nseq$\\[3pt]
\Cref{cor:affine} & affine $g$, smooth $f$ & $\sqrt{\frac{L_0D^2}{\eps}}$ & $\frac{dM_0^2D^2}{\eps^2}$ ($=\Norc$) & $0$ & $\Nseq+\frac{\rowinf A\Rmult D\sqrt{\log m}}{\eps}$\\[3pt]
\Cref{cor:affine} & affine $g$, nonsmooth $f$ & $\frac{d^{1/4}M_0D}{\eps}$ & $\frac{dM_0^2D^2}{\eps^2}$ ($=\Norc$) & $0$ & $\Nseq+\frac{\rowinf A\Rmult D\sqrt{\log m}}{\eps}$\\
\bottomrule
\end{tabular}}
\end{table}

\begin{table}[t]
\centering\footnotesize
\caption{Upper bounds of this paper against the lower bounds of \Cref{sec:lower}. In \Cref{thm:lower-objective} the class is a ball of radius $R_x$ and $D=2R_x$; in
\Cref{thm:lower-constraint} the bound is stated directly in the diameter $D$ and the Lipschitz constant $M_g$ of the class; in all lower-bound constructions the
active multiplier equals $1$. ``Matching'' means matching up to universal constants in the listed parameters. The depth rows record what is known about $\Nseq$.}
\label{tab:lower}
\setlength{\tabcolsep}{3pt}\renewcommand{\arraystretch}{1.15}
\begin{tabular}{@{}>{\raggedright\arraybackslash}p{2.8cm}>{\raggedright\arraybackslash}p{4.3cm}>{\raggedright\arraybackslash}p{4.3cm}>{\raggedright\arraybackslash}p{4.3cm}@{}}
\toprule
Resource & Upper bound (this paper) & Lower bound & Status\\
\midrule
objective gradient, two-point calls & $dM_0^2D^2/\eps^2$ (\Cref{cor:affine}) & $dM^2D^2/\eps^2$ (\Cref{thm:lower-objective}) & matching up to constants\\
constraint gradients, two-point calls & $d\Rmult^2M_g^2D^2/\eps^2$ (\Cref{cor:sliding-nonsmooth}) & $dM_g^2D^2/\eps^2$ (\Cref{thm:lower-constraint}) & matching in $d$, $M_g$, $D$, $\eps$ at unit multipliers; factor $\Rmult^2$ open\\
constraint levels, vector oracle & $\Lam\Rmult^2\sigma_g^2/\eps^2$ (\Cref{thm:sliding,thm:mp}) & $\sigma_g^2\log(1+m)/\eps^2$ (\Cref{thm:lower-levels}, \Cref{rem:lower-levels-sc}) & one $\log$ necessary; factor $\kap\Rmult^2$ open\\
constraint levels, scalar oracle & $m\Lam\Rmult^2\sigma_g^2/\eps^2$ (one vector call $=m$ scalar calls) & $m\sigma_g^2/\eps^2$ (\Cref{thm:lower-scalar}) & linear in $m$ necessary; factor $\Lam\Rmult^2$ open\\
linearization term (\Cref{thm:sliding}) & $\Lam\Rmult^2\nuj M_g^2D^2/\eps^2$ & --- & dimension-free; below the objective term if $d\ge92\kap^2\log(m+1)$\\
\midrule
depth, smooth data & $\sqrt{(L_0+\Rmult L_g)D^2/\eps}$ (\Cref{cor:sliding-smooth}) & $\sqrt{L_0D^2/\eps}$, and $\sqrt{L_gD^2/\eps}$ at unit multipliers (\Cref{sec:depth}) & order in $\eps$ optimal; factor $\sqrt{1+\Rmult L_g/L_0}$ open\\
depth, nonsmooth data & $d^{1/4}\sqrt{M_RM}D/\eps$ (\Cref{cor:sliding-nonsmooth}) & $\tilde\Omega(\min\{\eps^{-2},d^{1/3}\eps^{-2/3}\})$, unconstrained (\Cref{sec:depth}) & best known two-point depth up to $\sqrt{M_R/M}$; not minimax\\
inner operations, affine $g$ & $\rowinf A\Rmult D\sqrt{\log(m+1)}/\eps$ (\Cref{cor:affine}) & $\Omega(1/\eps)$ for matrix--vector methods \citep{ouyang2021lower} & order in $\eps$ optimal; parameter dependence open\\
\bottomrule
\end{tabular}
\end{table}

The tables show three things. First, for smooth and nonsmooth data alike, the sliding method needs as many rounds as the unconstrained accelerated
batched methods. The constraints enter the round count only through the smoothness $H$ of the Lagrangian and the multiplier scale $\Rmult$. The coupling
constant $M_g\Rmult$ affects only the number of inner steps, which need no new rounds.

Second, the $\eps^{-2}$ call complexity cannot be avoided, and it has two parts: gradient information, with the factor $d$, and level information, with the
factor $\log(1+m)$. The objective calls $\Nf$ pay only for the first part, the constraint calls $\Ng$ for both. The difference shows up in the strongly convex
case: there $\Nf=O(\Nmv+dM_R^2/(\mu\eps))$, while $\Ng$ keeps the $\eps^{-2}$ level term, which is necessary by \Cref{rem:lower-levels-sc}.

Third, \Cref{tab:lower} shows where the lower bounds match. The objective term is tight up to constants. The constraint-gradient term is tight in all
parameters except $\Rmult^2$. The level term is tight up to $\kap\Rmult^2$; here $\kap$ comes from the $\ell_\infty$ bound on the dual noise (\Cref{lem:batch}),
and the entropy radius $\cB$ gives the necessary $\log(m+1)$. Open are the dependence on the multipliers, the extra $\kap$, the linearization term and the
nonsmooth depth (\Cref{sec:future}). For the depth and the inner operations, only the exponent of $\eps$ is known to be optimal.

\section{Numerical experiments}\label{sec:experiments}

We have three goals: to see the separation between rounds and calls predicted by \Cref{thm:sliding,thm:mp}, to test the restart scheme of
\Cref{sec:restarts}, and to illustrate two effects used in the analysis. These are the $\gamma$-independence of the common-sample two-point estimator
(\Cref{lem:twopoint}) and the $\log(1+m)$ scaling of the level information (\Cref{thm:lower-levels}). All numbers below are generated from stored raw records by
the code that accompanies the paper; implementation details and parameter values are in \Cref{app:experiments}.

\subsection{Test problems}\label{sec:exp-setup}
We take $d=24$, $m=8$ and $\Q=\Ball$, the unit ball ($D=2$). We use two objectives. The \emph{smooth} one is $f(x)=\frac12\norm{x-a}^2$ with $a=0.75\,e_1$;
it is $1$-smooth and $1$-strongly convex. The \emph{nonsmooth} ``cusp'' is $f(x)=-x_1+\frac14\norm{x_{2:d}}$, which is not differentiable on the line $x_{2:d}=0$
passing through the solution. The constraints are $g_1(x)=x_1-0.2$, which is active at the solution; $g_i(x)=\inner{A_i}{x}-0.3$ for $i=2,\dots,m-1$, with random unit
rows $A_i\perp e_1$, which are inactive; and $g_m(x)=\frac12\sum_{j=2}^{11}x_j^2-0.05$, smooth and inactive. In both problems $x^\ast=0.2\,e_1$. The optimal value is
$f^\ast=0.15125$ for the smooth objective and $f^\ast=-0.2$ for the cusp, and the only active multiplier is $y_1^\ast=0.55$ and $1$, respectively. We set $\Rmult=2$.

Both instances satisfy the assumptions of \Cref{sec:setting}. The data are convex on $\R^d$ and Lipschitz on $\Q_{\bar\gamma}$ (\Cref{ass:lip}; the constants in
\Cref{tab:params} are computed on the ball of radius $1+\bar\gamma$). The quadratic objective and $g_m$ are $1$-smooth (\Cref{ass:smooth}), and the quadratic is
$1$-strongly convex (\Cref{ass:sc}). The cusp is only Lipschitz and is covered by \Cref{conv:extension}.

The oracle returns $F(x,\xi)=f(x)+0.02\,s_0\inner{q_0}{x}$ and $G_i(x,\xi)=g_i(x)+0.02\,s_i\inner{q_i}{x}+0.02\,o_i$, where $s_0,\dots,s_m,o_1,\dots,o_m$ are independent
Rademacher variables and $q_i$ are fixed random unit vectors. The slope noise changes the Lipschitz constants of the samples (\Cref{ass:lip}). The offset noise
changes the constraint levels (\Cref{ass:levels}) and does not cancel in differences. We use smoothing radii $\gamma=0.04$ (smooth) and $\gamma=0.08$ (cusp).
Only $g_m$ needs a smoothing shift, and we use its exact bias $\gamma^2\cdot10/(2(d+2))$; the other constraints are affine.
We measure the joint error $J(x):=\max\{0,\,f(x)-f^\ast,\,\viol(x)\}$. We report the gap and the violation separately as well, since $f(x)-f^\ast$ can be negative
at infeasible points.

\subsection{Rounds versus calls}\label{sec:exp-main}
We compare \ZOS\ (\Cref{alg:sliding}) with \BSMP\ (\Cref{alg:mp}) at the \emph{same total number of vector calls} $\Norc$ (\Cref{def:measures}). In both methods
every call carries the constraint vector, so $\Norc=\Ng$, while $\Nf$ is smaller: $2b_pK_N$ and $4bN$, respectively.

\ZOS\ uses $N=60$ rounds, mini-batches $b_p=b_y=8$, dual scale $\rs=1$ and stabilizer $\lambda=0.1$. The inner schedule is \eqref{eq:schedule},
$S_t=\max\{1,\lceil M_g\rs t/L\rceil\}$, with the same $M_g$ and $L$ as in the theory: $M_g=1.06$, $L=3$ for the smooth instance and $M_g=1.10$,
$L=2+\sqrt d/(4\gamma)\approx17.3$ for the cusp (see \Cref{app:experiments}). With these values the condition $a_tM_g\rs\le L$ of \Cref{lem:schedule} holds in
every phase. This gives \SmoothCalls\ calls and \SmoothInner\ inner steps on the smooth instance and \CuspCalls\ calls and \CuspInner\ inner steps on the cusp.
We also run \ZOS\ with the worst-case stabilizer of \Cref{cor:sliding-lambda} ($\lambda=\SmoothStabilizer$ and $\lambda=\CuspStabilizer$).
\BSMP\ uses $\alpha=\sqrt{20}$, $\eta=0.45/L_\alpha$ and mini-batches $b\in\{1,8\}$, and we choose its number of iterations to match the budget of \ZOS.
We did not tune either method. The point is to see the separation between the three resources, not to find the best performance of each method.
Every configuration is run with \NumSeeds\ seeds. We report means and percentile bootstrap confidence intervals for the mean final joint error
(\Cref{tab:results} and \Cref{fig:main}).

\begin{table}[t]
\centering\footnotesize\setlength{\tabcolsep}{4pt}
\caption{Final results at equal budgets (\NumSeeds\ seeds). $\Norc$ is the number of vector calls (for both methods every call carries the constraint vector, so
$\Ng=\Norc$), $\Nf$ the number of calls whose objective component is used ($2b_pK_N$ for \ZOS, $4bN$ for \BSMP), $\Nseq$ the number of sequential oracle rounds. Means over seeds of the objective gap,
the maximal violation and the joint error $J$ at the last iterate; the last column is a $95\%$ percentile bootstrap confidence interval for the mean of $J$.}
\label{tab:results}
\begin{tabular}{@{}llrrrrrrr@{}}\toprule
Instance & Method & $N_{\mathrm{orc}}$ & $N_f$ & $N_{\mathrm{seq}}$ & $f$-gap & Violation & $J$ & 95\% CI for $J$\\\midrule
Smooth & ZO-Sliding, $b=8$, $\lambda=0.1$ & 28,064 & 10,848 & 60 & 0.0016 & 0.0000 & 0.0016 & [0.0013, 0.0020] \\
Smooth & ZO-Sliding, $b=8$, worst-case $\lambda$ & 28,064 & 10,848 & 60 & -0.0632 & 0.3355 & 0.3355 & [0.3341, 0.3369] \\
Smooth & B-SMP, $b=1$ & 28,062 & 18,708 & 9354 & -0.0403 & 0.0806 & 0.0806 & [0.0775, 0.0838] \\
Smooth & B-SMP, $b=8$ & 28,032 & 18,688 & 1168 & 0.0014 & 0.0000 & 0.0014 & [0.0013, 0.0015] \\
\midrule
Cusp & ZO-Sliding, $b=8$, $\lambda=0.1$ & 6,824 & 2,352 & 60 & 0.0078 & 0.0093 & 0.0228 & [0.0186, 0.0275] \\
Cusp & ZO-Sliding, $b=8$, worst-case $\lambda$ & 6,824 & 2,352 & 60 & -0.7021 & 0.7422 & 0.7422 & [0.7409, 0.7434] \\
Cusp & B-SMP, $b=1$ & 6,822 & 4,548 & 2274 & 0.0363 & 0.0029 & 0.0363 & [0.0324, 0.0405] \\
Cusp & B-SMP, $b=8$ & 6,816 & 4,544 & 284 & 0.0856 & 0.0000 & 0.0856 & [0.0759, 0.0950] \\
\bottomrule\end{tabular}

\end{table}

\begin{figure}[t]
\centering
\includegraphics[width=\textwidth]{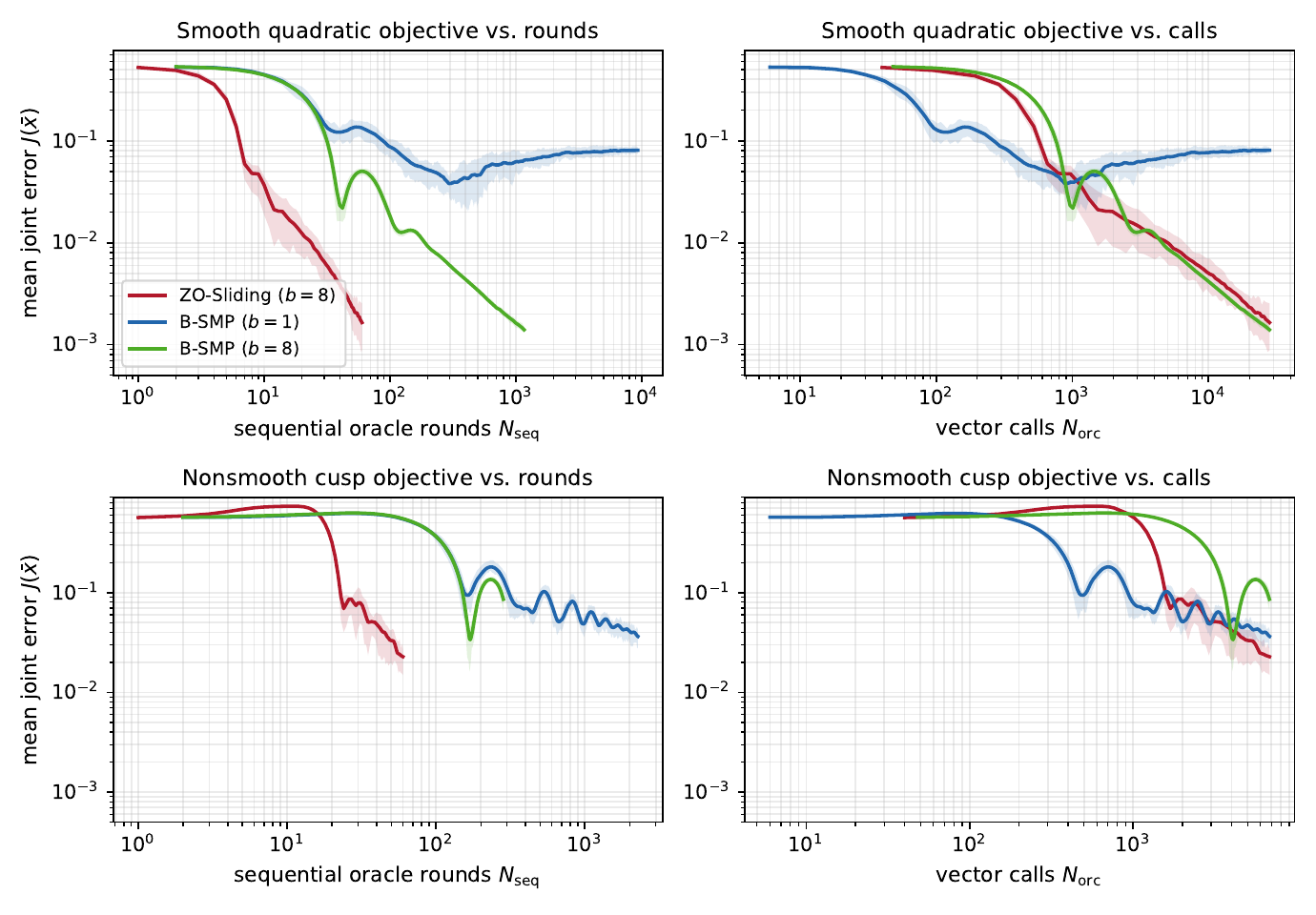}
\caption{Mean joint error against sequential rounds (left) and against vector calls $\Norc$ (right), smooth instance (top) and cusp instance (bottom); bands are the
interquartile range over \NumSeeds\ seeds. On the smooth instance \BSMP\ with $b=8$ (green) needs \SmoothBsmpRoundsToZO\ rounds to reach the final mean error of \ZOS\
(red), which uses $60$; on the cusp instance \ZOS\ reaches the final mean error of \BSMP\ with $b=8$ after \CuspZORoundsToBsmpEight\ rounds (against
\CuspBsmpEightRounds) and that of \BSMP\ with $b=1$ (blue) after \CuspZORoundsToBsmpOne\ rounds (against \CuspBsmpOneRounds). \BSMP\ with $b=1$ stalls on the
smooth instance at the level of its dual noise.}
\label{fig:main}
\end{figure}

The results agree with the theory. On the smooth instance, \ZOS\ and \BSMP\ with $b=8$ end at about the same accuracy: $J\approx\SmoothZOJoint$ (95\% interval
\SmoothZOCI) against \SmoothBsmpEightJoint\ (\SmoothBsmpEightCI). The difference is within sampling error. But \ZOS\ uses $60$ rounds, while \BSMP\ uses
\SmoothBsmpEightRounds\ and needs \SmoothBsmpRoundsToZO\ of them just to reach the final error of \ZOS. The reason is the inner loop: \ZOS\ makes
\SmoothInner\ steps without new oracle rounds, and a parallel implementation of \BSMP\ would have to spend each of them as a round.

\BSMP\ with $b=1$ runs for \SmoothBsmpOneRounds\ rounds and stalls at $J\approx\SmoothBsmpOneJoint$, with the dual iterates oscillating around the active constraint.
This fits the bound of \Cref{lem:batch}, which gives no reduction of the $\ell_\infty$ noise for $b<8\kap$, although the bound does not predict the stall.
\ZOS\ also has $b=8<8\kap$, but it additionally averages the dual noise over the $S_t$ slots of a phase and over phases through the weights $a_k$; this is why
\eqref{eq:sliding-rate} depends on $bK_N$.

On the cusp instance the smoothed Lagrangian has a large smoothness constant, $L\approx17.3$ (\Cref{lem:smoothing}(c)). With this budget \BSMP\ with $b=8$ gets
only \CuspBsmpEightIters\ iterations and stays far from the solution ($J\approx\CuspBsmpEightJoint$). \BSMP\ with $b=1$ reaches \CuspBsmpOneJoint, and \ZOS\ reaches
\CuspZOJoint\ in $60$ rounds. The remaining error reflects the noise level at this budget: the constraint offsets have standard deviation $0.02$, and the
solution lies on the active constraint. It keeps decreasing with the budget (\Cref{fig:main}, right); we see no floor in the data.

With the worst-case stabilizer, $\lambda\approx10^3$, the inner steps barely move, since $\beta_t=(L+\lambda)/a_t$. The iterates stay near $x_0=a$ and the output is
infeasible: $J\approx\SmoothWorstJoint$ (smooth) and $J\approx\CuspWorstJoint$ (cusp). \Cref{thm:sliding} holds for every $\lambda>0$, and in practice the variance is
far smaller than the worst-case value $b\,\Sb^2=\SA^2\approx3\cdot10^4$ that sets this stabilizer.

\subsection{Restarts on the strongly convex instance}\label{sec:exp-restart}
\Cref{fig:restarts} shows \RZOS\ (\Cref{alg:restart} with \Cref{alg:sliding} as the base method) on the smooth instance, $\mu=1$. It runs $7$ stages of $10$ rounds
with mini-batches $b_{p,k}=b_{y,k}=2^{k-1}$ and dual scale $\rs_k=2^{(k-1)/2}$, so the number of inner steps $S_t=\max\{1,\lceil M_g\rs_kt/L\rceil\}$ grows from stage to
stage. The dual variable is reset at each stage. We compare it with the plain method with $N=60$ rounds and mini-batches $b=8$, $b=16$ and $b=\EqBatch$. The last
choice gives the plain run the same number of vector calls as the restarted one, up to rounding, so that comparison is at \emph{equal budget}.

\begin{figure}[t]
\centering
\includegraphics[width=.8\textwidth]{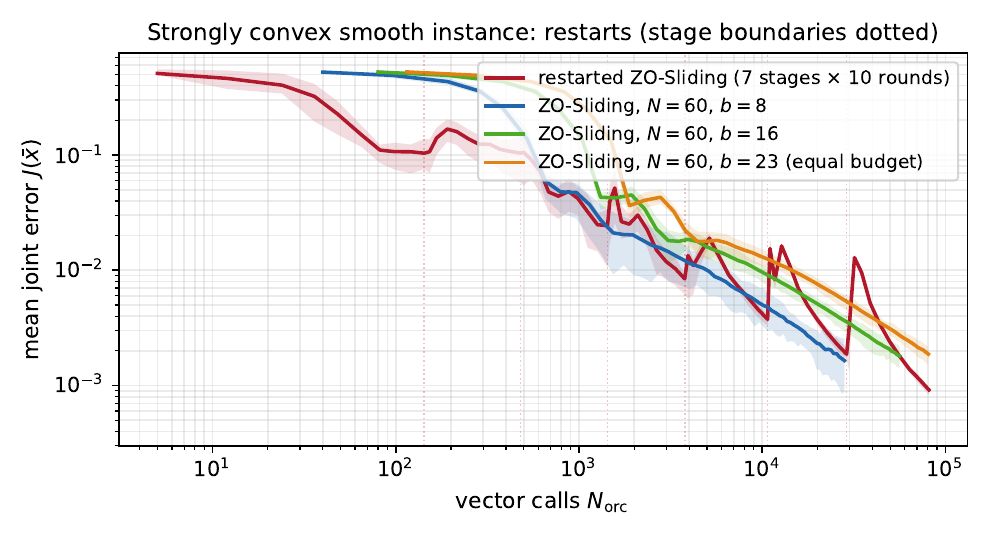}
\caption{Restarted versus plain \ZOS\ on the strongly convex smooth instance; mean joint error against vector calls, interquartile bands over \NumSeeds\ seeds. Dotted
vertical lines mark the stage boundaries of the restarted method; the plain run with $b=\EqBatch$ has the same budget as the restarted run.}
\label{fig:restarts}
\end{figure}

The restarted method reaches $J=\RestartJoint$ (\RestartCI) after \RestartCalls\ calls and \RestartRounds\ rounds. The plain method reaches $J=\PlainEightJoint$
(\PlainEightCI) with $b=8$ (\PlainEightCalls\ calls), $J=\PlainSixteenJoint$ (\PlainSixteenCI) with $b=16$ (\PlainSixteenCalls\ calls) and $J=\PlainEqJoint$ (\PlainEqCI)
with $b=\EqBatch$ (\PlainEqCalls\ calls). Larger mini-batches do not help the plain method. At $N=60$ rounds it is limited by its deterministic part and its fixed
dual scale, not by the number of samples. At equal budget the restarted method has about half the error, and the confidence intervals do not overlap. The gain
therefore comes from the schedule, i.e.\ the shrinking primal radius and the growing dual scale and inner loop, as \Cref{thm:restart-sliding} predicts.

The saw-tooth shape of the restarted curve comes from the two resets at each stage boundary. The dual variable returns to the uniform point and the average
$\Xr$ restarts from the last stage output, so the first rounds of a stage are dominated by the dual transient. The gain from restarting is moderate here because
at these budgets the level-noise term $\Lam\Rmult^2\sigma_h^2/\eps^2$, which strong convexity does not improve, is already comparable to the primal term.

\subsection{Two diagnostics}\label{sec:exp-diag}
\Cref{fig:noise} checks \Cref{lem:twopoint}. We estimate the second moment of the two-point estimator $\hat g_{y}(x)$ of the Lagrangian gradient at $x=x^\ast$, $y=y^\ast$.
When both points share the sample $\xi$, it is flat in $\gamma$ over two orders of magnitude. With independent measurement noise of standard deviation $0.01$ at the
two points, it grows like $\gamma^{-2}$; one would then have to keep $\gamma$ large, and the bias $\gamma M$ would dominate.

\Cref{fig:levels} illustrates \Cref{thm:lower-levels}. We take the testing problem from its proof: one of $m$ constraint levels is shifted by $-2a$, with $a=0.04$ and
$\sigma=0.2$. We plot the error of the likelihood-ratio test between the null $\beta^{(0)}$ and the uniform mixture of the $m$ alternatives against the normalized budget
$Ta^2/(\sigma^2\log(1+m))$. The curves for $m\in\{4,16,64,256\}$ coincide and fall below $10\%$ only once the normalized budget exceeds $1$. So the number of vector
calls needed for a fixed testing error grows like $\sigma^2\log(1+m)/a^2$; this is the origin of the $\log(m+1)$ factors in the dual geometry.

\begin{figure}[t]
\centering
\begin{subfigure}[t]{.49\textwidth}\centering\includegraphics[width=\textwidth]{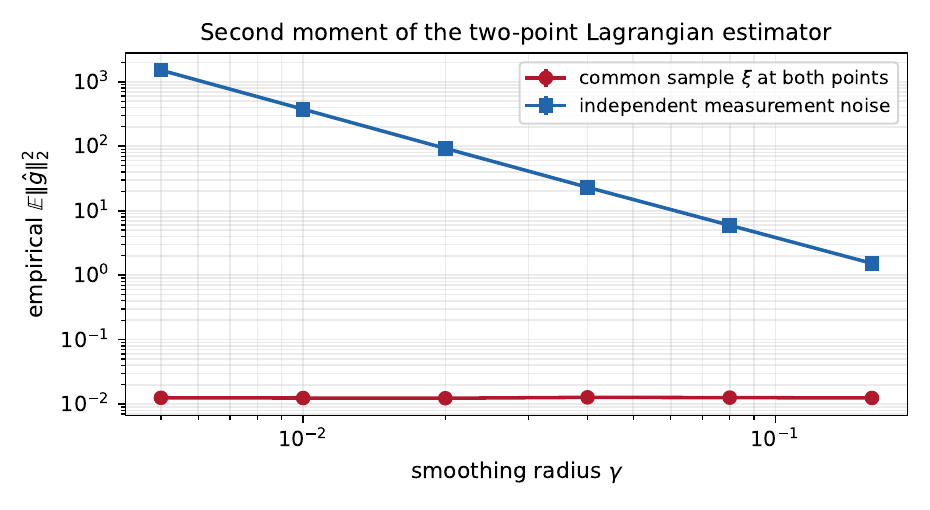}\caption{Common sample versus independent noise.}\label{fig:noise}\end{subfigure}\hfill
\begin{subfigure}[t]{.49\textwidth}\centering\includegraphics[width=\textwidth]{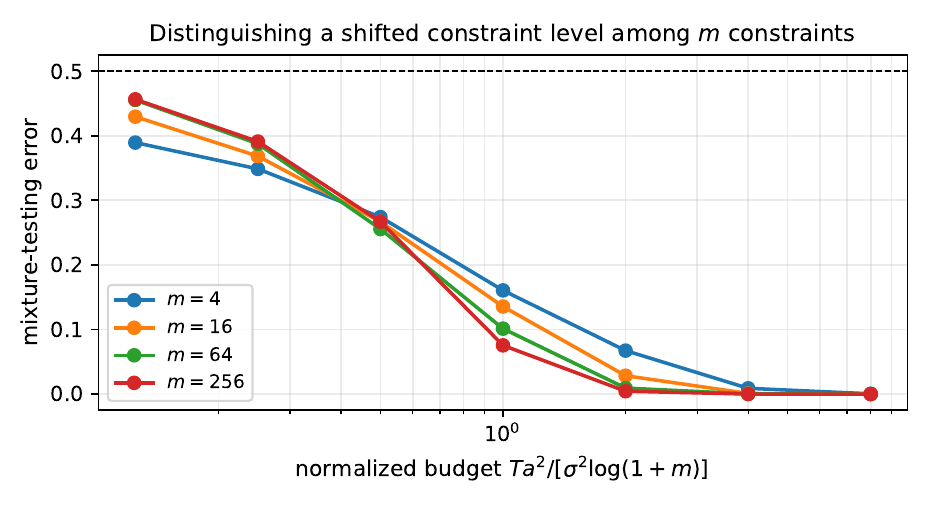}\caption{Level information and $\log(1+m)$.}\label{fig:levels}\end{subfigure}
\caption{Monte Carlo diagnostics ($12\,000$ samples per point in (a), $4\,000$ trials per point in (b)).}
\end{figure}

\section{Conclusion and open problems}\label{sec:conclusion}

We studied two-point zeroth-order methods for convex stochastic optimization with stochastic functional constraints, counting separately the sequential
rounds, the oracle calls (split between objective and constraints) and the inner operations that need no new round.

Smoothing with an explicit shift of the constraint levels reduces the problem to a smooth saddle-point problem. On this problem, batched stochastic
Mirror--Prox needs $O(\eps^{-1})$ rounds for smooth data. The zeroth-order sliding method brings the number of rounds down to that of the best unconstrained
methods: $O(\sqrt{HD^2/\eps})$ for smooth data, which is optimal in $\eps$, and $O(d^{1/4}\sqrt{M_RM}D/\eps)$ for nonsmooth data, which is the best known depth for the
two-point model. It still uses $\tilde O(\eps^{-2})$ calls, and the coupling is paid for in inner steps only. With separate primal and dual mini-batches, the factor
$d$ is charged to the objective calls and the factor $\log(1+m)$ to the constraint calls. For strongly convex objectives, restarts give
$O(\sqrt{L/\mu}\log(1/\eps))$ rounds and $O(dM_R^2/(\mu\eps))$ objective calls, while the level term stays at $\eps^{-2}$, as the lower bound of
\Cref{rem:lower-levels-sc} requires. For known affine constraints, the constraints cost no calls and $O(\eps^{-1})$ matrix--vector products.

Our lower bounds show that the $\eps^{-2}$ call terms cannot be avoided, with the factor $d$ tied to gradient information and $\log(1+m)$ to level information.
\Cref{tab:lower} shows which of these bounds are tight. The gaps that remain lead to the questions below.

\subsection{Open problems}\label{sec:future}

\begin{problem}[Slater-free variant through a level-set reformulation]\label{conj:levelset}
Let $\ell(\vartheta):=\min_{x\in\Q}\max\{f_\gamma(x)-\vartheta,\ h_1(x),\dots,h_m(x)\}$; its root is $f^\ast_\gamma$. Suppose we apply a zeroth-order sliding method to the inner
minimax problem at fixed $\vartheta$ and combine it with root-finding in $\vartheta$ \citep{aravkin2019level,lin2018level,lin2020data,boob2025level}. Does this give the
round complexity of \Cref{thm:sliding} without \Cref{ass:slater}, with $\Rmult$ replaced by a constant that depends only on a condition measure of $\ell$? Two things make
this hard. The inner problem must be solved to an accuracy that depends on the unknown slope of $\ell$ at its root, and the zeroth-order estimates of the level
values enter the root-finding test with noise that does not vanish.
\end{problem}

\begin{problem}[dual noise without the diameter term]\label{conj:dualnoise}
Can the term $\Lam\Rmult^2\nuj M_g^2D^2/\eps^2$ in the call complexity of \Cref{thm:sliding,thm:restart-sliding} be replaced by
$\Lam\Rmult^2\sigma_h^2/\eps^2+\text{lower order}$ without increasing the number of rounds? One idea is to estimate the linearized constraint value at $p_k$ not by
$\hat J(u_t)(p_k-u_t)$ but by a single query at a point announced one round ahead, i.e.\ a prediction of $p_k$. Controlling the resulting bias needs new tools.
In the strongly convex case, another option is to run stage $k$ of \Cref{alg:restart} on $\Q\cap B(\bar x^{(k-1)},r_k)$ with $r_k^2=O(\eps_{k-1}/\mu)$. This would replace
$D^2$ by $r_k^2$ in \eqref{eq:sigy} and bring the term down to $\Lam\Rmult^2\nuj M_g^2/(\mu\eps)$, the order of the primal term. It requires
$x^\ast_\gamma\in B(\bar x^{(k-1)},r_k)$ with high probability, not just in expectation.
\end{problem}

\begin{problem}[depth lower bounds with constraints]\label{op:depth}
Must the coupling constant $M_g\Rmult$ appear in the number of \emph{rounds} of a zeroth-order method for \eqref{eq:P}, or can it be confined to the inner
operations, as in \Cref{thm:sliding}? Is the factor $\sqrt{1+\Rmult L_g/L_0}$ in the depth of \Cref{cor:sliding-smooth} necessary beyond the unit-multiplier case of
\Cref{sec:depth}? Lower bounds for parallel randomized algorithms \citep{bubeck2019complexity,diakonikolas2020lower} exist only for unconstrained problems.
A related question: can the ball-acceleration methods with depth $\tilde O(d^{1/3}\eps^{-2/3})$ \citep{bubeck2019complexity,carmon2023resqueing} be run with a two-point
value oracle? This would improve the nonsmooth depth of \Cref{cor:sliding-nonsmooth} for $\eps\lesssim d^{-1/4}$.
\end{problem}

\begin{problem}[optimality of the multiplier factor]\label{op:multiplier}
Our upper bounds have a factor $\Rmult^2$ in front of the $\eps^{-2}$ terms, while in the lower-bound constructions of \Cref{thm:lower-constraint,thm:lower-levels} the
active multiplier is $1$. What is the minimax dependence on the size of the multipliers, or on the Slater margin $\rho$? Two points need care. First, criterion
\eqref{eq:goal} is not invariant under rescaling $g\mapsto g/\kappa$: this multiplies the multipliers by $\kappa$, divides $M_g$ and $\sigma_g$ by $\kappa$, and relaxes the
violation requirement. Any statement about $\Rmult$ must therefore fix the normalization of $g$ through $M_g$ and $\sigma_g$. Second, part of the gap is between the a
priori bound $\Rmult=1+M_0D/\rho$ of \eqref{eq:R} and the actual size $1+\norm{y^\ast_\gamma}_1$ of the multipliers, so the question is partly about adapting to an
unknown multiplier scale.
\end{problem}

\begin{problem}[high probability and heavy tails]\label{op:hp}
All our guarantees hold in expectation. Running $O(\log(1/\delta))$ independent copies and keeping the best gives bounds in probability, but choosing the best
copy requires estimates of $f$ and $\viol$ at the candidates. That means extra single queries with noise variance $\sigma_g^2$, and a similar assumption on the
values of $F$, which we did not need so far. A direct high-probability analysis under light-tailed noise, and an analysis under heavy-tailed noise with gradient
clipping \citep{kornilov2023accelerated}, are open in the constrained setting.
\end{problem}

\begin{problem}[adversarial noise and one-point feedback]\label{op:adv}
For unconstrained problems, \citet{gasnikov2022power} showed that the accelerated batched scheme tolerates bounded adversarial noise of level
$\delta=O(\eps^2/(d^{1/2}MD))$ up to logarithms, and treated one-point feedback with $O(d^2/\eps^{3})$ or worse call complexity. Extending this to constraints requires
understanding how the bias $\delta/\gamma$ of the Jacobian estimates propagates through the dual dynamics.
\end{problem}

\begin{problem}[sharper constants and the factor $\kap$]\label{op:constants}
What are the sharp constants in \Cref{thm:sliding} (instead of $17$, $16\kap$, $92\kap$) and in the restart theorems? The factor $\kap$ in the level term comes only
from the worst-case $\ell_\infty$ bound of \Cref{lem:batch} on the dual noise. A local-norm analysis of the entropic dual step, as for exponential weights, where the
residual of a step is $\frac1{2\tau}\sum_i\tilde y_ie_i^2$ instead of $\frac1{2\tau}\norm e_\infty^2$, would replace $\E\norm{e}_\infty^2$ by $\Rmult\max_i\E e_i^2$. This would
remove $\kap$ from the level term and leave only the $\log(m+1)$ of the entropy radius. It needs control of $\max_i|\hat q_{k,i}|/\tau$, for instance by clipping the
level estimates or by assuming bounded noise.
\end{problem}

\section*{Acknowledgements}
\paragraph{Use of AI tools.} Large language models were used in preparing this paper. In particular, the text of the manuscript
was drafted and substantially edited with the help of AI assistants, including Claude (Anthropic), which was used to revise the
wording of the abstract, the introduction, the experimental section, the conclusion and parts of the main text for readability.
All mathematical statements, proofs, code and numerical results were checked by the authors, who take full responsibility for the
content of the paper.

\bibliographystyle{abbrvnat}
\bibliography{refs}

\appendix
\section{Auxiliary facts}\label{app:aux}

\begin{fact}[Poincar\'e inequality on the sphere]\label{fact:poincare}
Let $d\ge2$, $u\sim U(\Sphere)$ and $q:\Sphere\to\R$ be $\Lambda$-Lipschitz with respect to the Euclidean distance of $\R^d$. Then
$\Var q(u)=\E(q(u)-\E q(u))^2\le\Lambda^2/(d-1)$.
\end{fact}
\begin{proof}[Reference]
The uniform probability measure on $\Sphere=\Sphere^{d-1}$ satisfies the Poincar\'e inequality $\Var(q)\le\frac1{d-1}\E\norm{\nabla_{\Sphere}q}^2$, where $d-1$ is
the first nonzero eigenvalue of the Laplace--Beltrami operator (\eg\ \citealp[Sect.~2.1 and Prop.~2.10]{ledoux2001concentration},
\citealp[Sect.~4.5]{bakry2014analysis}). A function that is $\Lambda$-Lipschitz for the Euclidean (chordal) distance is $\Lambda$-Lipschitz for the geodesic
distance, hence $\norm{\nabla_{\Sphere}q}\le\Lambda$ almost everywhere by Rademacher's theorem.
\end{proof}

\begin{fact}[Gaussian concentration]\label{fact:gaussconc}
Let $\mathsf g\sim\mathcal N(0,I_d)$ and $\mathsf Q:\R^d\to\R$ be $\Lambda$-Lipschitz. Then for all $t\ge0$,
$\Prob(\mathsf Q(\mathsf g)-\E\mathsf Q(\mathsf g)\ge t)\le e^{-t^2/(2\Lambda^2)}$ \citep[Thm.~5.6]{boucheron2013concentration}.
\end{fact}

\begin{fact}[moments on the sphere and for Gaussians]\label{fact:moments}
Let $u\sim U(\Sphere)$, $\mathsf g\sim\mathcal N(0,I_d)$ and $w\in\R^d$. Then $\E\inner uw^2=\norm w^2/d$, $\E\inner uw^4=3\norm w^4/(d(d+2))$,
$\E\norm{\mathsf g}^2=d$, $\E\norm{\mathsf g}^4=d(d+2)$, and $\mathsf g/\norm{\mathsf g}\sim U(\Sphere)$ is independent of $\norm{\mathsf g}$.
\end{fact}
\begin{proof}
By rotation invariance $\E\inner uw^{2k}=\norm w^{2k}\E u_1^{2k}$, and $\sum_i\E u_i^2=1$ gives $\E u_1^2=1/d$. Since $\mathsf g=\norm{\mathsf g}\,u$ with independent
factors, $\E\mathsf g_1^4=\E\norm{\mathsf g}^4\,\E u_1^4$, and $\E\mathsf g_1^4=3$, $\E\norm{\mathsf g}^4=\Var(\chi^2_d)+d^2=2d+d^2$ give $\E u_1^4=3/(d(d+2))$.
\end{proof}

\begin{fact}[Pinsker's inequality and $\chi^2$]\label{fact:pinsker}
For probability measures $P,Q$ on a common space, $\TV(P,Q):=\sup_A|P(A)-Q(A)|\le\sqrt{\KL(P\|Q)/2}$ and $\TV(P,Q)\le\frac12\sqrt{\chi^2(P\|Q)}$, where
$\chi^2(P\|Q)=\E_Q[(dP/dQ)^2]-1$. For probability vectors $p,q\in\R^{n}$, $\TV=\frac12\norm{p-q}_1$, hence $\KL(p\|q)\ge\frac12\norm{p-q}_1^2$
\citep[Lemma~2.5]{tsybakov2009introduction}.
\end{fact}

\begin{fact}[chain rule and data processing for adaptive experiments]\label{fact:chain}
Consider an adaptive algorithm that, for $t=1,\dots,T$, chooses a query $X_t$ as a measurable function of the previous observations $O_1,\dots,O_{t-1}$ and
of an internal random seed $\omega$, and receives an observation $O_t$ drawn from a distribution $P_{X_t}$ (respectively $P'_{X_t}$) independently of the past
given $X_t$. Let $P$, $P'$ be the laws of the transcript $(\omega,X_1,O_1,\dots,X_T,O_T)$ under the two families. Then
\[
 \KL(P\|P')=\sum_{t=1}^T\E_P\bigl[\KL(P_{X_t}\|P'_{X_t})\bigr],
\]
and for every measurable function $\psi$ of the transcript, $\KL(\mathrm{law}_P(\psi)\|\mathrm{law}_{P'}(\psi))\le\KL(P\|P')$ and
$\TV(\mathrm{law}_P(\psi),\mathrm{law}_{P'}(\psi))\le\TV(P,P')$ \citep[Sect.~2.4]{tsybakov2009introduction}. Moreover, if $P_{X_t}=P'_{X_t}$ does not depend
on $X_t$ (the observations are i.i.d.\ regardless of the queries), the likelihood ratio $dP/dP'$ is the product of the likelihood ratios of the observations.
\end{fact}
\begin{proof}
The chain rule follows from the factorization of the transcript density as $p(\omega)\prod_tp_{X_t}(O_t)$ (the query $X_t$ being a deterministic function of the
past, its conditional law is a point mass under both hypotheses) and the tower property; the data-processing inequalities are the monotonicity of
$f$-divergences under measurable maps.
\end{proof}

\begin{fact}[three-point inequality]\label{fact:threepoint}
Let $Z\subseteq Z'$ be closed convex sets, $\omega$ a convex function on $Z'$ that is $1$-strongly convex with respect to a norm $\norm\cdot_Z$ on $Z'$ and
differentiable on an open set containing the relative interior $Z'^\circ$ of $Z'$ (or the whole $Z'$), and $V(z,z')=\omega(z')-\omega(z)-\inner{\nabla\omega(z)}{z'-z}$. For
$\bar z\in Z'^\circ$ (not necessarily in $Z$), a vector $\xi$ and $\tau>0$, let $z^+:=\argmin_{z\in Z}\{\inner\xi z+\tau V(\bar z,z)\}$ and assume $z^+\in Z'^\circ$ (this holds automatically for the Euclidean distance and for the entropy
on the simplex, where the minimizer has positive coordinates, see the explicit formula of \Cref{lem:prox}(c)). Then for all $z\in Z$,
\[
 \inner{\xi}{z^+-z}\le\tau\bigl[V(\bar z,z)-V(z^+,z)-V(\bar z,z^+)\bigr],\qquad V(\bar z,z^+)\ge\tfrac12\norm{z^+-\bar z}_Z^2 .
\]
\end{fact}
\begin{proof}
The optimality condition of the convex problem defining $z^+$ is
\[
 \inner{\xi+\tau(\nabla\omega(z^+)-\nabla\omega(\bar z))}{z-z^+}\ge0\quad\text{for all }z\in Z,
\]
and the identity $\inner{\nabla\omega(z^+)-\nabla\omega(\bar z)}{z-z^+}=V(\bar z,z)-V(z^+,z)-V(\bar z,z^+)$ is verified by expanding the definitions. Strong convexity gives the last claim.
\end{proof}

\begin{fact}[conjugates of smooth convex functions]\label{fact:conjugate}
Let $\mathsf F:\R^d\to\R$ be convex and $L$-smooth, and $\mathsf F^\ast(\pi):=\sup_x\{\inner\pi x-\mathsf F(x)\}$ its Fenchel conjugate. For every $x'\in\R^d$, with
$g':=\nabla\mathsf F(x')$, we have $\mathsf F^\ast(g')=\inner{g'}{x'}-\mathsf F(x')$ (Fenchel--Young equality) and, for every $\pi\in\dom\mathsf F^\ast$,
\[
 \mathsf F^\ast(\pi)\ \ge\ \mathsf F^\ast(g')+\inner{x'}{\pi-g'}+\frac{1}{2L}\norm{\pi-g'}^2 .
\]
\end{fact}
\begin{proof}
The equality is the first-order optimality condition of the concave maximization defining $\mathsf F^\ast(g')$, attained at $x'$. By $L$-smoothness,
$\mathsf F(x)\le\mathsf F(x')+\inner{g'}{x-x'}+\frac L2\norm{x-x'}^2$ for all $x$, hence
\begin{align*}
\mathsf F^\ast(\pi)&\ge\sup_x\Bigl\{\inner\pi x-\mathsf F(x')-\inner{g'}{x-x'}-\frac L2\norm{x-x'}^2\Bigr\}\\
 &=\inner\pi{x'}-\mathsf F(x')+\sup_x\Bigl\{\inner{\pi-g'}{x-x'}-\frac L2\norm{x-x'}^2\Bigr\}\\
 &=\mathsf F^\ast(g')+\inner{x'}{\pi-g'}+\frac{1}{2L}\norm{\pi-g'}^2 .
\end{align*}
\end{proof}

\section{Proofs for Section~\ref{sec:smoothing}}\label{app:smoothing}

\begin{proof}[Proof of \Cref{lem:smoothing}]
(a) Convexity is inherited by averaging: $\varphi_\gamma(\theta x+(1-\theta)x')=\E\varphi(\theta(x+\gamma v)+(1-\theta)(x'+\gamma v))\le\theta\varphi_\gamma(x)+(1-\theta)\varphi_\gamma(x')$.
If $\varphi$ is $M$-Lipschitz on $C_\gamma$ and $x,x'\in C$, then $x+\gamma v,x'+\gamma v\in C_\gamma$ and $|\varphi_\gamma(x)-\varphi_\gamma(x')|\le\E|\varphi(x+\gamma v)-\varphi(x'+\gamma v)|\le M\norm{x-x'}$.
Since $\E v=0$, Jensen's inequality gives $\varphi(x)=\varphi(\E(x+\gamma v))\le\E\varphi(x+\gamma v)=\varphi_\gamma(x)$ for every $x$;
and for $x\in C$, $\varphi(x+\gamma v)\le\varphi(x)+\gamma M\norm v\le\varphi(x)+\gamma M$ because the segment $[x,x+\gamma v]$ lies in $C_\gamma$. If $\varphi$ is $L$-smooth,
$\varphi(x+\gamma v)\le\varphi(x)+\gamma\inner{\nabla\varphi(x)}v+\frac L2\gamma^2\norm v^2$, and taking expectations with
$\E v=0$ and $\E\norm v^2=d/(d+2)$ (for $v\sim U(\Ball)$, $\Prob(\norm v\le r)=r^d$, so $\E\norm v^2=\int_0^1r^2\,d(r^d)=d/(d+2)$) gives the second bound. Affine functions are their own averages.

(b) Let first $\psi\in C^1(\R^d)$. Then $\psi_\gamma(x)=|\gamma\Ball|^{-1}\int_{\gamma\Ball}\psi(x+z)\,dz$ is $C^1$ (differentiation under the integral sign over the compact ball, the
gradient being continuous), and applying the divergence theorem on the ball $\gamma\Ball$ (outer unit normal $z/\gamma$) to the $C^1$ vector fields $\psi(x+\cdot)e_i$,
\begin{multline*}
 \nabla\psi_\gamma(x)=\frac1{|\gamma\Ball|}\int_{\gamma\Ball}\nabla\psi(x+z)\,dz=\frac1{|\gamma\Ball|}\int_{\gamma\Sphere}\psi(x+z)\frac z\gamma\,d\sigma(z)\\
 =\frac{|\gamma\Sphere|}{|\gamma\Ball|}\cdot\frac1\gamma\,\E_{u\sim U(\Sphere)}[\psi(x+\gamma u)u]=\frac d\gamma\,\E[\psi(x+\gamma u)u],
\end{multline*}
because $|\gamma\Sphere|/|\gamma\Ball|=d/\gamma$ (the $(d-1)$-dimensional measure of $\gamma\Sphere$ is $d\gamma^{d-1}|\Ball|$). This settles the case of an $L$-smooth $\varphi$.
For a general $M$-Lipschitz convex $\varphi$, let
$\varphi^\delta:=\varphi\ast\rho_\delta$ be the convolution with a smooth probability density supported in $\delta\Ball$; then $\varphi^\delta\in C^\infty$ is $M$-Lipschitz and
$\sup_x|\varphi^\delta(x)-\varphi(x)|\le M\delta$. Consequently $\sup_x|(\varphi^\delta)_\gamma(x)-\varphi_\gamma(x)|\le M\delta$ and
$\sup_x\norm{\nabla(\varphi^\delta)_\gamma(x)-\frac d\gamma\E[\varphi(x+\gamma u)u]}\le\frac d\gamma M\delta$. Thus $(\varphi^\delta)_\gamma\to\varphi_\gamma$ and $\nabla(\varphi^\delta)_\gamma\to G(x):=\frac d\gamma\E[\varphi(x+\gamma u)u]$
uniformly on $\R^d$ as $\delta\to0$, where $G$ is continuous (indeed $M d$-Lipschitz, since $\norm{G(x)-G(x')}\le\frac d\gamma\E|\varphi(x+\gamma u)-\varphi(x'+\gamma u)|\le\frac{dM}\gamma\norm{x-x'}$).
Uniform convergence of $C^1$ functions together with uniform convergence of their gradients to a continuous limit implies that the limit is $C^1$ with the limit
gradient: for every $x,h$, $(\varphi^\delta)_\gamma(x+h)-(\varphi^\delta)_\gamma(x)=\int_0^1\inner{\nabla(\varphi^\delta)_\gamma(x+sh)}{h}ds\to\int_0^1\inner{G(x+sh)}hds$, so
$\varphi_\gamma(x+h)-\varphi_\gamma(x)=\int_0^1\inner{G(x+sh)}hds=\inner{G(x)}h+o(\norm h)$ by continuity of $G$. The symmetric form follows from $u\mapsto-u$ invariance of $U(\Sphere)$.

(c) For $x,x'$ and any unit vector $e$, by \eqref{eq:gradformula}, $\inner{\nabla\varphi_\gamma(x)-\nabla\varphi_\gamma(x')}{e}=\frac d\gamma\E[(\varphi(x+\gamma u)-\varphi(x'+\gamma u))\inner ue]$, whose absolute
value is at most $\frac d\gamma M\norm{x-x'}\,\E|\inner ue|\le\frac d\gamma M\norm{x-x'}\sqrt{\E\inner ue^2}=\frac{\sqrt d M}{\gamma}\norm{x-x'}$ by \Cref{fact:moments}. Taking the supremum over
$e$ gives the Lipschitz constant $\sqrt dM/\gamma$. Now let $\varphi$ be $L$-smooth. On the compact ball $x+\gamma\Ball$ we have $\norm{\nabla\varphi(x+\gamma v)}\le\norm{\nabla\varphi(x)}+L\gamma$, so by dominated
convergence we may differentiate under the expectation; no Lipschitz bound on $\varphi$ is needed. This gives $\nabla\varphi_\gamma(x)=\E\nabla\varphi(x+\gamma v)$,
which is $L$-Lipschitz.

(d) If $\psi:=\varphi-\frac\mu2\norm\cdot^2$ is convex on $C_\gamma$, then for $x,x'\in C$ and $\theta\in[0,1]$ the points $x+\gamma v,x'+\gamma v$ and their convex combination lie in
$C_\gamma$, so $\psi_\gamma(x):=\E\psi(x+\gamma v)$ is convex on $C$; and $\varphi_\gamma(x)=\psi_\gamma(x)+\frac\mu2\E\norm{x+\gamma v}^2=\psi_\gamma(x)+\frac\mu2\norm x^2+\frac\mu2\gamma^2\E\norm v^2$, so
$\varphi_\gamma-\frac\mu2\norm\cdot^2$ equals on $C$ the convex function $\psi_\gamma$ plus a constant.

(e) $\tilde\varphi$ is the infimal convolution of the convex function $\varphi+\iota_{C_{\bar\gamma}}$ (where $\iota$ is the indicator, $0$ on the set and $+\infty$ outside) with the
convex function $M\norm\cdot$, hence convex; it is finite because $C_{\bar\gamma}\ne\emptyset$ and $\varphi$ is bounded below on the bounded set $C_{\bar\gamma}$ (for unbounded $C$
finiteness follows from the Lipschitz bound $\varphi(z)\ge\varphi(z_0)-M\norm{z-z_0}$ on $C_{\bar\gamma}$, which makes $z\mapsto\varphi(z)+M\norm{x-z}$ bounded below by
$\varphi(z_0)-M\norm{x-z_0}$); and it is $M$-Lipschitz as an infimum of $M$-Lipschitz functions of $x$. For $x\in C_{\bar\gamma}$ the choice $z=x$ gives
$\tilde\varphi(x)\le\varphi(x)$, while $\varphi(z)+M\norm{x-z}\ge\varphi(x)$ for every $z\in C_{\bar\gamma}$ by the Lipschitz property on the convex set $C_{\bar\gamma}$; hence
$\tilde\varphi=\varphi$ on $C_{\bar\gamma}$. Finally, for $x\in C$ and $\gamma\le\bar\gamma$ the points $x+\gamma v$ lie in $C_{\bar\gamma}$, so $\tilde\varphi_\gamma(x)=\E\tilde\varphi(x+\gamma v)=\E\varphi(x+\gamma v)=\varphi_\gamma(x)$.
\end{proof}

\begin{proof}[Proof of \Cref{lem:action}(d)]
Fix $\xi$ and $w\ne0$, and put $q_i(u):=\frac1{2\gamma}[G_i(x+\gamma u,\xi)-G_i(x-\gamma u,\xi)]$ for $u\in\Sphere$. Each $q_i$ is odd and $M_g(\xi)$-Lipschitz on $\Sphere$ for the Euclidean
distance, and $|q_i(u)|=\frac12|q_i(u)-q_i(-u)|\le M_g(\xi)$. Extend $q_i$ to $\R^d$ by positive homogeneity: $\mathsf Q_i(\mathsf g):=\norm{\mathsf g}\,q_i(\mathsf g/\norm{\mathsf g})$ for $\mathsf g\ne0$,
$\mathsf Q_i(0):=0$. We claim that $\mathsf Q_i$ is $3M_g(\xi)$-Lipschitz on $\R^d$. Indeed, for $\mathsf g,\mathsf g'\ne0$ with unit vectors $\hat{\mathsf g},\hat{\mathsf g}'$,
\[
 |\mathsf Q_i(\mathsf g)-\mathsf Q_i(\mathsf g')|\le\bigl|\norm{\mathsf g}-\norm{\mathsf g'}\bigr|\,|q_i(\hat{\mathsf g})|+\norm{\mathsf g'}\,|q_i(\hat{\mathsf g})-q_i(\hat{\mathsf g}')|
 \le M_g(\xi)\norm{\mathsf g-\mathsf g'}+M_g(\xi)\norm{\mathsf g'}\,\norm{\hat{\mathsf g}-\hat{\mathsf g}'},
\]
and $\norm{\mathsf g'}\norm{\hat{\mathsf g}-\hat{\mathsf g}'}=\norm{\tfrac{\norm{\mathsf g'}}{\norm{\mathsf g}}\mathsf g-\mathsf g'}\le\norm{\tfrac{\norm{\mathsf g'}}{\norm{\mathsf g}}\mathsf g-\mathsf g}+\norm{\mathsf g-\mathsf g'}
=\bigl|\norm{\mathsf g'}-\norm{\mathsf g}\bigr|+\norm{\mathsf g-\mathsf g'}\le2\norm{\mathsf g-\mathsf g'}$; the case $\mathsf g'=0$ is $|\mathsf Q_i(\mathsf g)|\le M_g(\xi)\norm{\mathsf g}$.
Now let $\mathsf g\sim\mathcal N(0,I_d)$, so that $u:=\mathsf g/\norm{\mathsf g}\sim U(\Sphere)$ is independent of $\norm{\mathsf g}$ (\Cref{fact:moments}). Since $\mathsf Q_i$ is odd, $\E\mathsf Q_i(\mathsf g)=0$,
and \Cref{fact:gaussconc} applied to $\pm\mathsf Q_i$ gives $\Prob(|\mathsf Q_i(\mathsf g)|\ge t)\le2e^{-t^2/(18M_g(\xi)^2)}$. Writing $c:=18M_g(\xi)^2$ and using
$\Prob(\max_i\mathsf Q_i^4>t)\le\min\{1,2me^{-\sqrt t/c}\}$,
\[
 \E\max_i\mathsf Q_i(\mathsf g)^4\le\alpha+\int_\alpha^\infty2me^{-\sqrt t/c}\,dt=\alpha+4mc\,(\sqrt\alpha+c)\,e^{-\sqrt\alpha/c}\qquad(\alpha>0),
\]
because $\int_\alpha^\infty e^{-\sqrt t/c}dt=2\int_{\sqrt\alpha}^\infty se^{-s/c}ds=2c(\sqrt\alpha+c)e^{-\sqrt\alpha/c}$. With $\sqrt\alpha:=c\log(2m)$ this gives
$\E\max_i\mathsf Q_i(\mathsf g)^4\le c^2[\log^2(2m)+2\log(2m)+2]\le2c^2(1+\log(2m))^2=648\,M_g(\xi)^4\kap^2$. On the other hand $\max_i\mathsf Q_i(\mathsf g)^4=\norm{\mathsf g}^4\max_iq_i(u)^4$ with
independent factors, so $\E_u\max_iq_i(u)^4=\E\max_i\mathsf Q_i(\mathsf g)^4/\E\norm{\mathsf g}^4\le648M_g(\xi)^4\kap^2/(d(d+2))$. Since $(\hat J(x)w)_i=d\,q_i(u)\inner uw$, the Cauchy--Schwarz
inequality and \Cref{fact:moments} yield, still for fixed $\xi$,
\begin{multline*}
 \E_u\norm{\hat J(x)w}_\infty^2=d^2\,\E_u\Bigl[\max_iq_i(u)^2\inner uw^2\Bigr]\le d^2\sqrt{\E_u\max_iq_i(u)^4}\sqrt{\E_u\inner uw^4}\\
 \le d^2\cdot\frac{\sqrt{648}\,M_g(\xi)^2\kap}{\sqrt{d(d+2)}}\cdot\frac{\sqrt3\,\norm w^2}{\sqrt{d(d+2)}}\le\sqrt{1944}\,\kap M_g(\xi)^2\norm w^2 .
\end{multline*}
Taking the expectation over $\xi$ with $\E M_g(\xi)^2\le M_g^2$ and $\sqrt{1944}<45$ gives the first claim. The second follows from
$\E\norm{(\hat J-\Jg)w}_\infty^2\le2\E\norm{\hat Jw}_\infty^2+2\norm{\Jg w}_\infty^2\le(90\kap+2)M_g^2\norm w^2\le92\kap M_g^2\norm w^2$ by part (b).
(All steps up to the Cauchy--Schwarz inequality are performed for fixed $\xi$, so that only the second moment of $M_g(\xi)$ is needed at the end; no assumption on
higher moments of the Lipschitz modulus is required.)
\end{proof}

\begin{proof}[Proof of \Cref{lem:batch}]
The bound $\mathsf H$ is Jensen's inequality: $\norm{\frac1b\sum_jZ_j}_\infty^2\le\frac1b\sum_j\norm{Z_j}_\infty^2$. For the second bound let $Z'_1,\dots,Z'_b$ be an independent copy of
$Z_1,\dots,Z_b$ and $\varepsilon_1,\dots,\varepsilon_b$ independent Rademacher signs, independent of everything else. Since $\E Z'_j=0$ and $\norm\cdot_\infty^2$ is convex,
$\norm{\sum_jZ_j}_\infty^2=\norm{\E'[\sum_j(Z_j-Z'_j)]}_\infty^2\le\E'\norm{\sum_j(Z_j-Z'_j)}_\infty^2$, where $\E'$ integrates over the copy; hence
$\E\norm{\sum_jZ_j}_\infty^2\le\E\norm{\sum_j(Z_j-Z'_j)}_\infty^2=\E\norm{\sum_j\varepsilon_j(Z_j-Z'_j)}_\infty^2$, the last equality because $Z_j-Z'_j$ is symmetric and the pairs are
independent. Condition on $(Z_j,Z'_j)_j$ and put $a_{ji}:=(Z_j-Z'_j)_i$, $X_i:=\sum_j\varepsilon_ja_{ji}$ and $s^2:=\sum_j\norm{Z_j-Z'_j}_\infty^2\ge\sum_ja_{ji}^2$ for every $i$. By Hoeffding's lemma,
$\E[e^{\theta X_i}\mid Z,Z']\le e^{\theta^2s^2/2}$ for all $\theta\in\R$, hence $\Prob(X_i^2>t\mid Z,Z')\le2e^{-t/(2s^2)}$. Therefore, for every $\alpha>0$,
\[
 \E\Bigl[\max_iX_i^2\Bigm|Z,Z'\Bigr]\le\alpha+\sum_{i=1}^m\int_\alpha^\infty\Prob(X_i^2>t\mid Z,Z')\,dt\le\alpha+4ms^2e^{-\alpha/(2s^2)},
\]
and the choice $\alpha=2s^2\log(2m)$ gives $\E[\max_iX_i^2\mid Z,Z']\le2s^2(1+\log(2m))=2\kap s^2$. Finally $\E s^2\le\sum_j2(\E\norm{Z_j}_\infty^2+\E\norm{Z'_j}_\infty^2)\le4b\mathsf H$, so
$\E\norm{\sum_jZ_j}_\infty^2\le8\kap b\mathsf H$; dividing by $b^2$ proves the claim. The conditional version is proved identically.
\end{proof}

\section{Proof of Theorem~\ref{thm:mp}}\label{app:mp}

Throughout this appendix $\cZ:=\Q\times\tilde Y_\Rmult$ (we identify $y\in Y_\Rmult$ with its lifting $\tilde y\in\tilde Y_\Rmult$), $V$ is the joint Bregman distance
\eqref{eq:jointgeometry}, $\norm\cdot_\alpha$ and $\norm\cdot_{\alpha,\ast}$ are the norms of \eqref{eq:jointgeometry}, and the operator $\cG$ of \eqref{eq:operator} is
extended to $\cZ$ by padding its dual component with a zero slack coordinate: $\cG(z)=(\nabla f_\gamma(x)+\Jg(x)^\top y,\ -(0,h(x)))$. The estimator of an iteration is
$\hat\cG(z)=(\hat g,-(0,\hat h))$ with the mini-batch averages of \Cref{alg:mp}. By \Cref{lem:prox}(a), $V(z,z')\ge\frac12\norm{z-z'}_\alpha^2$, and $\norm\cdot_{\alpha,\ast}$ is the
dual norm of $\norm\cdot_\alpha$. The proximal steps of \Cref{alg:mp} are $w=\argmin_{z'\in\cZ}\{\inner{\eta\hat\cG(z)}{z'}+V(z,z')\}$ (the two blocks separate).

\begin{lemma}[Lipschitz constant and monotonicity]\label{lem:mp-lip}
For all $z=(x,y),z'=(x',y')\in\cZ$: $\norm{\cG(z)-\cG(z')}_{\alpha,\ast}\le L_\alpha\norm{z-z'}_\alpha$ with $L_\alpha$ from \eqref{eq:Lalpha}, and
$\inner{\cG(z)}{z-z'}\ge\LagS(x,y')-\LagS(x',y)$.
\end{lemma}
\begin{proof}
By \eqref{eq:rowinf}, \Cref{lem:action}(b) and $\norm y_1\le\Rmult$,
$\norm{\nabla_x\LagS(z)-\nabla_x\LagS(z')}\le\norm{\nabla f_\gamma(x)-\nabla f_\gamma(x')}+\norm{(\Jg(x)-\Jg(x'))^\top y}+\norm{\Jg(x')^\top(y-y')}\le H\norm{x-x'}+M_g\norm{y-y'}_1$,
and $\norm{h(x)-h(x')}_\infty\le M_g\norm{x-x'}$. With $\mathsf a:=\sqrt\alpha\norm{x-x'}$ and $\mathsf c:=\norm{\tilde y-\tilde y'}_1/\sqrt\alpha\ge\norm{y-y'}_1/\sqrt\alpha$,
\begin{multline*}
 \norm{\cG(z)-\cG(z')}_{\alpha,\ast}^2\le\frac{(H\norm{x-x'}+M_g\norm{y-y'}_1)^2}{\alpha}+\alpha M_g^2\norm{x-x'}^2\\
 \le\Bigl(\frac H\alpha\mathsf a+M_g\mathsf c\Bigr)^2+(M_g\mathsf a)^2=\norm{\mathsf B(\mathsf a,\mathsf c)^\top}^2,
\end{multline*}
where $\mathsf B:=\bigl(\begin{smallmatrix}H/\alpha&M_g\\M_g&0\end{smallmatrix}\bigr)$;
hence $\norm{\cG(z)-\cG(z')}_{\alpha,\ast}\le\norm{\mathsf B}_2\sqrt{\mathsf a^2+\mathsf c^2}=\norm{\mathsf B}_2\norm{z-z'}_\alpha$, and the spectral norm of the symmetric matrix $\mathsf B$ is
its largest eigenvalue $\frac12(H/\alpha+\sqrt{H^2/\alpha^2+4M_g^2})=L_\alpha$. For the second claim, convexity of $\LagS(\cdot,y)$ gives
$\inner{\nabla_x\LagS(x,y)}{x-x'}\ge\LagS(x,y)-\LagS(x',y)$, and linearity in $y$ gives $\inner{-h(x)}{y-y'}=\LagS(x,y')-\LagS(x,y)$; add.
\end{proof}

\begin{lemma}[one iteration]\label{lem:mp-onestep}
Let $z_k\in\cZ$, $w_k:=\argmin_{z\in\cZ}\{\inner{\eta\hat\cG(z_k)}{z}+V(z_k,z)\}$ and $z_{k+1}:=\argmin_{z\in\cZ}\{\inner{\eta\hat\cG'(w_k)}{z}+V(z_k,z)\}$, where $\hat\cG(z_k)=\cG(z_k)+\Delta_k$ and
$\hat\cG'(w_k)=\cG(w_k)+\Delta'_k$. If $\eta\le1/(\sqrt3L_\alpha)$, then for every $z\in\cZ$
\[
 \eta\inner{\hat\cG'(w_k)}{w_k-z}\le V(z_k,z)-V(z_{k+1},z)+\frac{3\eta^2}{2}\bigl(\norm{\Delta_k}_{\alpha,\ast}^2+\norm{\Delta'_k}_{\alpha,\ast}^2\bigr).
\]
\end{lemma}
\begin{proof}
\Cref{fact:threepoint} for the two proximal steps gives, for all $z\in\cZ$,
$\eta\inner{\hat\cG'(w_k)}{z_{k+1}-z}\le V(z_k,z)-V(z_{k+1},z)-V(z_k,z_{k+1})$ and, with $z=z_{k+1}$ in the first step,
$\eta\inner{\hat\cG(z_k)}{w_k-z_{k+1}}\le V(z_k,z_{k+1})-V(w_k,z_{k+1})-V(z_k,w_k)$. Adding and rearranging,
\begin{multline*}
 \eta\inner{\hat\cG'(w_k)}{w_k-z}\le V(z_k,z)-V(z_{k+1},z)-V(w_k,z_{k+1})-V(z_k,w_k)\\
 +\eta\inner{\hat\cG(z_k)-\hat\cG'(w_k)}{w_k-z_{k+1}}.
\end{multline*}
Put $\mathsf A:=\norm{z_k-w_k}_\alpha$, $\mathsf B:=\norm{w_k-z_{k+1}}_\alpha$ and $\delta:=\norm{\Delta_k}_{\alpha,\ast}+\norm{\Delta'_k}_{\alpha,\ast}$. Since
$\hat\cG(z_k)-\hat\cG'(w_k)=[\cG(z_k)-\cG(w_k)]+\Delta_k-\Delta'_k$, \Cref{lem:mp-lip} and the definition of the dual norm give
$\eta\inner{\hat\cG(z_k)-\hat\cG'(w_k)}{w_k-z_{k+1}}\le\eta L_\alpha\mathsf A\mathsf B+\eta\delta\mathsf B\le\frac12\mathsf A^2+\frac{\eta^2L_\alpha^2}{2}\mathsf B^2+\frac{c}{2}\mathsf B^2+\frac{\eta^2\delta^2}{2c}$
with $c:=1-\eta^2L_\alpha^2\ge\frac23$. Using $V(w_k,z_{k+1})\ge\frac12\mathsf B^2$ and $V(z_k,w_k)\ge\frac12\mathsf A^2$, the four terms $-V(w_k,z_{k+1})-V(z_k,w_k)+\frac12\mathsf A^2+\frac12(\eta^2L_\alpha^2+c)\mathsf B^2$
are at most $0$, and $\frac{\eta^2\delta^2}{2c}\le\frac{3\eta^2}{4}\delta^2\le\frac{3\eta^2}{2}(\norm{\Delta_k}_{\alpha,\ast}^2+\norm{\Delta'_k}_{\alpha,\ast}^2)$.
\end{proof}

\begin{proof}[Proof of \Cref{thm:mp}]
Let $\cF_{k}$ denote the $\sigma$-field generated by $x_1$ and by the mini-batches of the iterations $1,\dots,k$, and $\cF_{k-1/2}$ the one that additionally contains the first mini-batch of
iteration $k$. Then $z_k$ is $\cF_{k-1}$-measurable, $w_k$ is $\cF_{k-1/2}$-measurable, $\E[\Delta_k\mid\cF_{k-1}]=0$ and $\E[\Delta'_k\mid\cF_{k-1/2}]=0$ by \Cref{lem:action}(a) and
\Cref{lem:levels}, and
\begin{multline*}
 \E\bigl[\norm{\Delta_k}_{\alpha,\ast}^2\mid\cF_{k-1}\bigr]=\alpha^{-1}\E[\norm{\hat g_k-\nabla_x\LagS(z_k)}^2\mid\cF_{k-1}]+\alpha\,\E[\norm{\hat h_k-h(x_k)}_\infty^2\mid\cF_{k-1}]\\
 \le\frac{2dM_R^2}{\alpha b}+\alpha\min\Bigl\{1,\frac{8\kap}b\Bigr\}\sigma_h^2=\sigma_\ast^2
\end{multline*}
by \Cref{lem:action}(a) (the $b$ paired samples are conditionally i.i.d.\ and centered, with weight vector $y_k$, $\norm{y_k}_1\le\Rmult$), \Cref{lem:levels} and \Cref{lem:batch}; the
same bound holds for $\Delta'_k$ given $\cF_{k-1/2}$.

By \Cref{lem:mp-onestep} and \Cref{lem:mp-lip}, for every $z=(x,\tilde y)\in\cZ$,
\begin{multline*}
 \eta\bigl[\LagS(x^w_k,y)-\LagS(x,y^w_k)\bigr]\le\eta\inner{\cG(w_k)}{w_k-z}=\eta\inner{\hat\cG'(w_k)}{w_k-z}-\eta\inner{\Delta'_k}{w_k-z}\\
 \le V(z_k,z)-V(z_{k+1},z)+\frac{3\eta^2}2\bigl(\norm{\Delta_k}_{\alpha,\ast}^2+\norm{\Delta'_k}_{\alpha,\ast}^2\bigr)+\eta\inner{\Delta'_k}{z-w_k}.
\end{multline*}
Summing over $k=1,\dots,N$, dividing by $\eta N$, and using convexity of $\LagS(\cdot,y)$ and linearity of $\LagS(x,\cdot)$,
\begin{multline*}
 \LagS(\Xr_N,y)-\LagS(x,\yr_N)\le\frac1N\sum_{k=1}^N\bigl[\LagS(x^w_k,y)-\LagS(x,y^w_k)\bigr]\\
 \le\frac{V(z_1,z)}{\eta N}+\frac{3\eta}{2N}\sum_{k=1}^N\bigl(\norm{\Delta_k}_{\alpha,\ast}^2+\norm{\Delta'_k}_{\alpha,\ast}^2\bigr)+\frac1N\sum_{k=1}^N\inner{\Delta'_k}{z-w_k}.
\end{multline*}
Now fix $x:=x^\ast_\gamma$ and let $Z:=\{x^\ast_\gamma\}\times\tilde Y_\Rmult$. By \eqref{eq:restrictedgap}, $\Cert_{\Rmult,\gamma}(\Xr_N)\le\sup_{y\in Y_\Rmult}\LagS(\Xr_N,y)-\LagS(x^\ast_\gamma,\yr_N)$, and
$\sup_{z\in Z}V(z_1,z)=\alpha\cA+\alpha^{-1}\cB=\cD_\alpha^2$ by \eqref{eq:radii} ($\tilde y_1=\tilde y^{\mathrm u}$). Hence
\[
 \Cert_{\Rmult,\gamma}(\Xr_N)\le\frac{\cD_\alpha^2}{\eta N}+\frac{3\eta}{2N}\sum_{k=1}^N\bigl(\norm{\Delta_k}_{\alpha,\ast}^2+\norm{\Delta'_k}_{\alpha,\ast}^2\bigr)+\frac1N\sup_{z\in Z}\sum_{k=1}^N\inner{\Delta'_k}{z-w_k}.
\]
Taking expectations, the second term is at most $3\eta\sigma_\ast^2$. For the third we apply \Cref{lem:martsup} with the compact set $Z$, the ambient set $Z':=\cZ=\Q\times\tilde Y_\Rmult$
(the distance-generating function $\alpha\frac12\norm x^2+\alpha^{-1}\Rmult\sum_i\tilde y_i\log\tilde y_i$ is $1$-strongly convex for $\norm\cdot_\alpha$ on $\cZ$), the starting point
$z_0:=z_1=(x_1,\tilde y^{\mathrm u})\in\cZ$, which lies in $Z'$ but in general not in $Z$ (its primal part is $x_1\ne x^\ast_\gamma$), and $\Theta:=\sup_{z\in Z}V(z_1,z)=\cD_\alpha^2$, with the filtration
$(\cF_{k-1/2})_k$, $e_k:=\Delta'_k$ (which is $\cF_{k+1/2}$-measurable and centered given $\cF_{k-1/2}$), $\zeta_k:=w_k$ (bounded and $\cF_{k-1/2}$-measurable), $a_k:=1$ and
$v_k^2:=\sigma_\ast^2$. It gives
$\E\sup_{z\in Z}\sum_k\inner{\Delta'_k}{z-w_k}\le\sqrt{2\cD_\alpha^2N\sigma_\ast^2}$, and dividing by $N$ yields \eqref{eq:mp-bound}.

For \eqref{eq:mp-rate}, if $\eta=\cD_\alpha/(\sigma_\ast\sqrt{3N})\le1/(\sqrt3L_\alpha)$ then $\cD_\alpha^2/(\eta N)=3\eta\sigma_\ast^2=\sqrt3\sigma_\ast\cD_\alpha/\sqrt N$ and the right-hand side of
\eqref{eq:mp-bound} equals $(2\sqrt3+\sqrt2)\sigma_\ast\cD_\alpha/\sqrt N\le5\sigma_\ast\cD_\alpha/\sqrt N$. Otherwise $\eta=1/(\sqrt3L_\alpha)<\cD_\alpha/(\sigma_\ast\sqrt{3N})$, \ie\ $\sigma_\ast\sqrt N<L_\alpha\cD_\alpha$, and then
$\cD_\alpha^2/(\eta N)=\sqrt3L_\alpha\cD_\alpha^2/N$, $3\eta\sigma_\ast^2=\sqrt3\sigma_\ast^2/L_\alpha\le\sqrt3\sigma_\ast\cD_\alpha/\sqrt N$, so the right-hand side is at most
$\sqrt3L_\alpha\cD_\alpha^2/N+(\sqrt3+\sqrt2)\sigma_\ast\cD_\alpha/\sqrt N$.
\end{proof}

\begin{remark}[random starting point]\label{rem:mp-random-start}
If $x_1$ is random and independent of the mini-batches of the run, the proof above applies conditionally on $x_1$ with the random radius $\cD_\alpha^2=\alpha\cA+\alpha^{-1}\cB$,
$\cA=\frac12\norm{x_1-x^\ast_\gamma}^2$, and any deterministic $\eta\le1/(\sqrt3L_\alpha)$:
\[
 \E\bigl[\Cert_{\Rmult,\gamma}(\Xr_N)\mid x_1\bigr]\le\frac{\alpha\cA+\cB/\alpha}{\eta N}+3\eta\sigma_\ast^2+\sqrt{\frac{2(\alpha\cA+\cB/\alpha)\sigma_\ast^2}{N}} .
\]
The right-hand side is a nondecreasing concave function of $\cA$; this is the form used in the restart analysis of \Cref{thm:restart-mp}.
\end{remark}

\section{Proofs for Section~\ref{sec:sliding}}\label{app:sliding}

\subsection{Proof of Lemma~\ref{lem:schedule}}
Recall $S_t=\max\{1,\lceil ct\rceil\}$ with $c=M_g\rs/L$, $a_t=t/S_t$.

(i) $S_t\ge\lceil ct\rceil\ge ct$, hence $a_t\le1/c=L/(M_g\rs)$.

(ii) Let $t\ge2$. If $\lceil ct\rceil\le1$ then $S_t=S_{t-1}=1$ and $a_{t-1}/a_t=(t-1)/t<1$. Otherwise $ct>1$ and $S_t=\lceil ct\rceil<ct+1$. If $c(t-1)\ge1$, then $S_{t-1}\ge c(t-1)$ and
$a_{t-1}/a_t=\frac{(t-1)S_t}{tS_{t-1}}\le\frac{(t-1)(ct+1)}{t\,c(t-1)}=1+\frac1{ct}<2$. If $c(t-1)<1$, then $S_{t-1}=1$, $c<1/(t-1)$, and
$a_{t-1}/a_t\le\frac{(t-1)(ct+1)}{t}<\frac{(t-1)}{t}\Bigl(\frac{t}{t-1}+1\Bigr)=\frac{2t-1}{t}<2$.

(iii) From $\max\{1,ct\}\le S_t\le1+ct$ we get $\max\{N,\frac c2N(N+1)\}\le K_N\le N+\frac c2N(N+1)$.

(iv) The lower bound is the Cauchy--Schwarz inequality $W_N^2=(\sum_tt)^2=(\sum_t\frac{t}{\sqrt{S_t}}\sqrt{S_t})^2\le A_NK_N$. For the upper bound, note $t^2/S_t\le\min\{t^2,t/c\}$.
If $cN\le1$, then $S_t=1$ for all $t\le N$, $K_N=N$, $A_N=\frac{N(N+1)(2N+1)}6$ and $A_NK_N/W_N^2=\frac{2(2N+1)}{3(N+1)}\le\frac43$. If $cN>1$, let $T_0:=\lfloor1/c\rfloor\le N-1$. Then
\begin{multline*}
 A_N\le\sum_{t\le T_0}t^2+\frac1c\sum_{t=T_0+1}^Nt\le\frac{T_0(T_0+1)(2T_0+1)}{6}+\frac{N(N+1)}{2c}\\
 \le\frac{(T_0+1)(2T_0+1)}{6c}+\frac{N(N+1)}{2c}\le\frac{5N(N+1)}{6c},
\end{multline*}
using $T_0\le1/c$, $T_0+1\le N$ and $2T_0+1\le2N$. Since $K_N\ge\frac c2N(N+1)$, \ie\ $\frac1c\le\frac{N(N+1)}{2K_N}$, we get $A_N\le\frac{5N^2(N+1)^2}{12K_N}=\frac53\frac{W_N^2}{K_N}$.

\subsection{Proof of Lemma~\ref{lem:centers}}
For $t=1$, $\tau_1=\theta_1=0$ gives $u_1=x_0$. Multiplying the recursion by $t(1+\tau_t)=t(t+1)/2$ and using $t\tau_t=t(t-1)/2$, $t\theta_t=t-1$,
\[
 t(t+1)u_t=t(t-1)u_{t-1}+2t\,x_{t-1}+2(t-1)(x_{t-1}-x_{t-2}).
\]
For $t=2$ this gives $6u_2=2x_0+4x_1+2(x_1-x_0)=6x_1$, which is \eqref{eq:barycenter} (empty sum). Assume \eqref{eq:barycenter} for $t-1\ge2$, \ie\
$(t-1)t\,u_{t-1}=\sum_{j\le t-3}2jx_j+(4t-6)x_{t-2}$. Then
$t(t+1)u_t=\sum_{j\le t-3}2jx_j+(4t-6)x_{t-2}-2(t-1)x_{t-2}+(2t+2t-2)x_{t-1}=\sum_{j\le t-3}2jx_j+2(t-2)x_{t-2}+(4t-2)x_{t-1}$, which is \eqref{eq:barycenter} for $t$. The coefficients
are nonnegative and sum to $(t-2)(t-1)+4t-2=t(t+1)$, so $u_t$ is a convex combination of $x_1,\dots,x_{t-1}$. Finally, the recursion gives
$x_{t-1}=(1+\tau_t)u_t-\tau_tu_{t-1}-\theta_td_{t-1}$, hence $x_t-u_t=d_t+x_{t-1}-u_t=d_t-\theta_td_{t-1}+\tau_t(u_t-u_{t-1})$.

\subsection{Proof of Lemma~\ref{lem:outer}}
Write $g_t:=\nabla\mathsf F(u_t)$, $\pi:=\nabla\mathsf F(\Xr_N)$, and let $\mathsf F^\ast$ be the conjugate of $\mathsf F$. By \Cref{fact:conjugate}, $\mathsf F^\ast(g_t)=\inner{g_t}{u_t}-\mathsf F(u_t)$,
so that $\ell^{\mathsf F}_t(x)=\inner{g_t}x-\mathsf F^\ast(g_t)$, and the ``Bregman distances'' $D_t(\pi'):=\mathsf F^\ast(\pi')-\mathsf F^\ast(g_t)-\inner{u_t}{\pi'-g_t}$ satisfy
\begin{equation}\label{eq:Dt}
 D_t(\pi')\ge\frac1{2L}\norm{\pi'-g_t}^2\ge0\qquad\text{for all }\pi'\in\dom\mathsf F^\ast .
\end{equation}
Fenchel--Young at $\Xr_N$ gives $\mathsf F(\Xr_N)=\inner\pi{\Xr_N}-\mathsf F^\ast(\pi)$, and $W_N\Xr_N=\sum_ttx_t$; therefore
\[
 W_N\mathsf F(\Xr_N)-\sum_{t=1}^Nt\,\ell^{\mathsf F}_t(x_t)=\sum_{t=1}^Nt\bigl[\inner{\pi-g_t}{x_t}-\mathsf F^\ast(\pi)+\mathsf F^\ast(g_t)\bigr]=\sum_{t=1}^Nt\bigl[\inner{\pi-g_t}{x_t-u_t}-D_t(\pi)\bigr],
\]
where we used $\mathsf F^\ast(\pi)-\mathsf F^\ast(g_t)=D_t(\pi)+\inner{u_t}{\pi-g_t}$. Insert \eqref{eq:xminusu}, $x_t-u_t=d_t-\theta_td_{t-1}+\tau_t(u_t-u_{t-1})$, and the identity
\[
 \inner{\pi-g_t}{u_t-u_{t-1}}=D_{t-1}(\pi)-D_t(\pi)-D_{t-1}(g_t)\qquad(t\ge2),
\]
which is checked by expanding the three definitions (the terms $\mathsf F^\ast(\pi)$, $\mathsf F^\ast(g_{t-1})$ and $\inner{u_{t-1}}{\pi}$ cancel). For $t=1$ we have $\tau_1=0$, so the identity is not needed.
We obtain
\begin{multline*}
 W_N\mathsf F(\Xr_N)-\sum_tt\,\ell^{\mathsf F}_t(x_t)=\sum_{t=1}^Nt\inner{\pi-g_t}{d_t-\theta_td_{t-1}}\\
 +\sum_{t=2}^N\bigl[t\tau_tD_{t-1}(\pi)-t(1+\tau_t)D_t(\pi)-t\tau_tD_{t-1}(g_t)\bigr]-1\cdot(1+\tau_1)D_1(\pi).
\end{multline*}
Since $t(1+\tau_t)=t(t+1)/2=(t+1)\tau_{t+1}$, the $D(\pi)$-terms telescope: $\sum_{t=1}^N[t\tau_tD_{t-1}(\pi)-(t+1)\tau_{t+1}D_t(\pi)]=-(N+1)\tau_{N+1}D_N(\pi)=-W_ND_N(\pi)$ (recall $\tau_1=0$).
For the first sum, $t\theta_t=t-1$ and a shift of the index give
\[
 \sum_{t=1}^Nt\inner{\pi-g_t}{d_t}-\sum_{t=2}^N(t-1)\inner{\pi-g_t}{d_{t-1}}=N\inner{\pi-g_N}{d_N}+\sum_{t=1}^{N-1}t\inner{g_{t+1}-g_t}{d_t}.
\]
Hence
\begin{multline*}
 W_N\mathsf F(\Xr_N)-\sum_tt\,\ell^{\mathsf F}_t(x_t)\\
 =\Bigl[N\inner{\pi-g_N}{d_N}-W_ND_N(\pi)\Bigr]+\sum_{t=1}^{N-1}\Bigl[t\inner{g_{t+1}-g_t}{d_t}-(t+1)\tau_{t+1}D_t(g_{t+1})\Bigr].
\end{multline*}
By \eqref{eq:Dt} and the elementary bound $\max_{r\ge0}\{\mathsf pr-\mathsf qr^2\}=\mathsf p^2/(4\mathsf q)$ for $\mathsf p,\mathsf q>0$,
\begin{multline*}
 N\inner{\pi-g_N}{d_N}-W_ND_N(\pi)\le N\norm{\pi-g_N}\norm{d_N}-\frac{W_N}{2L}\norm{\pi-g_N}^2\\
 \le\frac{LN^2\norm{d_N}^2}{2W_N}=\frac{LN}{N+1}\norm{d_N}^2\le L\norm{d_N}^2,
\end{multline*}
and, with $(t+1)\tau_{t+1}=t(t+1)/2$,
\begin{multline*}
 t\inner{g_{t+1}-g_t}{d_t}-\frac{t(t+1)}{2}D_t(g_{t+1})\le t\norm{g_{t+1}-g_t}\norm{d_t}-\frac{t(t+1)}{4L}\norm{g_{t+1}-g_t}^2\\
 \le\frac{Lt}{t+1}\norm{d_t}^2\le L\norm{d_t}^2 .
\end{multline*}
Summing proves \eqref{eq:outer}.

\subsection{Proof of Lemma~\ref{lem:inner}}
Fix $y\in Y_\Rmult$ and write $x:=x^\ast_\gamma$, $\Delta y_k:=y_k-y_{k-1}$ ($\Delta y_0=0$), $\Delta p_k:=p_k-p_{k-1}$, $K:=K_N$, and $t(k)$ for the phase of slot $k$. All Bregman distances below are $\Vx$ and $\Vy$ of
\eqref{eq:bregman}.

\emph{Step 1: three-point inequalities.} By \Cref{lem:prox}(d) applied to the primal step with the direction $\hat r_k=r_k+e^p_k$, for all $x'\in\Q$,
\begin{multline*}
 \inner{r_k}{p_k-x'}\le\eta_{t(k)}\bigl[\Vx(x_{t(k)-1},x')-\Vx(x_{t(k)-1},p_k)-\Vx(p_k,x')\bigr]\\
 +\beta_{t(k)}\bigl[\Vx(p_{k-1},x')-\Vx(p_{k-1},p_k)-\Vx(p_k,x')\bigr]-\inner{e^p_k}{p_k-x'},
\end{multline*}
and by \Cref{lem:prox}(c) applied to the dual step with $\hat q_k=q_k+e^y_k$ and $\tau=\beta_{t(k)}/\rs^2$, for all $y'\in Y_\Rmult$,
\[
 \inner{q_k}{y'-y_k}\le\frac{\beta_{t(k)}}{\rs^2}\bigl[\Vy(\tilde y_{k-1},\tilde y')-\Vy(\tilde y_k,\tilde y')-\Vy(\tilde y_{k-1},\tilde y_k)\bigr]+\inner{e^y_k}{y_k-y'}.
\]

\emph{Step 2: the linearized gap and the cross terms.} From \eqref{eq:linearization} and \eqref{eq:exactdirections}, a direct computation gives, for every slot $k$ of phase $t$,
\[
 \ell_t(p_k,y)-\ell_t(x,y_k)=\inner{r_k}{p_k-x}+\inner{q_k}{y-y_k}+\inner{\Delta y_k}{J_t(p_k-x)}-\rho_k\inner{\Delta y_{k-1}}{J_{t'(k)}(p_k-x)} .
\]
(Indeed, $\ell_t(p_k,y)-\ell_t(x,y_k)=\inner{\nabla f_\gamma(u_t)}{p_k-x}+\inner y{q_k}-\inner{y_k}{h(u_t)+J_t(p_k-u_t)}+\inner{y_k}{J_t(p_k-x)}$, while
$\inner{r_k}{p_k-x}+\inner{q_k}{y-y_k}$ equals the same expression with $\inner{y_k}{J_t(p_k-x)}$ replaced by $\inner{y_{k-1}}{J_t(p_k-x)}+\rho_k\inner{\Delta y_{k-1}}{J_{t'(k)}(p_k-x)}$.)
Define $C_k:=a_k\inner{\Delta y_k}{J_{t(k)}(p_k-x)}$ and $C_0:=0$. For $k\ge2$ we have $a_k\rho_k=a_{k-1}$ and $J_{t'(k)}=J_{t(k-1)}$ in both cases of \eqref{eq:exactdirections} (for a non-first
slot $\rho_k=1$, $a_k=a_{k-1}$, $t'(k)=t(k)=t(k-1)$; for the first slot of phase $t\ge2$, $a_k\rho_k=a_t\cdot a_{t-1}/a_t=a_{t-1}=a_{k-1}$ and $t'(k)=t-1=t(k-1)$), so that
$a_k\rho_k\inner{\Delta y_{k-1}}{J_{t'(k)}(p_k-x)}=C_{k-1}+a_{k-1}\inner{\Delta y_{k-1}}{J_{t(k-1)}\Delta p_k}$; for $k=1$ both sides vanish. Multiplying the identity by $a_k$ and summing over $k$,
\begin{multline*}
 \sum_{k=1}^Ka_k\bigl[\ell_{t(k)}(p_k,y)-\ell_{t(k)}(x,y_k)\bigr]\\
 =\sum_ka_k\inner{r_k}{p_k-x}+\sum_ka_k\inner{q_k}{y-y_k}+C_K-\sum_{k=2}^Ka_{k-1}\inner{\Delta y_{k-1}}{J_{t(k-1)}\Delta p_k}.
\end{multline*}
By \eqref{eq:rowinf}, \Cref{lem:action}(b), \Cref{lem:schedule}(i) and Young's inequality,
\begin{equation}\label{eq:cross-bound}
\begin{gathered}
 \bigl|a_{k-1}\inner{\Delta y_{k-1}}{J_{t(k-1)}\Delta p_k}\bigr|\le a_{k-1}M_g\rs\cdot\frac{\norm{\Delta y_{k-1}}_1}{\rs}\norm{\Delta p_k}\le\frac L2\Bigl[\frac{\norm{\Delta y_{k-1}}_1^2}{\rs^2}+\norm{\Delta p_k}^2\Bigr],\\
 |C_K|\le\frac L2\Bigl[\frac{\norm{\Delta y_K}_1^2}{\rs^2}+\norm{p_K-x}^2\Bigr].
\end{gathered}
\end{equation}
Since $\ell_t$ is affine in each argument and $\sum_{k\in\text{phase }t}a_t=t$, the left-hand side equals $\sum_tt[\ell_t(x_t,y)-\ell_t(x,\bar y_t)]$. Therefore
\begin{multline}\label{eq:step2}
 \sum_tt\bigl[\ell_t(x_t,y)-\ell_t(x,\bar y_t)\bigr]\le\sum_ka_k\inner{r_k}{p_k-x}+\sum_ka_k\inner{q_k}{y-y_k}\\
 +\frac L{2\rs^2}\sum_{k=1}^K\norm{\Delta y_k}_1^2+\frac L2\sum_{k=2}^K\norm{\Delta p_k}^2+\frac L2\norm{p_K-x}^2 .
\end{multline}

\emph{Step 3: telescoping.} Insert Step 1 with $x'=x$ and $y'=y$. Since $a_k\beta_{t(k)}=L+\lambda$ for all $k$ and $p_0=x_0$,
\[
 \sum_ka_k\beta_{t(k)}\bigl[\Vx(p_{k-1},x)-\Vx(p_k,x)-\Vx(p_{k-1},p_k)\bigr]=(L+\lambda)\bigl[\cA-\Vx(p_K,x)\bigr]-\frac{L+\lambda}2\sum_{k=1}^K\norm{\Delta p_k}^2 .
\]
For the outer terms, $a_t\eta_tS_t=t\eta_t=2L$ and the convexity of $\Vx$ in each argument (Jensen over the $S_t$ slots of phase $t$, with $x_t$ the average of the $p_k$) give
\begin{multline*}
 \sum_ka_k\eta_{t(k)}\bigl[\Vx(x_{t-1},x)-\Vx(x_{t-1},p_k)-\Vx(p_k,x)\bigr]\\
 \le2L\sum_{t=1}^N\bigl[\Vx(x_{t-1},x)-\Vx(x_{t-1},x_t)-\Vx(x_t,x)\bigr]=2L\bigl[\cA-\Vx(x_N,x)\bigr]-L\sum_{t=1}^N\norm{d_t}^2 .
\end{multline*}
For the dual terms, $a_k\beta_{t(k)}/\rs^2=(L+\lambda)/\rs^2$, $\Vy(\tilde y_0,\tilde y)\le\cB$ and \Cref{lem:prox}(a) give
\[
 \sum_k\frac{a_k\beta_{t(k)}}{\rs^2}\bigl[\Vy(\tilde y_{k-1},\tilde y)-\Vy(\tilde y_k,\tilde y)-\Vy(\tilde y_{k-1},\tilde y_k)\bigr]\le\frac{L+\lambda}{\rs^2}\cB-\frac{L+\lambda}{2\rs^2}\sum_{k=1}^K\norm{\Delta y_k}_1^2 .
\]

\emph{Step 4: collecting.} Substituting into \eqref{eq:step2} and dropping $-2L\Vx(x_N,x)\le0$, the terms in $\norm{p_K-x}^2$ combine to $-(L+\lambda)\Vx(p_K,x)+\frac L2\norm{p_K-x}^2=-\frac\lambda2\norm{p_K-x}^2\le0$, the terms in
$\norm{\Delta p_k}^2$ to $-\frac{L+\lambda}2\sum_{k=1}^K\norm{\Delta p_k}^2+\frac L2\sum_{k=2}^K\norm{\Delta p_k}^2\le-\frac\lambda2\sum_{k=1}^K\norm{\Delta p_k}^2$, and the terms in $\norm{\Delta y_k}_1^2$ to
$-\frac\lambda{2\rs^2}\sum_{k=1}^K\norm{\Delta y_k}_1^2$. What remains is
\begin{multline*}
 \sum_tt\bigl[\ell_t(x_t,y)-\ell_t(x,\bar y_t)\bigr]+L\sum_t\norm{d_t}^2+\frac\lambda2\sum_k\Bigl[\norm{\Delta p_k}^2+\frac{\norm{\Delta y_k}_1^2}{\rs^2}\Bigr]\\
 \le(3L+\lambda)\cA+\frac{(L+\lambda)\cB}{\rs^2}-\sum_ka_k\inner{e^p_k}{p_k-x}+\sum_ka_k\inner{e^y_k}{y_k-y},
\end{multline*}
and $(3L+\lambda)\cA+(L+\lambda)\cB/\rs^2\le(L+\lambda)(3\cA+\cB/\rs^2)=(L+\lambda)\dA^2$. This is \eqref{eq:energy}.

\subsection{Proof details for Theorem~\ref{thm:affine}}\label{app:affine}
In \Cref{alg:affine} the exact directions are $r_k=\nabla f_\gamma(u_t)+A^\top y_{k-1}+\rho_kA^\top\Delta y_{k-1}$ and $q_k=Ap_k-c$, the errors are $e^p_k=\hat e_{t(k)}:=\hat g_{t(k)}-\nabla f_\gamma(u_{t(k)})$
and $e^y_k=0$. Steps 1 and 2 of the previous proof apply verbatim with $J_t=A$, $M_g=M_A$ (\Cref{lem:schedule}(i) holds with $S_t=\max\{1,\lceil M_A\rs t/L\rceil\}$), and \eqref{eq:cross-bound} requires
$\frac L2$ per increment on both sides. In Step 3 the inner coefficients are $a_k\beta_{t(k)}=L$, so the primal inner terms telescope to $L[\cA-\Vx(p_K,x)]-\frac L2\sum_k\norm{\Delta p_k}^2$ and the dual terms to
$\frac{L}{\rs^2}\cB-\frac{L}{2\rs^2}\sum_k\norm{\Delta y_k}_1^2$; these exactly cancel the cross terms of \eqref{eq:cross-bound} (no negative increment terms remain, which is why the stabilizer is not needed
in the inner loop). The outer terms have $t\eta_t=2L+\lambda$ and telescope to $(2L+\lambda)[\cA-\Vx(x_N,x)]-(L+\frac\lambda2)\sum_t\norm{d_t}^2$, so that Step 4 yields
\[
 \sum_tt\bigl[\ell_t(x_t,y)-\ell_t(x,\bar y_t)\bigr]+\Bigl(L+\frac\lambda2\Bigr)\sum_t\norm{d_t}^2\le(3L+\lambda)\cA+\frac{L\cB}{\rs^2}-\sum_ka_k\inner{e^p_k}{p_k-x},
\]
as claimed in the proof of \Cref{thm:affine}. Finally, $\sum_{k\in\text{phase }t}a_t\inner{\hat e_t}{p_k-x}=t\inner{\hat e_t}{x_t-x}$ because $e^p_k$ is the same vector for all slots of the phase.

\section{Proofs for Section~\ref{sec:lower}}\label{app:lower}

\subsection{Proof of Theorem~\ref{thm:lower-objective}}
\emph{The family.} For $v\in\{-1,1\}^d$ let $\theta_v:=av$ with a parameter $a\in(0,M/(2\sqrt d)]$ fixed below, so that $\norm{\theta_v}=a\sqrt d\le M/2$, and let $\xi\sim\mathcal N(\theta_v,s^2I_d)$ with
$s^2=M^2/(2d)$. Then $F(x,\xi)=\inner\xi x$ is $\norm\xi$-Lipschitz with $\E\norm\xi^2=a^2d+s^2d\le M^2/4+M^2/2<M^2$, so \Cref{ass:lip} holds with $M_0=M$; the constraints $G_i\equiv-1$ are trivially
Lipschitz, noiseless and satisfy Slater's condition with $\rho=1$. The objective is $f_v(x)=a\inner vx$, and on $\Q=R_x\Ball$ its minimum is $f_v^\ast=-aR_x\sqrt d$, attained at $-R_xv/\sqrt d$.

\emph{From the gap to a Hamming distance.} For $\hat x\in\Q$ let $\hat v_j:=-\operatorname{sign}(\hat x_j)$ (with $\operatorname{sign}(0):=1$) and $\mathsf J:=\{j:\hat v_j=v_j\}$, so that $v_j\hat x_j\le0$ for $j\in\mathsf J$ and
$v_j\hat x_j\ge0$ otherwise. Then $\inner v{\hat x}\ge\sum_{j\in\mathsf J}v_j\hat x_j\ge-\sum_{j\in\mathsf J}|\hat x_j|\ge-\sqrt{|\mathsf J|}\,\norm{\hat x}\ge-R_x\sqrt{d-\Ham(\hat v,v)}$, whence
\[
 f_v(\hat x)-f_v^\ast=a\bigl(\inner v{\hat x}+R_x\sqrt d\bigr)\ge aR_x\bigl(\sqrt d-\sqrt{d-\Ham(\hat v,v)}\bigr)\ge\frac{aR_x}{2\sqrt d}\Ham(\hat v,v),
\]
using $\sqrt d-\sqrt{d-k}=k/(\sqrt d+\sqrt{d-k})\ge k/(2\sqrt d)$.

\emph{Assouad's argument.} Let $P_v$ be the law of the transcript of the algorithm (internal seed, queries and observations) under $\theta_v$; $\hat v$ is a measurable function of the transcript.
For $j\in\{1,\dots,d\}$ and $v$ let $v^{(j)}$ denote $v$ with the $j$-th coordinate flipped, and let $P_{+j}:=2^{-(d-1)}\sum_{v:v_j=1}P_v$, $P_{-j}:=2^{-(d-1)}\sum_{v:v_j=-1}P_v$. Then
\[
 \frac1{2^d}\sum_vP_v(\hat v_j\ne v_j)=\frac12\bigl[P_{+j}(\hat v_j=-1)+P_{-j}(\hat v_j=1)\bigr]\ge\frac12\bigl(1-\TV(P_{+j},P_{-j})\bigr),
\]
because $P(A)+P'(A^c)\ge1-\TV(P,P')$ for every event $A$. Summing over $j$ and using the joint convexity of $\TV$ (the two mixtures are averages over the matched pairs $(v,v^{(j)})$, $v_j=1$),
\[
 \frac1{2^d}\sum_v\E_v\Ham(\hat v,v)\ge\frac d2-\frac12\sum_{j=1}^d\TV(P_{+j},P_{-j})\ge\frac d2-\frac12\cdot\frac1{2^d}\sum_v\sum_{j=1}^d\TV\bigl(P_v,P_{v^{(j)}}\bigr),
\]
where each unordered pair $\{v,v^{(j)}\}$ appears twice in the double sum, which accounts for the factor $2^{-d}$ (instead of $2^{-(d-1)}$) and the symmetry of $\TV$.

\emph{Information per query.} Fix $v$ and $j$. The observation of a paired query at $(x^+,x^-)$ is $O=(\inner\xi{x^+},\inner\xi{x^-})=X^\top\xi$ with $X:=[x^+\ x^-]\in\R^{d\times2}$ (a single
query is the case $x^+=x^-$), a Gaussian vector with mean $X^\top\theta$ and covariance $s^2X^\top X$. For two means $\theta,\theta'$ the Kullback--Leibler divergence between the corresponding
observation laws is $\frac1{2s^2}\norm{\Pi_X(\theta-\theta')}^2$, where $\Pi_X=X(X^\top X)^+X^\top$ is the orthogonal projector onto the span of $x^+,x^-$ (this is the standard formula for Gaussians
with a common, possibly singular, covariance and a mean difference in its range). For $\theta=\theta_v$, $\theta'=\theta_{v^{(j)}}$ the difference is $\pm2ae_j$, so the divergence equals
$\frac{2a^2}{s^2}\norm{\Pi_Xe_j}^2$. By the chain rule (\Cref{fact:chain}), $\KL(P_v\|P_{v^{(j)}})=\frac{2a^2}{s^2}\E_v\sum_{t=1}^T\norm{\Pi_{X_t}e_j}^2$, where $X_t$ is the (random) $t$-th query.
For fixed $v$, Pinsker's inequality, the Cauchy--Schwarz inequality over $j$, and $\sum_j\norm{\Pi_Xe_j}^2=\tr\Pi_X\le2$ give
\begin{multline*}
 \sum_{j=1}^d\TV\bigl(P_v,P_{v^{(j)}}\bigr)\le\sum_{j=1}^d\sqrt{\tfrac12\KL(P_v\|P_{v^{(j)}})}\le\sqrt{\frac d2\sum_{j=1}^d\KL(P_v\|P_{v^{(j)}})}\\
 =\sqrt{\frac{da^2}{s^2}\,\E_v\sum_{t=1}^T\tr\Pi_{X_t}}\le\frac{a\sqrt{2dT}}{s}.
\end{multline*}
(The expectation $\E_v$ is common to all $j$, which is what allows the trace identity to be used inside it.) Consequently
$2^{-d}\sum_v\E_v\Ham(\hat v,v)\ge\frac12\bigl(d-a\sqrt{2dT}/s\bigr)$.

\emph{Choice of $a$.} If $T\ge d/4$, take $a:=s\sqrt{d/(8T)}=M/(4\sqrt T)$; then $a\sqrt d=M\sqrt{d/T}/4\le M/2$ and $a\sqrt{2dT}/s=d/2$, so the average Hamming distance is at least $d/4$ and the
average gap at least $\frac{aR_x}{2\sqrt d}\cdot\frac d4=\frac{aR_x\sqrt d}8=\frac{MR_x}{32}\sqrt{\frac dT}$. If $T<d/4$, take $a:=M/(2\sqrt d)$; then $a\sqrt{2dT}/s=\sqrt{dT}<d/2$, the average Hamming
distance exceeds $d/4$, and the average gap exceeds $\frac{aR_x\sqrt d}{8}=\frac{MR_x}{16}\ge\frac{MR_x}{32}$. In both cases some $v$ attains at least the average, which proves the theorem. For the
consequence, note that $\frac{MR_x}{32}\min\{1,\sqrt{d/T}\}\le\eps<\frac{MR_x}{32}$ forces $\sqrt{d/T}\le32\eps/(MR_x)$.

\subsection{Proof of Theorem~\ref{thm:lower-constraint}}
Let $\varphi_v(x):=\inner{\theta_v}{x}$ with $\theta_v$ as above, and consider the problem $\min\{f(z)=Mt:\ g(z)=\varphi_v(x)-Mt\le0,\ z=(x,t)\in\Q'\}$. Its optimal value is
$f^\ast=\min_{x\in R_x\Ball}\varphi_v(x)=\varphi_v^\ast=-\norm{\theta_v}R_x\in[-MR_x/2,0]$, attained at $t^\ast=\varphi_v^\ast/M\in[-R_x/2,0]$, which lies inside the interval
$[-2R_x,2R_x]$ allowed for $t$. The stated constants are verified as follows. The diameter of $\Q'=R_x\Ball\times[-2R_x,2R_x]$ is $\sqrt{(2R_x)^2+(4R_x)^2}=2\sqrt5R_x$.
The objective $f(z)=Mt$ is deterministic and $M$-Lipschitz, so $M_0=M$. $G(\cdot,\xi)$ is $\norm{(\xi,-M)}$-Lipschitz in $z$, with
$\E\norm{(\xi,-M)}^2=\E\norm\xi^2+M^2\le\frac34M^2+M^2=\frac74M^2$, so $M_g=\frac{\sqrt7}2M$; $G(z,\xi)-g(z)=\inner{\xi-\theta_v}{x}\sim\mathcal N(0,s^2\norm x^2)$ with
$\norm x\le R_x+\bar\gamma$ on $\Q'_{\bar\gamma}$, so $\sigma_g^2=s^2(R_x+\bar\gamma)^2$; and $g(0,R_x)=-MR_x$ gives Slater's condition with $\rho=MR_x$. The Lagrangian is
$Mt+y(\varphi_v(x)-Mt)=M(1-y)t+y\varphi_v(x)$; for $y<1$ its minimum over $\Q'$ is attained at $t=-2R_x$ and equals $-2MR_x(1-y)+y\varphi_v^\ast<\varphi_v^\ast$, for $y>1$ at
$t=2R_x$ with value $2MR_x(1-y)+y\varphi_v^\ast<\varphi_v^\ast$, and for $y=1$ it equals $\varphi_v^\ast=f^\ast$; hence the dual function is maximized at $y^\ast=1$, the active
multiplier equals $1$, and $\Rmult=2$ is admissible in \eqref{eq:R}. Finally $\frac{MR_x}{32}=\frac{2M_g}{\sqrt7}\cdot\frac{D}{2\sqrt5}\cdot\frac1{32}=\frac{M_gD}{32\sqrt{35}}$ and
$dM^2R_x^2/4096=dM_g^2D^2/(4096\cdot35)=dM_g^2D^2/143360$.

A paired constraint query at $(z^+,z^-)$ returns $\inner\xi{x^\pm}-Mt^\pm$; since $t^\pm$ are known to the algorithm, this
is equivalent to observing $X^\top\xi$ with $X=[x^+\ x^-]$, exactly as in the proof of \Cref{thm:lower-objective}; objective queries return the known value $Mt$ and carry no information.
Therefore any algorithm for the constrained problem with output $\hat z=(\hat x,\hat t)$ induces an algorithm for minimizing $\varphi_v$ over $R_x\Ball$ with output $\hat x$ and the same
transcript laws, and \Cref{thm:lower-objective} (whose proof only used the transcript laws and the output) yields $\max_v\E_v[\varphi_v(\hat x)-\varphi_v^\ast]\ge\frac{MR_x}{32}\min\{1,\sqrt{d/T}\}$.
Finally, $\varphi_v(\hat x)=M\hat t+g(\hat z)\le M\hat t+\pos{g(\hat z)}$, so
$\varphi_v(\hat x)-\varphi_v^\ast\le(M\hat t-f^\ast)+\pos{g(\hat z)}=(f(\hat z)-f^\ast)+\pos{g(\hat z)}$, and taking expectations gives the claim. For the consequence,
\eqref{eq:goal} gives $\E[(f(\hat z)-f^\ast)+\pos{g(\hat z)}]\le2\eps$, and $2\eps<MR_x/32$ forces $\sqrt{d/T}\le64\eps/(MR_x)$.

\subsection{Proof of Theorem~\ref{thm:lower-levels}}
\emph{The class.} Since $a\le1/2$, both feasible sets below are nonempty subsets of $[-1,1]$, and $g_i(-1)=-1-\beta_i\le-1+a\le-1/2$ for all $i$, which is the
Slater condition at $x^{\mathrm s}=-1$ with $\rho=1/2$. The samples $G_i(\cdot,\xi)$ are $1$-Lipschitz and $\norm{\xi-\beta}_\infty\le\sigma+a\le2\sigma$, so $\sigma_g=2\sigma$.

\emph{Reduction to a test.} Under $\beta^{(0)}$ the feasible set is $\{x\le a\}$, $x^\ast=a$ and $f^\ast=-a$; under $\beta^{(j)}$ it is $\{x\le-a\}$, $x^\ast=-a$ and $f^\ast=a$. Let $P_0$ and $P_j$ be the
transcript laws, and suppose the algorithm satisfies the guarantee. Under $P_0$: $f(\hat x)-f^\ast=a-\hat x$ and $\viol(\hat x)\ge\pos{\hat x-a}$; since $a-\hat x\ge a\1_{\{\hat x<0\}}-\pos{\hat x-a}$
pointwise, $a\,P_0(\hat x<0)\le\E_0[(a-\hat x)+\pos{\hat x-a}]\le2\eps$, \ie\ $P_0(\hat x<0)\le2\eps/a=1/4$. Under $P_j$: $\viol(\hat x)\ge\pos{\hat x+a}\ge a\1_{\{\hat x\ge0\}}$, so
$P_j(\hat x\ge0)\le\eps/a=1/8$. Hence, for the mixture $\bar P:=\frac1m\sum_{j=1}^mP_j$,
\[
 \TV(P_0,\bar P)\ge P_0(\hat x\ge0)-\bar P(\hat x\ge0)\ge\frac34-\frac18=\frac58,\qquad\text{so}\qquad \chi^2(\bar P\|P_0)\ge4\TV(P_0,\bar P)^2\ge\frac{25}{16}
\]
by \Cref{fact:pinsker}.

\emph{The $\chi^2$-divergence of the mixture.} A vector call at $x$ returns $x-\xi$ for a fresh $\xi$, which is equivalent to observing $\xi\in\{\pm\sigma\}^m$ itself; a paired call reveals the same $\xi$ twice.
The law of $\xi$ does not depend on the query point, so (\Cref{fact:chain}) the likelihood ratio of the transcript is the product of the likelihood ratios of the $T$ observed vectors
$\xi^{(1)},\dots,\xi^{(T)}$, and under $\beta^{(j)}$ only the $j$-th coordinates have a different law: $dP_j/dP_0=\prod_{t=1}^Tr(\xi^{(t)}_j)$ with $r:=d\nu_-/d\nu_+$, where $\nu_\pm$ is the law of a
$\pm\sigma$-valued variable with mean $\pm a$. Hence, using the independence of the coordinates under $P_0$ and $\E_{\nu_+}r=1$,
\begin{multline*}
 \chi^2(\bar P\|P_0)=\E_0\Bigl[\Bigl(\frac1m\sum_j\prod_tr(\xi^{(t)}_j)\Bigr)^2\Bigr]-1=\frac1{m^2}\sum_{j,l}\prod_{t=1}^T\E_0\bigl[r(\xi^{(t)}_j)r(\xi^{(t)}_l)\bigr]-1\\
 =\frac{m(m-1)+m(1+\eta)^T}{m^2}-1=\frac{(1+\eta)^T-1}{m},
\end{multline*}
where $\eta:=\E_{\nu_+}[r^2]-1=\chi^2(\nu_-\|\nu_+)$. With $p:=\nu_+(\{\sigma\})=(1+a/\sigma)/2$ we have $\nu_-(\{\sigma\})=1-p$ and
$\eta=(2p-1)^2\bigl(\frac1p+\frac1{1-p}\bigr)=\frac{(a/\sigma)^2}{p(1-p)}=\frac{4a^2}{\sigma^2-a^2}\le\frac{16a^2}{3\sigma^2}$ for $a\le\sigma/2$, which holds because $a=8\eps\le\sigma/2$.

\emph{Conclusion.} Combining, $(1+\eta)^T\ge1+\frac{25}{16}m\ge1+m$, so $T\ge\log(1+m)/\log(1+\eta)\ge\log(1+m)/\eta\ge\frac{3\sigma^2\log(1+m)}{16a^2}=\frac{3\sigma^2\log(1+m)}{1024\,\eps^2}$.
The second form uses $\sigma_g=2\sigma$.

\subsection{Proof of Remark~\ref{rem:lower-levels-sc}}
Only the reduction to a test changes. Let $f(x)=-x+\frac\mu2x^2$ with $\mu\in(0,1]$ and $a=16\eps\le1/2$; then $f'(x)=-1+\mu x\le0$ on $[-1,1]$, so $f$ is decreasing, $f$ is
$(1+\mu)$-Lipschitz on $[-1,1]$, and the minimizers are again $x^\ast=a$ under $\beta^{(0)}$ and $x^\ast=-a$ under $\beta^{(j)}$. Under $\beta^{(0)}$: if $\hat x<0$ then
$f(\hat x)-f^\ast\ge f(0)-f(a)=a-\frac\mu2a^2\ge\frac a2$ because $\mu a\le1$; if $\hat x>a$ then $f(\hat x)-f(a)=-(\hat x-a)+\frac\mu2(\hat x^2-a^2)\ge-\pos{\hat x-a}$; hence
$f(\hat x)-f^\ast\ge\frac a2\1_{\{\hat x<0\}}-\pos{\hat x-a}$ pointwise and, as $\viol(\hat x)\ge\pos{\hat x-a}$, $\frac a2P_0(\hat x<0)\le\E_0[(f(\hat x)-f^\ast)+\viol(\hat x)]\le2\eps$, \ie\
$P_0(\hat x<0)\le4\eps/a=1/4$. Under $\beta^{(j)}$, $\viol(\hat x)\ge\pos{\hat x+a}\ge a\1_{\{\hat x\ge0\}}$ gives $P_j(\hat x\ge0)\le\eps/a=1/16$. Hence
$\TV(P_0,\bar P)\ge\frac34-\frac1{16}\ge\frac58$ and $\TV(P_0,P_j)\ge\frac58$, exactly the bounds used above and in the proof of \Cref{thm:lower-scalar}; the condition
$a\le\sigma/2$ holds because $\eps\le\sigma/32$. The remaining steps are unchanged and give $T\ge3\sigma^2\log(1+m)/(16a^2)=3\sigma^2\log(1+m)/(4096\,\eps^2)$ for the vector oracle
and $\E_0[T]\ge25m\sigma^2/(128a^2)=25m\sigma^2/(32768\,\eps^2)$ for the scalar oracle.

\subsection{Proof of Theorem~\ref{thm:lower-scalar}}
With a scalar oracle, the $t$-th observation is $x_t-\xi^{(t)}_{i_t}$ for an index $i_t$ chosen by the algorithm, equivalently $\xi^{(t)}_{i_t}$. Let $T_j:=\#\{t\le T:i_t=j\}$. Under $P_0$ every
observation has law $\nu_+$; under $P_j$ the observations with $i_t=j$ have law $\nu_-$ and the others $\nu_+$. Conditionally on the past, the divergence of the $t$-th observation is
$\KL(\nu_+\|\nu_-)\1_{\{i_t=j\}}$, so the chain rule (\Cref{fact:chain}, which applies to a random number of calls by padding with uninformative calls) gives
$\KL(P_0\|P_j)=\E_0[T_j]\,\KL(\nu_+\|\nu_-)$. Now
\begin{multline*}
 \KL(\nu_+\|\nu_-)=p\log\frac p{1-p}+(1-p)\log\frac{1-p}p=(2p-1)\log\frac{p}{1-p}\\
 =\frac a\sigma\log\frac{1+a/\sigma}{1-a/\sigma}\le\frac a\sigma\cdot\frac{2a/\sigma}{1-a/\sigma}\le\frac{4a^2}{\sigma^2},
\end{multline*}
using $\log\frac{1+x}{1-x}\le x+\frac{x}{1-x}\le\frac{2x}{1-x}$ for $x\in[0,1)$ and $a/\sigma\le\frac12$. On the other hand, as in the previous proof, $\TV(P_0,P_j)\ge P_0(\hat x\ge0)-P_j(\hat x\ge0)\ge\frac58$,
so by Pinsker's inequality $\KL(P_0\|P_j)\ge2\TV^2\ge\frac{25}{32}$. Hence $\E_0[T_j]\ge\frac{25}{32}\cdot\frac{\sigma^2}{4a^2}=\frac{25\sigma^2}{128a^2}=\frac{25\sigma^2}{8192\,\eps^2}$ for every $j$, and
summing over $j$ gives $\E_0[T]=\sum_j\E_0[T_j]\ge\frac{25m\sigma^2}{8192\,\eps^2}$.

\section{Experimental details}\label{app:experiments}

\paragraph{Instances.} The random data ($A_i$, $q_i$) are generated once from a fixed seed and stored with the results (\texttt{data/instance\_*.npz}). The constraint rows $A_i$
($i=2,\dots,m-1$) are standard Gaussian vectors with the first coordinate set to zero and normalized to unit length, so that $g_i(x^\ast)=-0.3<0$; the noise directions $q_0,\dots,q_m$
are independent random unit vectors. With $\Q=\Ball$ and $x_0=a$, the primal radius is $\cA=\frac12\norm{a-x^\ast}^2=\frac12(0.55)^2$ for both instances and the dual radius is
$\cB=\Rmult^2\log(m+1)=4\log9$.

\paragraph{Oracle.} Each paired query draws one direction $u\sim U(\Sphere)$ (a normalized Gaussian vector) and one sample $\xi=(s_0,\dots,s_m,o_1,\dots,o_m)$ of independent Rademacher
variables, evaluates the noisy values at $c\pm\gamma u$ with the \emph{same} $\xi$, and returns the two vectors; the objective and the constraint estimates of \eqref{eq:estimators} are then
formed by the algorithm. Level estimates use one point $c+\gamma v$ with $v\sim U(\Ball)$ (a normalized Gaussian scaled by $U^{1/d}$, $U$ uniform on $[0,1]$) and a fresh $\xi$. The
implementation guards against any reuse of a mini-batch: every batch receives an identifier when it is drawn and raises an error if it is consumed twice. The oracle
counts every point evaluation, i.e.\ the number $\Norc$ of vector calls of \Cref{def:measures}; the identities $\Norc=(2b_p+3b_y)K_N+2b_p(N-1)$ of
\eqref{eq:sliding-counts} for \ZOS\ (with $b_p=b_y=8$ in all runs) and $\Norc=6bN$ for \BSMP\ are checked at the end of every run, as is the schedule condition
$a_tM_g\rs\le L$ in every phase, and the columns $\Nf$ of \Cref{tab:results} are $2b_pK_N$ and $4bN$.

\paragraph{Local constants of the quadratic instance.} The quadratic objective $\frac12\norm{x-a}^2$ and the constraint $g_m$ are convex and $1$-smooth on $\R^d$
(\Cref{ass:smooth}), the objective is $1$-strongly convex (\Cref{ass:sc}), and both are Lipschitz on every bounded set; \Cref{ass:lip} only requires the Lipschitz
property on $\Q_{\bar\gamma}$ (radius $1+\gamma$), where the objective is $(1.75+\gamma)$-Lipschitz and $g_m$ is $(1+\gamma)$-Lipschitz, and these are the constants used
below (\Cref{tab:params}). No global Lipschitz property is assumed and none is needed: by \Cref{rem:assumptions}(i), a global Lipschitz bound would be incompatible with
the strong convexity of the objective, and the theory uses only the local constants. In the smooth regime the smoothed Lagrangian is $L$-smooth on $\R^d$ with
$L=L_0+\Rmult L_g=3$ directly from \Cref{ass:smooth}, which is the property required by \Cref{lem:outer}.

\paragraph{Smoothness constant of the cusp instance.} For the cusp instance the value $L=2+\sqrt d/(4\gamma)$ of \Cref{tab:params} uses the additive structure of the data
rather than the generic bound $H\le\sqrt dM_R/\gamma\approx240$: the linear part $-x_1$ and the affine constraints are unchanged by smoothing and contribute nothing
to $H$; the nonsmooth part $\frac14\norm{x_{2:d}}$ is $\frac14$-Lipschitz, so by \Cref{lem:smoothing}(c) its smoothed version is $\sqrt d/(4\gamma)$-smooth; and
$g_m$ is $1$-smooth, contributing $\Rmult L_g=2$. Hence $H\le2+\sqrt d/(4\gamma)$, which is all that \Cref{thm:sliding} requires ($L\ge H$). The slope noise is linear
in $x$ and does not affect the smoothness of the expected functions.

\paragraph{Parameters.} \Cref{tab:params} lists the parameters. For \ZOS\ the constants entering the theoretical stabilizer are the upper bounds implied by the instance:
$M_g=1+\gamma+0.02$, $M_0=1.75+\gamma+0.02$ (smooth) or $M_0=\sqrt{1+1/16}+0.02$ (cusp), $\sigma_h^2=(2\gamma(1+\gamma)+0.02(1+\gamma)+0.02)^2$ (a deterministic bound on
$\norm{\hat h-h}_\infty$), $\nuj=4d$, $\rs=1$, $D=2$; the stabilizer of \Cref{cor:sliding-lambda} is then $\lambda=\Sb\sqrt{A_N/2}/\dA=\SA\sqrt{A_N/(2b)}/\dA$ (with $b_p=b_y=b=8$) and $\dA^2=3\cA+\cB$. The inner
schedule is exactly the one of \eqref{eq:schedule}, $S_t=\max\{1,\lceil M_g\rs t/L\rceil\}$ with $\rs=1$ and the same $M_g=1+\gamma+0.02$ that enters the stabilizer
(a bound that covers the affine constraints, $g_m$ and the slope noise); it gives $K_N=\SmoothInner$ (smooth) and $K_N=\CuspInner$ (cusp), and the condition
$a_tM_g\rs\le L$ of \Cref{lem:schedule}(i) is asserted by the code in every phase. For \BSMP\ the constant $L_\alpha$ is
computed from \eqref{eq:Lalpha} with $H=L$ and the mean-square Lipschitz constant $M_g=1+\gamma$ of the smoothed constraints.

\begin{table}[htbp]
\centering\footnotesize
\caption{Parameters of the experiments.}\label{tab:params}
\setlength{\tabcolsep}{4pt}
\begin{tabular}{@{}lll@{}}
\toprule
& smooth instance & cusp instance\\
\midrule
$\gamma$, shift of $g_m$ & $0.04$, $\gamma^2\cdot10/52$ & $0.08$, $\gamma^2\cdot10/52$\\
$L$ (upper bound of $H$) & $L_0+\Rmult L_g=3$ & $2+\sqrt d/(4\gamma)\approx17.31$\\
$\Rmult$ & $2$ & $2$\\
\ZOS: $N$, $b_p=b_y$, $\rs$, $S_t$ & $60$, $8$, $1$, $\lceil1.06\,t/3\rceil$ & $60$, $8$, $1$, $\lceil1.10\,t/17.31\rceil$\\
\ZOS: $\lambda$ (practical / theoretical) & $0.1$ / \SmoothStabilizer & $0.1$ / \CuspStabilizer\\
\ZOS: $K_N$, $\Norc$, $\Nf$ & \SmoothInner, \SmoothCalls, \SmoothNf & \CuspInner, \CuspCalls, \CuspNf\\
\BSMP: $\alpha$, $\eta$ & $\sqrt{20}$, $0.45/L_\alpha$ & $\sqrt{20}$, $0.45/L_\alpha$\\
\BSMP: iterations for $b=1$ / $b=8$ & \SmoothBsmpOneIters\ / \SmoothBsmpEightIters & \CuspBsmpOneIters\ / \CuspBsmpEightIters\\
restarts (\Cref{sec:exp-restart}) & \begin{tabular}[t]{@{}l@{}}$7$ stages of $10$ rounds,\\ $b_k=2^{k-1}$, $\rs_k=2^{(k-1)/2}$;\\ plain runs: $N=60$, $b\in\{8,16,\EqBatch\}$\end{tabular} & ---\\
seeds & $10000,\dots,10029$ & $10000,\dots,10029$\\
\bottomrule
\end{tabular}
\end{table}

\paragraph{Level-testing diagnostic.} In \Cref{fig:levels} the null hypothesis is $\beta^{(0)}$ and the alternative is the uniform mixture of $\beta^{(1)},\dots,\beta^{(m)}$
of the proof of \Cref{thm:lower-levels}, with $\sigma=0.2$, $a=0.04$ and $T$ vector calls, each returning the $m$ independent $\pm\sigma$-valued coordinates. The test
statistic is the exact log-likelihood ratio of the mixture against the null, $\log\frac1m\sum_{j=1}^m r_j^{T}$ with $r_j^{T}=\exp\bigl((2c_j-T)\log\frac{\sigma-a}{\sigma+a}\bigr)$
and $c_j$ the number of $+\sigma$ observations in coordinate $j$, thresholded at $0$; the reported error is the average of the two error probabilities (the Bayes
risk with equal priors), estimated from $4\,000$ independent trials under each hypothesis, so that its standard error is at most $0.01$.

\paragraph{Metrics.} After every round (resp.\ every Mirror--Prox iteration) the current output $\Xr$ is evaluated exactly (the exact $f$ and $g$ are used only for reporting, never
by the algorithms): the objective gap $f(\Xr)-f^\ast$, the violation $\viol(\Xr)=\max_i\pos{g_i(\Xr)}$ and the joint error $J(\Xr)=\max\{0,f(\Xr)-f^\ast,\viol(\Xr)\}$. The means in
\Cref{tab:results} are over the \NumSeeds\ seeds; the confidence intervals are percentile bootstrap intervals ($10\,000$ resamples) for the mean of the final $J$. The bands in the figures are
pointwise interquartile ranges over seeds.

\paragraph{Reproducibility.} \begin{sloppypar}The archive contains the code in the directory \texttt{code/}: the file \texttt{zo\_constrained.py}
(problem, oracle and the two methods), \texttt{run\_all.py} (all experiments) and \texttt{make\_figures.py} (figures and table); the raw records in
\texttt{data/}: \texttt{trajectories.csv}, \texttt{restarts.csv}, \texttt{level\_testing.csv} and \texttt{noise\_diagnostic.csv}; the configurations
\texttt{config.json}, \texttt{restart\_config.json} and \texttt{environment.json}; and the figures.\end{sloppypar}
Running \texttt{python3 run\_all.py} followed by \texttt{python3 make\_figures.py} in the \texttt{code} directory regenerates all data, figures, the table \texttt{sections/results\_table.tex}
and the numerical macros \texttt{sections/results\_macros.tex} used in the text. The total running time is about five minutes on a single CPU core (Python~3 with NumPy, SciPy,
pandas and Matplotlib).

\end{document}